\documentclass[english]{amsart}
\usepackage[T1]{fontenc}
\usepackage[latin9]{inputenc}
\usepackage{babel}
\usepackage{amsbsy}
\usepackage{amsthm}
\usepackage{amssymb}
\usepackage[all]{xy}
\usepackage[unicode=true,pdfusetitle,
 bookmarks=true,bookmarksnumbered=false,bookmarksopen=false,
 breaklinks=false,pdfborder={0 0 1},backref=false,colorlinks=false]
 {hyperref}

\makeatletter
\numberwithin{equation}{section}
\numberwithin{figure}{section}
\theoremstyle{plain}
\newtheorem{thm}{\protect\theoremname}
  \theoremstyle{plain}
  \newtheorem{lem}[thm]{\protect\lemmaname}
  \theoremstyle{plain}
  \newtheorem{prop}[thm]{\protect\propositionname}
  \theoremstyle{remark}
  \newtheorem{rem}[thm]{\protect\remarkname}
  \theoremstyle{definition}
  \newtheorem{defn}[thm]{\protect\definitionname}
  \theoremstyle{plain}
  \newtheorem{cor}[thm]{\protect\corollaryname}

\newcommand{\xyR}[1]{\xydef@\xymatrixrowsep@{#1}}
\newcommand{\xyC}[1]{\xydef@\xymatrixcolsep@{#1}}

\newcommand{\Red}{\mathsf{Red}}
\newcommand{\Ext}{\mathsf{Ext}}
\newcommand{\Top}{\mathsf{Top}}
\newcommand{\CCAT}{\mathsf{CCat(0)}}
\newcommand{\CAT}{\mathsf{Cat(0)}}
\newcommand{\HV}{\mathsf{HV}}
\renewcommand{\C}{\mathsf{C}}
\newcommand{\A}{\mathsf{A}}
\newcommand{\F}{\mathsf{F}}

\newcommand{\Bun}{\mathsf{Bun}}

\newcommand{\Fil}{\mathsf{Fil}}
\newcommand{\Vect}{\mathsf{Vect}}
\newcommand{\Rep}{\mathsf{Rep}}
\newcommand{\Norm}{\mathsf{Norm}}

\newcommand{\Gr}{\mathrm{Gr}}
\newcommand{\rank}{\mathrm{rank}}
\newcommand{\Hom}{\mathrm{Hom}}
\newcommand{\End}{\mathrm{End}}
\newcommand{\Gal}{\mathrm{Gal}}
\newcommand{\Aut}{\mathrm{Aut}}
\newcommand{\Sym}{\mathrm{Sym}}
\newcommand{\loc}{\mathrm{loc}}
\newcommand{\Spec}{\mathrm{Spec}}

\newcommand{\Jump}{\mathrm{Jump}}
\newcommand{\Ji}{\mathrm{Ji}}
\newcommand{\Atom}{\mathrm{Atom}}
\newcommand{\height}{\mathrm{height}}
\newcommand{\length}{\mathrm{length}}

\newcommand{\im}{\mathrm{im}}
\newcommand{\coim}{\mathrm{coim}}
\newcommand{\coker}{\mathrm{coker}}
\newcommand{\Sub}{\mathrm{Sub}}

\newcommand{\sk}{\mathrm{sk}}

\newcommand{\eqd}{\stackrel{\mathrm{def}}{=}}

\makeatother

  \providecommand{\corollaryname}{Corollary}
  \providecommand{\definitionname}{Definition}
  \providecommand{\lemmaname}{Lemma}
  \providecommand{\propositionname}{Proposition}
  \providecommand{\remarkname}{Remark}
\providecommand{\theoremname}{Theorem}

\begin{document}

\title{On Harder-Narasimhan filtrations and their compatibility with tensor
products}

\author{Christophe Cornut}

\maketitle
\tableofcontents{}

\section{Introduction}

The Harder-Narasimhan formalism, as set up for instance by André in~\cite{An09},
requires a category $\C$ with kernels and cokernels, along with rank
and degree functions 
\[
\rank:\sk\,\C\rightarrow\mathbb{N}\quad\mbox{and}\quad\deg:\sk\,\C\rightarrow\mathbb{R}
\]
on the skeleton of $\C$, subject to various axioms. It then functorially
equips every object $X$ of $\C$ with a Harder-Narasimhan filtration
$\mathcal{F}_{HN}(X)$ by strict subobjects. This categorical formalism
is very nice and useful, but it does not say much about what $\mathcal{F}_{HN}(X)$
really is. The build-in characterization of this filtration only involves
the restriction of the rank and degree functions to the poset $\Sub(X)$
of strict subobjects of $X$, and a first aim of this paper is to
pin down the relevant formalism.

André's axioms on $(\C,\rank)$ imply that the poset $\Sub(X)$ is
a modular lattice of finite length \cite{Gr11}. Thus, starting in
section~\ref{sec:HN4ModLatt} with an arbitrary modular lattice $\mathcal{X}$
of finite length, we introduce a space $\mathbf{F}(\mathcal{X})$
of $\mathbb{R}$-filtrations on $\mathcal{X}$. This looks first like
a combinatorial object with building-like features: apartments, facets
and chambers. The choice of a rank function on $\mathcal{X}$ equips
$\mathbf{F}(\mathcal{X})$ with a distance $d$, and we show that
$(\mathbf{F}(\mathcal{X}),d)$ is a complete, CAT(0)-metric space,
whose underlying topology and geodesic segments do not depend upon
the chosen rank function. The choice of a degree function on $\mathcal{X}$
amounts to the choice of a concave function on $\mathbf{F}(\mathcal{X})$,
and we show that a closely related continuous function has a unique
minimum $\mathcal{F}\in\mathbf{F}(\mathcal{X})$: this is the Harder-Narasimhan
filtration for the triple $(\mathcal{X},\rank,\deg)$. The fact that
modular lattices provide a natural framework for the Harder-Narasimhan
theory was discovered independently by Hugues Randriambololona, see~\cite[\S 1]{Ra16}.

In section~\ref{sec:HN4Cat}, we derive our own Harder-Narasimhan
formalism for categories from this Harder-Narasimhan formalism for
modular lattices. It differs slightly from André's: we are perhaps
a bit more flexible in our axioms on $\C$, but a bit more demanding
in our axioms for the rank and degree functions. 

When the category $\C$ is also equipped with a $k$-linear tensor
product, is the Harder-Narasimhan filtration compatible with this
auxiliary structure? Many cases have already been considered and solved
by ad-hoc methods, often building on Totaro's pioneering work \cite{To96},
which itself relied on tools borrowed from Mumford's Geometric Invariant
Theory \cite{MuFoKi94}. Trying to understand and generalize the latest
installment of this trend \cite{LeWE16}, we came up with some sort
of axiomatized version of its overall strategy in which the GIT tools
are replaced by tools from convex metric geometry. This is exposed
in section~\ref{sec:HN4QT}, which gives a numerical criterion for
the compatibility of HN-filtrations with various tensor product constructions.
Our approach simultaneously yields some results towards exactness
of HN-filtrations, which classically required separate proofs, often
using Haboush's theorem \cite{Ha75}.

In the last section, we verify our criterion in three cases (which
could be combined as explained in section~\ref{subsec:FiberProductGood}):
filtered vector spaces (\ref{subsec:FilteredVectorSpaces}), normed
vector spaces (\ref{subsec:NormedVectorSpaces}) and normed $\varphi$-modules
(\ref{subsec:NormedPhiModules}). The first case has been known for
some times, see for instance~\cite{DaOrRa10}. The second case seems
to be new, and it applies for instance to the isogeny category of
sthukas with one paw, as considered in Scholze's Berkeley course or
in \cite{BaMoSc16}. The third case is a mild generalization of \cite[3.1.1]{LeWE16}.

\thanks{I would like to thank Brandon Levin and Carl Wang-Erickson for their
explanations on \cite{LeWE16}. In my previous attempts to deal with
the second and third of the above cases, a key missing step was part
$(3)$ of the proof of proposition~\ref{prop:CompBusemProj}. The
related statement appears to be lemma~$3.6.6$ of \cite{LeWE16}.
Finally, I would like to end this introduction with a question: in
all three cases, the semi-stable objects of slope $0$ form a full
subcategory $\C^{0}$ of $\C$ which is a neutral $k$-linear tannakian
category. }

\thanks{\begin{center}
What are the corresponding Tannaka groups?
\par\end{center}}

\section{The Harder-Narasimhan formalism for modular lattices\label{sec:HN4ModLatt}}

\subsection{Basic notions}

We refer to \cite{Gr11} for all things pertaining to basic lattice
theory.

\subsubsection{~\label{subsec:definitions4lattices}}

A \emph{lattice} is a partially ordered set (a poset) $(X,\leq)$
such that every pair of elements $(x,y)\in X$ has a \emph{meet} $x\vee y:=\sup\{x,y\}$
and a \emph{join} $x\wedge y:=\inf\{x,y\}$. It is \emph{bounded}
if it has both a minimal element $0_{X}$ and a maximal element $1_{X}$.
It is \emph{distributive} (resp.~\emph{modular}) if and only if $x\wedge(y\vee z)=(x\wedge y)\vee(x\wedge z)$
for all $x,y,z\in X$ (resp.~for all $x,y,z\in X$ with $z\leq x$).
A \emph{subposet} of $X$ is a subset equipped with the induced partial
order, a \emph{sublattice} is a subposet stable under the meet and
join operators of $X$, and a \emph{chain} in $X$ is a totally ordered
subposet. A chain of \emph{length} $\ell$ is a finite chain of order
$\ell+1$ and the \emph{length} of $X$ is the supremum of the length
of its finite chains (with values in $\mathbb{N}\cup\{\infty\}$).
An element $x$ of a bounded lattice $X$ is \emph{join-irreducible}
if $x\neq0_{X}$ and $x=y\vee z$ implies $x=y$ or $x=z$; it is
an \emph{atom} if $x\neq0_{X}$ and $y\leq x$ implies $y=0_{X}$
or $y=x$. We denote by $\Atom(X)\subset\Ji(X)$ the set of atoms
and join-irreducible elements of $X$. A \emph{complement} of $x$
is an element $y$ of $X$ such that $x\wedge y=0_{X}$ and $x\vee y=1_{X}$.
A \emph{complemented} lattice is a bounded lattice in which every
element as a complement. A \emph{boolean} lattice is a complemented
distributive lattice. A non-decreasing map between bounded lattices
is a \emph{lattice map} (resp.~a $\{0,1\}$-\emph{map}) if it is
compatible with the meet and join operators (resp.~with the minimal
and maximal elements). For $x\leq y$ in $X$, we denote by $[x,y]$
or $\frac{y}{x}$ the subposet $\{z\in X:x\leq z\leq y\}$ of $X$. 

\subsubsection{~}

Let $X$ be a fixed bounded modular lattice of finite length $r$.
An \emph{apartment} in $X$ is a maximal distributive sublattice $S$
of $X$. Any such $S$ is finite \cite[Theorem 4.28]{Ro08}, of length
$r$ \cite[Corollary 2]{Le96}, with also $\left|\Ji(S)\right|=r$
by \cite[Corollary 108]{Gr11}. The formula $c_{i}=c_{i-1}\wedge s_{i}$
yields a bijection between the set of all maximal chains $C=\{c_{0}<\cdots<c_{r}\}$
in $S$ and the set of all bijections $i\mapsto s_{i}$ from $\{1,\cdots,r\}$
to $\Ji(S)$ whose inverse $s_{i}\mapsto i$ is non-decreasing. The
theorem of Birkhoff and Dedekind \cite[Theorem 363]{Gr11} asserts
that any two chains in $X$ are contained in some apartment.

\subsubsection{~}

A \emph{degree function} on $X$ is a function $\deg:X\rightarrow\mathbb{R}$
such that 
\[
\deg(0_{X})=0\quad\mbox{and}\quad\deg(x\vee y)+\deg(x\wedge y)\geq\deg(x)+\deg(y)
\]
for every $x$, $y$ in $X$. We say that it is \emph{exact} if also
$-\deg$ is a degree function, i.e.
\[
\deg(x\vee y)+\deg(x\wedge y)=\deg(x)+\deg(y)
\]
for every $x$, $y$ in $X$. A \emph{rank function} on $X$ is an
increasing exact degree function. Thus a rank function on $X$ is
a function $\rank:X\rightarrow\mathbb{R}_{+}$ such that $\rank(0_{X})=0$,
\[
\rank(x\vee y)+\rank(x\wedge y)=\rank(x)+\rank(y)
\]
for every $x$, $y$ in $X$ and $\rank(x)<\rank(y)$ if $x<y$. The
\emph{standard rank function} is given by $\rank(x)=\height(x)$,
the length of any maximal chain in $[0_{X},x]$. 

\subsubsection{~}

For a chain $C=\{c_{0}<\cdots<c_{s}\}$ in $X$, set 
\[
\Gr_{C}^{\bullet}\eqd\prod_{i=1}^{s}\Gr_{C}^{i}\quad\mbox{with}\quad\Gr_{C}^{i}\eqd[c_{i-1},c_{i}].
\]
For the direct product partial order on $\Gr_{C}^{\bullet}$ defined
by 
\[
(x_{1},\cdots,x_{s})\leq(y_{1},\cdots,y_{s})\stackrel{\mathrm{def}}{\iff}\forall i\in\{1,\cdots,s\}:\quad x_{i}\leq y_{i},
\]
this is again plainly a modular lattice of finite length $\leq r$,
which is even a finite boolean lattice of length $r$ if $C$ is maximal.
We denote by $\varphi_{C}:X\rightarrow\Gr_{C}^{\bullet}$ the non-decreasing
$\{0,1\}$-map which sends $x\in X$ to $\varphi_{C}(x)=((x\wedge c_{i})\vee c_{i-1})_{i=1}^{s}$.
The restriction of $\varphi_{C}$ to any apartment containing $C$
is a lattice $\{0,1\}$-map. 

\subsubsection{~\label{subsec:GrC}}

For $\deg:X\rightarrow\mathbb{R}$, $\rank:X\rightarrow\mathbb{R}_{+}$
and $C$ as above, we still denote by 
\[
\deg:\Gr_{C}^{\bullet}\rightarrow\mathbb{R}\quad\mbox{and}\quad\rank:\Gr_{C}^{\bullet}\rightarrow\mathbb{R}_{+}
\]
the induced degree and rank functions on $\Gr_{C}^{\bullet}$ defined
by 
\begin{eqnarray*}
\deg\left((z_{i})_{i=1}^{s}\right) & \eqd & \sum_{i=1}^{s}\deg(z_{i})-\deg(c_{i-1})\\
\rank\left((z_{i})_{i=1}^{s}\right) & \eqd & \sum_{i=1}^{s}\rank(z_{i})-\rank(c_{i-1})
\end{eqnarray*}
for $z_{i}\in\Gr_{C}^{i}=[c_{i-1},c_{i}]$. If $C$ is a $\{0,1\}$-chain,
i.e.~$c_{0}=0_{X}$ and $c_{s}=1_{X}$, then 
\[
\deg(x)\leq\deg\left(\varphi_{C}(x)\right)\quad\mbox{and}\quad\rank(x)=\rank\left(\varphi_{C}(x)\right)
\]
for every $x$ in $X$. Indeed since $x\wedge c_{i-1}=(x\wedge c_{i})\wedge c_{i-1}$
for all $i\in\{1,\cdots,s\}$, 
\[
\begin{array}{ccc}
\underbrace{{\textstyle \sum_{i=1}^{s}}\deg(x\wedge c_{i})-\deg(x\wedge c_{i-1})} & \leq & \underbrace{{\textstyle \sum_{i=1}^{s}}\deg((x\wedge c_{i})\vee c_{i-1})-\deg(c_{i-1})}\\
=\deg(x) &  & =\deg\left(\varphi_{C}(x)\right)
\end{array}
\]
with equality if and only if for every $i\in\{1,\cdots,s\}$, 
\[
\deg(x\wedge c_{i})+\deg(c_{i-1})=\deg\left((x\wedge c_{i})\vee c_{i-1}\right)+\deg(x\wedge c_{i-1}).
\]
This occurs for instance if $\deg$ is exact on the sublattice of
$X$ spanned by $C\cup\{x\}$. 

\subsubsection{~}

In particular, a rank function on $X$ is uniquely determined by its
values on any maximal chain $C=\{c_{0}<\cdots<c_{r}\}$ of $X$. Indeed
for every $x\in X$,
\[
\rank(x)=\sum_{\begin{subarray}{c}
i\in\{1,\cdots,r\}\\
(x\wedge c_{i})\vee c_{i-1}=c_{i}
\end{subarray}}\rank(c_{i})-\rank(c_{i-1}).
\]
If $C$ is a maximal chain in $X$, the degree map on $\Gr_{C}^{\bullet}$
is exact and 
\[
\deg(x)\leq\sum_{\begin{subarray}{c}
i\in\{1,\cdots,r\}\\
(x\wedge c_{i})\vee c_{i-1}=c_{i}
\end{subarray}}\deg(c_{i})-\deg(c_{i-1})
\]
for every $x\in X$. In particular, $\deg:X\rightarrow\mathbb{R}$
is bounded above.

\subsubsection{~\label{subsec:RankImpliesModular}}

We started with a modular lattice of finite length, but the definition
of a rank function makes sense for an arbitrary bounded lattice $X$.
We claim that:
\begin{lem}
\label{lem:CaractModularbyExistRank}A bounded lattice $X$ is modular
of finite length if and only if it has an integer-valued rank function
$\rank:X\rightarrow\mathbb{N}$. 
\end{lem}
\begin{proof}
One direction is obvious: if $X$ is modular of finite length, then
the standard rank function $\height:X\rightarrow\mathbb{N}$ works.
Suppose conversely that $\rank:X\rightarrow\mathbb{N}$ is a rank
function. Then $\rank(1_{X})$ bounds the length of any chain in $X$,
thus $X$ already has finite length. For modularity, we have to show
that for every $a,b,c\in X$ with $a\leq c$, $(a\vee b)\wedge c=a\vee(b\wedge c).$
Replacing $c$ by $c'=(a\vee b)\wedge c$, we may assume that $a\leq c\leq a\vee b$,
thus $a\vee b=c\vee b$. Replacing $a$ by $a'=a\vee(b\wedge c)$,
we may assume that also $a\wedge b=c\wedge b$. In other words, we
have to show that if $a\leq c$, $a\wedge b=c\wedge b$ and $a\vee b=c\vee b$,
then $a=c$. But these assumptions imply that 
\[
\begin{array}{rcccl}
\rank(a)+\rank(b) & = & \rank(a\vee b)+\rank(a\wedge b)\\
 & = & \rank(c\vee b)+\rank(c\wedge b) & = & \rank(c)+\rank(b)
\end{array}
\]
thus $\rank(a)=\rank(c)$ and indeed $a=c$ since otherwise $\rank(a)<\rank(c)$.
\end{proof}

\subsubsection{~}

An apartment $S$ of $X$ is \emph{special} if $S$ is a (finite)
boolean lattice.
\begin{lem}
\label{lem:XcomplImpliesExistsSpecial}Suppose that $X$ is complemented.
Then any chain $C$ in $X$ is contained in a special apartment $S$
of $X$. 
\end{lem}
\begin{proof}
Indeed, we may assume that $C=\{c_{0}<\cdots<c_{r}\}$ is maximal.
Since $X$ is complemented, an induction on the length $r$ of $X$
shows that there is another maximal chain $C'=\{c_{0}^{\prime}<\cdots<c_{r}^{\prime}\}$
in $X$ such that $c_{r-i}^{\prime}$ is a complement of $c_{i}$
for all $i\in\{0,\cdots,r\}$ \textendash{} we then say that $C'$
is \emph{opposed} to $C$. We claim that any apartment $S$ of $X$
containing $C$ and $C'$ is special. Indeed, if $\Ji(S)=\{x_{1},\cdots,x_{r}\}$
with $c_{i}=c_{i-1}\vee x_{i}$ for all $i\in\{1,\cdots,r\}$, then
$c_{i}^{\prime}=c_{i-1}^{\prime}\vee x_{r+1-i}$ for all $i\in\{1,\cdots,r\}$,
thus $x_{i}\mapsto i$ and $x_{i}\mapsto r+1-i$ are non-decreasing
bijections $\Ji(S)\rightarrow\{1,\cdots,r\}$, so $\Ji(S)$ is unordered
and $S$ is indeed boolean by \cite[II.1.2]{Gr11}.
\end{proof}

\subsection{$\mathbb{R}$-filtrations}

Let again $X$ be a modular lattice of finite length $r$. 

\subsubsection{~}

An $\mathbb{R}$\emph{-filtration} on $X$ is a function $f:\mathbb{R}\rightarrow X$
which is non-increasing, exhaustive, separated and left continuous:
$f(\gamma_{1})\geq f(\gamma_{2})$ for $\gamma_{1}\leq\gamma_{2}$,
$f(\gamma)=1_{X}$ for $\gamma\ll0$, $f(\gamma)=0_{X}$ for $\gamma\gg0$
and $f(\gamma)=\inf\{f(\eta):\eta<\gamma\}$ for $\gamma\in\mathbb{R}$.
We set
\[
f_{+}(\gamma)\eqd\sup\left\{ f(\eta):\eta>\gamma\right\} \leq f(\gamma)\quad\mbox{and}\quad\Gr_{f}^{\gamma}\eqd\left[f_{+}(\gamma),f(\gamma)\right].
\]
Note that $f_{+}(\gamma)$ is indeed well-defined since $f(\mathbb{R})$
is a (finite) chain in $X$. Equivalently, an $\mathbb{R}$-filtration
on $X$ is a pair $(C,\underline{\gamma})$ where $C=\{c_{0}<\cdots<c_{s}\}$
is a $\{0,1\}$-chain in $X$ (i.e.~with $c_{0}=0_{X}$, $c_{s}=1_{X}$)
and $\underline{\gamma}=(\gamma_{1}>\cdots>\gamma_{s})$ is a decreasing
sequence in $\mathbb{R}$. The correspondence $f\leftrightarrow(C,\underline{\gamma})$
is characterized by 
\[
C=F(f)\eqd f(\mathbb{R})\quad\mbox{and}\quad\underline{\gamma}=\Jump(f)\eqd\left\{ \gamma\in\mathbb{R}:\Gr_{f}^{\gamma}\neq0\right\} ,
\]
where $\Gr_{f}^{\gamma}\neq0$ means $f_{+}(\gamma)\neq f(\gamma)$.
Thus for every $\gamma\in\mathbb{R}$, 
\[
f(\gamma)=\begin{cases}
c_{0}=0_{X} & \mbox{for }\gamma>\gamma_{1},\\
c_{i} & \mbox{for }\gamma_{i+1}<\gamma\leq\gamma_{i},\,i\in\{1,\cdots,s-1\},\\
c_{s}=1_{X} & \mbox{for }\gamma\leq\gamma_{s}.
\end{cases}
\]

\subsubsection{~}

We denote by $\mathbf{F}(X)$ the set of all $\mathbb{R}$-filtrations
on $X$. We say that $f,f'\in\mathbf{F}(X)$ are in the same \emph{facet}
if $F(f)=F(f')$. We write $F^{-1}(C)\eqd\{f:f(\mathbb{R})=C\}$ for
the facet defined by a chain $C$; thus $\Jump$ yields a bijection
from $F^{-1}(C)$ to
\[
\mathbb{R}_{>}^{s}\eqd\left\{ (\gamma_{1},\cdots,\gamma_{s})\in\mathbb{R}^{s}:\gamma_{1}>\cdots>\gamma_{s}\right\} ,\quad s=\mbox{length}(C).
\]
The \emph{closed facet} of $C$ is $\mathbf{F}(C)=\{f:f(\mathbb{R})\subset C\}$,
isomorphic to 
\[
\mathbb{R}_{\geq}^{s}\eqd\left\{ (\gamma_{1},\cdots,\gamma_{s})\in\mathbb{R}^{s}:\gamma_{1}\geq\cdots\geq\gamma_{s}\right\} .
\]
We call \emph{chambers} (open or closed) the facets of the maximal
$C$'s. 

\subsubsection{~}

For any $\mu\in\mathbb{R}$, we denote by $X(\mu)$ the unique element
of $F^{-1}(\{0_{X},1_{X}\})$ such that $\Jump(X(\mu))=\mu$, i.e.~$X(\mu)(\gamma)=1_{X}$
for $\gamma\leq\mu$ and $X(\mu)(\gamma)=0_{X}$ for $\gamma>\mu$.
We define a scalar multiplication and a symmetric addition map 
\[
\mathbb{R}_{+}\times\mathbf{F}(X)\rightarrow\mathbf{F}(X)\quad\mbox{and}\quad\mathbf{F}(X)\times\mathbf{F}(X)\rightarrow\mathbf{F}(X)
\]
by the following formulas: for $\lambda>0$, $f,g\in\mathbf{F}(X)$
and $\gamma\in\mathbb{R}$, 
\[
(\lambda\cdot f)(\gamma)\eqd f(\lambda^{-1}\gamma)\quad\mbox{and}\quad(f+g)(\gamma)\eqd\bigvee\left\{ f(\gamma_{1})\wedge g(\gamma_{2}):\gamma_{1}+\gamma_{2}=\gamma\right\} ,
\]
while for $\lambda=0$, we set $0\cdot f=X(0)$. Note that the formula
defining $f+g$ indeed makes sense since $f(\mathbb{R})$ and $g(\mathbb{R})$
are finite. One checks easily that 
\begin{eqnarray*}
X(\mu_{1})+X(\mu_{2}) & = & X(\mu_{1}+\mu_{2})\\
\lambda\cdot X(\mu) & = & X(\lambda\mu)\\
\lambda\cdot(f+g) & = & \lambda\cdot f+\lambda\cdot g\\
\mbox{and}\quad(f+X(\mu))(\gamma) & = & f(\gamma-\mu)
\end{eqnarray*}
for every $\mu_{1},\mu_{2},\mu\in\mathbb{R}$, $\lambda\in\mathbb{R}_{+}$,
$f,g\in\mathbf{F}(X)$ and $\gamma\in\mathbb{R}$. 

\subsubsection{Examples\label{subsec:Examples}}

If $(X,\leq)=\{c_{0}<\cdots<c_{r}\}$ is a finite chain, the formula
\[
f_{i}^{\sharp}\eqd\sup\left\{ \gamma\in\mathbb{R}:c_{i}\leq f(\gamma)\right\} 
\]
yields a bijection $f\mapsto f^{\sharp}$ between $\left(\mathbf{F}(X),\cdot,+\right)$
and the closed cone 
\[
\mathbb{R}_{\geq}^{r}\eqd\left\{ (\gamma_{1},\cdots,\gamma_{r})\in\mathbb{R}^{r}:\gamma_{1}\geq\cdots\geq\gamma_{r}\right\} .
\]
Note that the left continuity of $f$ implies that for all $i\in\{1,\cdots,r\}$,
also 
\[
f_{i}^{\sharp}=\max\left\{ \gamma\in\mathbb{R}:c_{i}\leq f(\gamma)\right\} .
\]
More generally if $(X,\leq)$ is a finite distributive lattice (and
thus also a bounded modular lattice of finite length, so that $\mathbf{F}(X)$
is well-defined), the formula
\begin{eqnarray*}
f^{\sharp}(x) & \eqd & \sup\left\{ \gamma:x\leq f(\gamma)\right\} \\
 & = & \max\left\{ \gamma:x\leq f(\gamma)\right\} 
\end{eqnarray*}
yields a bijection $f\mapsto f^{\sharp}$ between $\left(\mathbf{F}(X),\cdot,+\right)$
and the cone of all non-increasing functions $f^{\sharp}:\Ji(X)\rightarrow\mathbb{R}$,
where $\Ji(X)\subset X$ is the subposet of all join-irreducible elements
of $X$ (compare with \cite[II.1.3]{Gr11}). The inverse bijection
is given by 
\[
f(\gamma)=\bigvee\left\{ x\in\Ji(X):f^{\sharp}(x)\geq\gamma\right\} .
\]
In particular if $(X,\leq)$ is a finite boolean lattice, $\Ji(X)=\Atom(X)$
is the unordered finite set of atoms in $X$ and the above formula
yields a bijection between $\left(\mathbf{F}(X),\cdot,+\right)$ and
the finite dimensional $\mathbb{R}$-vector space of all functions
$\Atom(X)\rightarrow\mathbb{R}$. 

\subsubsection{Functoriality}

Let $\varphi:X\rightarrow Y$ be a non-decreasing $\{0,1\}$-map between
bounded modular lattices of finite length. Then $\varphi$ induces
a map 
\[
\mathbf{F}(\varphi):\mathbf{F}(X)\rightarrow\mathbf{F}(Y),\qquad f\mapsto\varphi\circ f.
\]
Plainly for every $\mu\in\mathbb{R}$, $\lambda\in\mathbb{R}_{+}$
and $f\in\mathbf{F}(X)$, 
\[
\mathbf{F}(\varphi)(X(\mu))=Y(\mu)\quad\mbox{and}\quad\mathbf{F}(\varphi)(\lambda\cdot f)=\lambda\cdot\mathbf{F}(\varphi)(f).
\]
If moreover $\varphi$ is a lattice map, i.e.~if it is compatible
with the meet and join operations on $X$ and $Y$, then $\mathbf{F}(\varphi)$
is also compatible with the addition maps:
\[
\mathbf{F}(\varphi)(f+g)=\mathbf{F}(\varphi)(f)+\mathbf{F}(\varphi)(g).
\]

\subsubsection{~}

An \emph{apartment} of $\mathbf{F}(X)$ is a subset of the form $\mathbf{F}(S)$,
where $S$ is an apartment of $X$, i.e.~a maximal distributive sublattice
of $X$. Thus $\left(\mathbf{F}(S),\cdot,+\right)$ is isomorphic
to the cone of non-increasing maps $\Ji(S)\rightarrow\mathbb{R}$
by~\ref{subsec:Examples}. The map $S\mapsto\mathbf{F}(S)$ is a
bijection between apartments in $X$ and $\mathbf{F}(X)$. The apartment
$\mathbf{F}(S)$ is a finite disjoint union of facets of $\mathbf{F}(X)$,
indexed by the $\{0,1\}$-chains in $S$. By~\cite[Theorem 363]{Gr11},
for any $f,g\in\mathbf{F}(X)$, there is an apartment $\mathbf{F}(S)$
which contains $f$ and $g$.

We also write $0\in\mathbf{F}(X)$ for the trivial $\mathbb{R}$-filtration
$X(0)$ on $X$. It is a neutral element for the addition map on $\mathbf{F}(X)$.
More precisely, for every $f,g\in\mathbf{F}(X)$, $f+g=f$ if and
only if $g=0$: this follows from a straightforward computation in
any apartment $\mathbf{F}(S)$ containing $f$ and $g$. We say that
two $\mathbb{R}$-filtrations $f$ and $f'$ are \emph{opposed} if
$f+f'=0$. If $f$ belongs to a special apartment $\mathbf{F}(S)$
(i.e.~one with $S$ boolean), then there is a unique $f'\in\mathbf{F}(S)$
which is opposed to $f$. Thus if $X$ is complemented, any $f\in\mathbf{F}(X)$
has at least one opposed $\mathbb{R}$-filtration by lemma~\ref{lem:XcomplImpliesExistsSpecial}.

\subsubsection{~}

For any chain $C$ in $X$, the $\{0,1\}$-map $\varphi_{C}:X\rightarrow\Gr_{C}^{\bullet}$
induces a map 
\[
r_{C}:\mathbf{F}(X)\rightarrow\mathbf{F}(\Gr_{C}^{\bullet}),\qquad r_{C}\eqd\mathbf{F}(\varphi_{C}).
\]
If $S$ is an apartment of $X$ which contains $C$, the restriction
of $\varphi_{C}$ to $S$ is a lattice $\{0,1\}$-map and the restriction
of $r_{C}$ to $\mathbf{F}(S)$ is compatible with the addition maps. 

If $C$ is maximal, then $\Gr_{C}^{\bullet}=\prod_{i=1}^{r}\Gr_{C}^{i}$
is a finite boolean lattice and 
\[
\Atom(\Gr_{C}^{\bullet})=\{c_{1}^{\ast},\cdots,c_{r}^{\ast}\}
\]
with $c_{i}^{\ast}$ corresponding to the atom $c_{i}$ of $\Gr_{C}^{i}=\{c_{i-1},c_{i}\}$.
For $C\subset S\subset X$ as above, the $\{0,1\}$-lattice map $\varphi_{C}\vert_{S}:S\rightarrow\Gr_{C}^{\bullet}$
then induces a bijection 
\[
\Ji(\varphi_{C}\vert_{S}):\Atom(\Gr_{C}^{\bullet})=\Ji(\Gr_{C}^{\bullet})\stackrel{\simeq}{\longrightarrow}\Ji(S)
\]
mapping $c_{i}^{\ast}$ to $s_{i}$, characterized by $c_{i}=c_{i-1}\vee s_{i}$
for all $i\in\{1,\cdots,r\}$. Then
\[
r_{C}:\mathbf{F}(S)\rightarrow\mathbf{F}(\Gr_{C}^{\bullet})
\]
maps a non-increasing function $f^{\sharp}:\Ji(S)\rightarrow\mathbb{R}$
to the corresponding function $f^{\sharp}\circ\Ji(\varphi_{C}\vert_{S}):\Atom(\Gr_{C}^{\bullet})\rightarrow\mathbb{R}$.
In particular, it is injective.

\subsubsection{~\label{subsec:typemapdef}}

The rank function $\height:X\rightarrow\{0,\cdots,r\}$ is a non-decreasing
$\{0,1\}$-map, it thus induces a function $\mathbf{t}:=\mathbf{F}(\height)$
which we call the \emph{type map}: 
\[
\mathbf{t}:\mathbf{F}(X)\rightarrow\mathbf{F}(\{0,\cdots,r\})=\mathbb{R}_{\geq}^{r}.
\]
The restriction of $\mathbf{t}$ to an apartment $\mathbf{F}(S)$
maps $f^{\sharp}:\Ji(S)\rightarrow\mathbb{R}$ to 
\[
\mathbf{t}(f^{\sharp})=(\gamma_{1}\geq\cdots\geq\gamma_{r})\quad\mbox{with}\quad\left|\left\{ i:\gamma_{i}=\gamma\right\} \right|=\left|\left\{ x:f^{\sharp}(x)=\gamma\right\} \right|.
\]
The restriction of $\mathbf{t}$ to a closed chamber $\mathbf{F}(C)$
is a cone isomorphism (i.e.~a bijection compatible with the scalar
operations and addition maps). 

\subsubsection{~\label{subsec:embeddingX2F(X)}}

The set $\mathbf{F}(X)$ is itself a lattice, with meet and join given
by 
\[
(f\wedge g)(\gamma)\eqd f(\gamma)\wedge g(\gamma)\quad\mbox{and}\quad(f\vee g)(\gamma)\eqd f(\gamma)\vee g(\gamma)
\]
for every $f,g\in\mathbf{F}(X)$ and $\gamma\in\mathbb{R}$. Moreover,
there is a natural lattice embedding 
\[
X\hookrightarrow\mathbf{F}(X),\qquad x\mapsto x(-)\quad\mbox{with}\quad x(\gamma)\eqd\begin{cases}
1_{X} & \mbox{if }\gamma\leq0,\\
x & \mbox{if }0<\gamma\leq1,\\
0_{X} & \mbox{if }1<\gamma.
\end{cases}
\]
It maps $0_{X}$ to $X(0)$ and $1_{X}$ to $X(1)$. Viewing $X$
as a sublattice of $\mathbf{F}(X)$, the addition map on $\mathbf{F}(X)$
sends $(x,y)\in X^{2}$ to the $\mathbb{R}$-filtration $x+y\in\mathbf{F}(X)$
given by 
\[
(x+y)(\gamma)=\begin{cases}
1_{X} & \mbox{if }\gamma\leq0,\\
x\vee y & \mbox{if }0<\gamma\leq1,\\
x\wedge y & \mbox{if }1<\gamma\leq2,\\
0_{X} & \mbox{if }2<\gamma.
\end{cases}
\]
For every $f\in\mathbf{F}(X)$ with $\Jump(f)\subset\{\gamma_{1},\cdots,\gamma_{N}\}$
where $\gamma_{1}<\cdots<\gamma_{N}$, we have
\[
f=\gamma_{1}\cdot1_{X}+\sum_{i=2}^{N}(\gamma_{i}-\gamma_{i-1})\cdot f(\gamma_{i}).
\]
Since the addition map on $\mathbf{F}(X)$ is not associative, the
above sum is a priori not well-defined. However, all of its summands
belong to the closed facet $\mathbf{F}(C)$ of $f$ (with $C=f(\mathbb{R})$),
and the formula is easily checked inside this commutative monoid.

\subsubsection{~\label{subsec:degfctXandF(X)}}

A \emph{degree function} on $\mathbf{F}(X)$ is a function $\left\langle \star,-\right\rangle :\mathbf{F}(X)\rightarrow\mathbb{R}$
such that for $\lambda\in\mathbb{R}_{+}$ and $f,g\in\mathbf{F}(X)$,
$(1)$ $\left\langle \star,\lambda f\right\rangle =\lambda\left\langle \star,f\right\rangle $,
$(2)$ $\left\langle \star,f+g\right\rangle \geq\left\langle \star,f\right\rangle +\left\langle \star,g\right\rangle $
and $(3)$ $\left\langle \star,f+g\right\rangle =\left\langle \star,f\right\rangle +\left\langle \star,g\right\rangle $
if $f(\mathbb{R})\cup g(\mathbb{R})$ is a chain. We claim that:
\begin{lem}
Restriction from $\mathbf{F}(X)$ to its sublattice $X\hookrightarrow\mathbf{F}(X)$
yields a bijection between degree functions on $\mathbf{F}(X)$ and
degree functions on $X$. 
\end{lem}
\begin{proof}
If $\left\langle \star,-\right\rangle :\mathbf{F}(X)\rightarrow\mathbb{R}$
is a degree function on $\mathbf{F}(X)$, then for any $x,y\in X$,
\[
\left\langle \star,x\vee y\right\rangle +\left\langle \star,x\wedge y\right\rangle \stackrel{(a)}{=}\left\langle \star,x\vee y+x\wedge y\right\rangle \stackrel{(b)}{=}\left\langle \star,x+y\right\rangle \stackrel{(c)}{\geq}\left\langle \star,x\right\rangle +\left\langle \star,y\right\rangle 
\]
using $(3)$ for $(a)$, the equality $x+y=x\vee y+x\wedge y$ in
$\mathbf{F}(X)$ for $(b)$, and $(2)$ for $(c)$. Since also $\left\langle \star,0_{X}\right\rangle =0$
by $(1)$, it follows that $x\mapsto\left\langle \star,x\right\rangle $
is a degree function on $X$: our map is thus well-defined. It is
injective since any function $\deg:X\rightarrow\mathbb{R}$ with $\deg(0_{X})=0$
has a unique extension to a function $\left\langle \star,-\right\rangle :\mathbf{F}(X)\rightarrow\mathbb{R}$
satisfying $(1)$ and $(3)$, which is given by the following formula:
for any $f\in\mathbf{F}(X)$, 
\[
\left\langle \star,f\right\rangle =\sum_{\gamma\in\mathbb{R}}\gamma\cdot\deg\left(\Gr_{f}^{\gamma}\right)\quad\mbox{with}\quad\Gr_{f}^{\gamma}=\left[f_{+}(\gamma),f(\gamma)\right]
\]
where $\deg\left([x,y]\right)=\deg(y)-\deg(x)$ for $x\leq y$ in
$X$. Equivalently, 
\[
\left\langle \star,f\right\rangle =\gamma_{1}\cdot\deg\left(1_{X}\right)+\sum_{i=2}^{N}(\gamma_{i}-\gamma_{i-1})\cdot\deg\left(f(\gamma_{i})\right)
\]
whenever $\Jump(f)\subset\{\gamma_{1},\cdots,\gamma_{N}\}$ with $\gamma_{1}<\cdots<\gamma_{N}$. 

It remains to establish that if we start with a degree function on
$X$, this unique extension also satisfies our concavity axiom $(2)$.
Note that the last formula for $\left\langle \star,f\right\rangle $
then shows that for any $\{0,1\}$-chain $C$ in $X$, 
\[
\left\langle \star,f\right\rangle \leq\left\langle \star,r_{C}(f)\right\rangle 
\]
with equality if the initial degree function is exact on the sublattice
of $X$ spanned by $C\cup f(\mathbb{R})$. Here $r_{C}(f)=\varphi_{C}\circ f$
in $\mathbf{F}(\Gr_{C}^{\bullet})$ and $\left\langle \star,-\right\rangle :\mathbf{F}(\Gr_{C}^{\bullet})\rightarrow\mathbb{R}$
is the extension, as defined above, of the degree function $\deg:\Gr_{C}^{\bullet}\rightarrow\mathbb{R}$
induced by our initial degree function on $X$. Now for $f,g\in\mathbf{F}(X)$,
pick an apartment $S$ of $X$ containing $f(\mathbb{R})\cup g(\mathbb{R})$
and a maximal chain $C\subset S$ containing $(f+g)(\mathbb{R})$.
Then 
\[
\left\langle \star,f+g\right\rangle =\left\langle \star,r_{C}(f+g)\right\rangle \quad\mbox{with}\quad r_{C}(f+g)=r_{C}(f)+r_{C}(g)
\]
since $\deg$ is exact on the chain $C\supset(f+g)(\mathbb{R})$ and
$f,g\in\mathbf{F}(S)$ with $C\subset S$. Since also $\left\langle \star,f\right\rangle \leq\left\langle \star,r_{C}(f)\right\rangle $
and $\left\langle \star,g\right\rangle \leq\left\langle \star,r_{C}(g)\right\rangle $,
it is sufficient to establish that 
\[
\left\langle \star,r_{C}(f)+r_{C}(g)\right\rangle \geq\left\langle \star,r_{C}(f)\right\rangle +\left\langle \star,r_{C}(g)\right\rangle .
\]
We may thus assume that $X$ is a finite Boolean lattice equipped
with an exact degree function, in which case the function $\left\langle \star,-\right\rangle :\mathbf{F}(X)\rightarrow\mathbb{R}$
is actually linear: 
\[
\left\langle \star,f\right\rangle =\sum_{a\in\Atom(X)}f^{\sharp}(a)\deg(a).
\]
This finishes the proof of the lemma.
\end{proof}

\subsection{Metrics}

Let now $\rank:X\rightarrow\mathbb{R}_{+}$ be a rank function on
$X$.

\subsubsection{~}

We equip $\mathbf{F}(X)$ with a symmetric pairing 
\[
\left\langle -,-\right\rangle :\mathbf{F}(X)\times\mathbf{F}(X)\rightarrow\mathbb{R},\quad\left\langle f_{1},f_{2}\right\rangle \eqd\sum_{\gamma_{1},\gamma_{2}\in\mathbb{R}}\gamma_{1}\gamma_{2}\cdot\rank\left(\Gr_{f_{1},f_{2}}^{\gamma_{1},\gamma_{2}}\right)
\]
with notations as above, where for any $f_{1},f_{2}\in\mathbf{F}(X)$
and $\gamma_{1},\gamma_{2}\in\mathbb{R}$, 
\[
\Gr_{f_{1},f_{2}}^{\gamma_{1},\gamma_{2}}\eqd\frac{f_{1}(\gamma_{1})\wedge f_{2}(\gamma_{2})}{\left(f_{1,+}(\gamma_{1})\wedge f_{2}(\gamma_{2})\right)\vee\left(f_{1}(\gamma_{1})\wedge f_{2,+}(\gamma_{2})\right)}.
\]
Note that with these definitions and for any $\lambda\in\mathbb{R}_{+}$,
\[
\left\langle \lambda f_{1},f_{2}\right\rangle =\lambda\left\langle f_{1},f_{2}\right\rangle =\left\langle f_{1},\lambda f_{2}\right\rangle .
\]
If $\Jump(f_{\nu})\subset\{\gamma_{1}^{\nu},\cdots,\gamma_{s_{\nu}}^{\nu}\}$
with $\gamma_{1}^{\nu}<\cdots<\gamma_{s_{\nu}}^{\nu}$ and $x_{j}^{\nu}=f_{\nu}(\gamma_{j}^{\nu})$
for $\nu\in\{1,2\}$, 
\begin{eqnarray*}
\left\langle f_{1},f_{2}\right\rangle  & = & \sum_{i=1}^{s_{1}}\sum_{j=1}^{s_{2}}\gamma_{i}^{1}\gamma_{j}^{2}\cdot\rank\left(\frac{x_{i}^{1}\wedge x_{j}^{2}}{\left(x_{i+1}^{1}\wedge x_{j}^{2}\right)\vee\left(x_{i}^{1}\wedge x_{j+1}^{2}\right)}\right)
\end{eqnarray*}
with the convention that $x_{s_{\nu}+1}^{\nu}=0_{X}$. Thus with $r_{i,j}=\rank\left(x_{i}^{1}\wedge x_{j}^{2}\right)$,
also 
\begin{eqnarray*}
\left\langle f_{1},f_{2}\right\rangle  & = & \sum_{i=1}^{s_{1}}\sum_{j=1}^{s_{2}}\gamma_{i}^{1}\gamma_{j}^{2}\left(r_{i,j}-r_{i+1,j}-r_{i,j+1}+r_{i+1,j+1}\right)\\
 & = & \sum_{i=2}^{s_{1}}\sum_{j=2}^{s_{2}}\left(\gamma_{i}^{1}-\gamma_{i-1}^{1}\right)\left(\gamma_{j}^{2}-\gamma_{j-1}^{2}\right)r_{i,j}+\gamma_{1}^{1}\gamma_{1}^{2}r_{1,1}\\
 &  & +\sum_{i=2}^{s_{1}}(\gamma_{i}^{1}-\gamma_{i-1}^{1})\gamma_{1}^{2}r_{i,1}+\sum_{j=2}^{s_{2}}\gamma_{1}^{1}(\gamma_{j}^{2}-\gamma_{j-1}^{2})r_{1,j}
\end{eqnarray*}

\subsubsection{~\label{subsec:Funct4Scal}}

Let $\varphi:X\rightarrow Y$ be a non-decreasing $\{0,1\}$-map between
bounded modular lattices of finite length such that the rank function
on $X$ is induced by a rank function on $Y$. Then for the pairing
on $\mathbf{F}(Y)$,
\begin{eqnarray*}
\left\langle \varphi\circ f_{1},\varphi\circ f_{2}\right\rangle  & = & \sum_{i=2}^{s_{1}}\sum_{j=2}^{s_{2}}\left(\gamma_{i}^{1}-\gamma_{i-1}^{1}\right)\left(\gamma_{j}^{2}-\gamma_{j-1}^{2}\right)r'_{i,j}\\
 &  & +\gamma_{1}^{1}\gamma_{1}^{2}r'_{1,1}+\sum_{i=2}^{s_{1}}(\gamma_{i}^{1}-\gamma_{i-1}^{1})\gamma_{1}^{2}r'_{i,1}+\sum_{j=2}^{s_{2}}\gamma_{1}^{1}(\gamma_{j}^{2}-\gamma_{j-1}^{2})r'_{1,j}
\end{eqnarray*}
where $r'_{i,j}=\rank\left(\varphi(x_{i}^{1})\wedge\varphi(x_{j}^{2})\right)$.
Since $\varphi(x_{i}^{1}\wedge x_{j}^{2})\leq\varphi(x_{i}^{1})\wedge\varphi(x_{j}^{2})$
with equality when $i$ or $j$ equals $1$, $r'_{i,j}\geq r_{i,j}$
with equality when $i$ or $j$ equals $1$, thus 
\[
\left\langle f_{1},f_{2}\right\rangle \leq\left\langle \varphi\circ f_{1},\varphi\circ f_{2}\right\rangle .
\]
If $\varphi(z_{1}\wedge z_{2})=\varphi(z_{1})\wedge\varphi(z_{2})$
for all $z_{\nu}\in f_{\nu}(\mathbb{R})$, for instance if the restriction
of $\varphi$ to the sublattice of $X$ generated by $f_{1}(\mathbb{R})\cup f_{2}(\mathbb{R})$
is a lattice map, then 
\[
\left\langle f_{1},f_{2}\right\rangle =\left\langle \varphi\circ f_{1},\varphi\circ f_{2}\right\rangle .
\]
In particular, this holds whenever $f_{1}(\mathbb{R})\cup f_{2}(\mathbb{R})$
is a chain.

\subsubsection{~\label{subsec:ScalAndrC}}

For a $\{0,1\}$-chain $C=\{c_{0}<\cdots<c_{s}\}$ in $X$, we equip
$\Gr_{C}^{\bullet}=\prod_{i=1}^{s}\Gr_{C}^{i}$ with the induced rank
function as explained in \ref{subsec:GrC}. Applying the previous
discussion to the rank-compatible $\{0,1\}$-map $\varphi_{C}:X\rightarrow\Gr_{C}^{\bullet}$
(which restricts to a lattice map on any apartement $S$ of $X$ containing
$C$), we obtain the following lemma.
\begin{lem}
\label{lem:PairingIncreaseAlongRetract}Let $C$ be a $\{0,1\}$-chain.
Then for every $f_{1},f_{2}\in\mathbf{F}(X)$, 
\[
\left\langle f_{1},f_{2}\right\rangle \leq\left\langle r_{C}(f_{1}),r_{C}(f_{2})\right\rangle 
\]
with equality if $C$, $f_{1}$ and $f_{2}$ are contained in a common
apartement of $\mathbf{F}(X)$.
\end{lem}

\subsubsection{~\label{subsec:FormulaPairingApp}}

This yields another formula for the pairing on $\mathbf{F}(X)$: for
every apartment $\mathbf{F}(S)$, there is a function $\delta_{S}:\Ji(S)\rightarrow\mathbb{R}_{>0}$
such that for every $f_{1},f_{2}\in\mathbf{F}(S)$, 
\[
\left\langle f_{1},f_{2}\right\rangle =\sum_{x\in\Ji(S)}f_{1}^{\sharp}(x)f_{2}^{\sharp}(x)\cdot\delta_{S}(x)
\]
where $f^{\sharp}:\Ji(S)\rightarrow\mathbb{R}$ is the non-increasing
map attached to $f\in\mathbf{F}(S)$. Indeed, pick a maximal chain
$C\subset S$. Then $\left\langle f_{1},f_{2}\right\rangle =\left\langle r_{C}(f_{1}),r_{C}(f_{2})\right\rangle $.
But the pairing on $\mathbf{F}(\Gr_{C}^{\bullet})$ is easily computed,
and it is a positive definite symmetric bilinear form: for $g_{1}$
and $g_{2}$ in $\mathbf{F}(\Gr_{C}^{\bullet})$ corresponding to
functions $g_{1}^{\sharp}$ and $g_{2}^{\sharp}:\Atom(\Gr_{C}^{\bullet})\rightarrow\mathbb{R}$,
\[
\left\langle g_{1},g_{2}\right\rangle =\sum_{a\in\Atom(\Gr_{C}^{\bullet})}g_{1}^{\sharp}(a)g_{2}^{\sharp}(a)\,\rank(a).
\]
For $g_{\nu}=r_{C}(f_{\nu})=\varphi_{C}\circ f_{\nu}$, we have seen
that $g_{\nu}^{\sharp}=f_{\nu}^{\sharp}\circ\Ji(\varphi_{C}\vert S)$,
where $\Ji(\varphi_{C}\vert_{S})$ is the bijection $\Atom(\Gr_{C}^{\bullet})\simeq\Ji(S)$.
This proves our claim, with $\delta_{S}(x)=\rank(a)$ if $\Ji(\varphi_{C}\vert_{S})(a)=x$.
If $C=\{c_{0}<\cdots<c_{r}\}$, then $\Ji(S)=\{x_{1},\cdots,x_{r}\}$
with $c_{i}=c_{i-1}\wedge x_{i}$ and $\delta_{S}(x_{i})=\rank(c_{i})-\rank(c_{i-1})$
for all $i\in\{1,\cdots,r\}$. 

\subsubsection{~\label{subsec:concavpairingonfil}}

The next lemma says that our pairing is concave. 
\begin{lem}
\label{lem:PairingIsConcave}For every $f$, $g$ and $h$ in $\mathbf{F}(X)$,
we have 
\[
\left\langle f,g+h\right\rangle \geq\left\langle f,g\right\rangle +\left\langle f,h\right\rangle 
\]
with equality if $f$, $g$ and $h$ belong to a common apartement
of $\mathbf{F}(X)$.
\end{lem}
\begin{proof}
Indeed, choose $S$, $C$ and $S'$ as follows: $S$ is an apartment
of $X$ containing $g(\mathbb{R})$ and $h(\mathbb{R})$, $C$ is
a maximal chain in $S$ containing $(g+h)(\mathbb{R})\subset S$,
and $S'$ is an apartment of $X$ containing $f(\mathbb{R})$ and
$C$. If $f$, $g$ and $h$ belong to a common apartement, we may
and do also require that $S=S'$. In all cases, 
\[
\left\langle f,g+h\right\rangle \stackrel{(1)}{=}\left\langle r_{C}(f),r_{C}(g+h)\right\rangle \quad\mbox{and}\quad r_{C}(g+h)\stackrel{(2)}{=}r_{C}(g)+r_{C}(h)
\]
since respectively $(1)$ $C\subset S'$ and $f$, $g+h$ belong to
$\mathbf{F}(S')$ and $(2)$ $C\subset S$ and $g$, $h$ belong to
$\mathbf{F}(S)$. Since $C$ is maximal, $\Gr_{C}^{\bullet}$ is boolean,
$\mathbf{F}(\Gr_{C}^{\bullet})$ is an $\mathbb{R}$-vector space
and the pairing on $\mathbf{F}(\Gr_{C}^{\bullet})$ is a positive
definite symmetric bilinear form, thus 
\[
\left\langle r_{C}(f),r_{C}(g)+r_{C}(h)\right\rangle \stackrel{(3)}{=}\left\langle r_{C}(f),r_{C}(g)\right\rangle +\left\langle r_{C}(f),r_{C}(h)\right\rangle .
\]
Our claim now follows from $(1)$, $(2)$ and $(3)$ since also by~lemma~\ref{lem:PairingIncreaseAlongRetract},
\[
\left\langle r_{C}(f),r_{C}(g)\right\rangle \geq\left\langle f,g\right\rangle \quad\mbox{and}\quad\left\langle r_{C}(f),r_{C}(h)\right\rangle \geq\left\langle f,g\right\rangle 
\]
with equality if, along with $g$, $h$ and $C$, also $f$ belongs
to $\mathbf{F}(S)$. 
\end{proof}

\subsubsection{~}

It follows that for every $f\in\mathbf{F}(X)$, the function $g\mapsto\left\langle f,g\right\rangle $
is a degree function on $\mathbf{F}(X)$. The corresponding degree
function on $X$ maps $x\in X$ to 
\[
\deg_{f}(x)\eqd\sum_{\gamma\in\mathbb{R}}\gamma\,\rank\left(\Gr_{f\wedge x}^{\gamma}\right)\quad\mbox{with}\quad\Gr_{f\wedge x}^{\gamma}\eqd\left[f_{+}(\gamma)\wedge x,f(\gamma)\wedge x\right].
\]
For $f=X(1)$, we retrieve the rank: $\deg_{X(1)}(x)=\rank(x)$. For
$f\in\mathbf{F}(X)$, 
\[
\deg(f)\eqd\left\langle X(1),f\right\rangle =\sum_{\gamma\in\mathbb{R}}\gamma\,\rank\left(\Gr_{f}^{\gamma}\right)
\]
is the natural degree function on $\mathbf{F}(X)$ and the formula
\[
\deg(f+g)\geq\deg(f)+\deg(g)
\]
follows either from~\ref{subsec:concavpairingonfil} or from~\ref{subsec:degfctXandF(X)}.

\subsubsection{~}

For $f,g\in\mathbf{F}(X)$, $\left\langle f,f\right\rangle \geq0$
and $2\left\langle f,g\right\rangle \leq\left\langle f,f\right\rangle +\left\langle g,g\right\rangle $:
this follows from the formula in \ref{subsec:FormulaPairingApp}.
We may thus define 
\[
\left\Vert f\right\Vert \eqd\sqrt{\left\langle f,f\right\rangle }\quad\mbox{and}\quad d(f,g)\eqd\sqrt{\left\Vert f\right\Vert ^{2}+\left\Vert g\right\Vert ^{2}-2\left\langle f,g\right\rangle }.
\]
For every $\{0,1\}$-chain $C$ in $X$, $\left\Vert r_{C}(f)\right\Vert =\left\Vert f\right\Vert $
and 
\[
d\left(r_{C}(f),r_{C}(g)\right)\leq d(f,g)
\]
with equality if there is an apartment $\mathbf{F}(S)$ with $C\subset S$
and $f,g\in\mathbf{F}(S)$. Also,
\[
\begin{array}{c}
\left\Vert f\right\Vert =d(0_{X},f),\quad\left\Vert tf\right\Vert =t\left\Vert f\right\Vert ,\quad d(tf,tg)=td(f,g)\\
\mbox{and}\quad\left\Vert f+g\right\Vert ^{2}=\left\Vert f\right\Vert ^{2}+\left\Vert g\right\Vert ^{2}+2\left\langle f,g\right\rangle 
\end{array}
\]
for every $f,g\in\mathbf{F}(X)$ and $t\in\mathbb{R}_{+}$. The first
three formulas are obvious, and the last one follows from the additivity
of the symmetric pairing on any apartment. If $f$ and $f'$ are opposed
in $\mathbf{F}(X)$, then $\left\Vert f\right\Vert =\left\Vert f'\right\Vert =\frac{1}{2}d(f,f')$
and $\left\langle f,f'\right\rangle =-\left\Vert f\right\Vert ^{2}$.

\subsubsection{~}

We refer to~\cite{BrHa99} for all things pertaining to geodesic
and CAT(0)-spaces. 
\begin{prop}
The function $d:\mathbf{F}(X)\times\mathbf{F}(X)\rightarrow\mathbb{R}_{\geq0}$
is a CAT(0)-distance. 
\end{prop}
\begin{proof}
If $X$ is a finite boolean lattice, then $d$ is the euclidean distance
attached to the positive definite symmetric bilinear form (in short:
scalar product) $\left\langle -,-\right\rangle $ on the $\mathbb{R}$-vector
space $\mathbf{F}(X)$, which proves the proposition. For the general
case:
\[
\forall f,g\in\mathbf{F}(X):\qquad d(f,g)=0\,\Longrightarrow\,f=g.
\]
Indeed, choose an apartment with $f,g\in\mathbf{F}(S)$, a maximal
chain $C\subset S$. Then $d(r_{C}(f),r_{C}(g))=0$, thus $r_{C}(f)=r_{C}(g)$
since $d$ is a (euclidean) distance on $\mathbf{F}(\Gr_{C}^{\bullet})$
and $f=g$ since the restriction $r_{C}\vert_{\mathbf{F}(S)}:\mathbf{F}(S)\rightarrow\mathbf{F}(\Gr_{C}^{\bullet})$
is injective. 
\[
\forall f,g,h\in\mathbf{F}(X):\qquad d(f,h)\leq d(f,g)+d(g,h).
\]
Indeed, choose an apartment with $f,h\in\mathbf{F}(S)$, a maximal
chain $C\subset S$. Then 
\begin{eqnarray*}
d(f,h) & = & d(r_{C}(f),r_{C}(h))\\
 & \leq & d(r_{C}(f),r_{C}(g))+d(r_{C}(g),r_{C}(h))\\
 & \leq & d(f,g)+d(g,h).
\end{eqnarray*}
Thus $d$ is a distance, and a similar argument shows that $(\mathbf{F}(X),d)$
is a geodesic metric space. More precisely, for every $g,h\in\mathbf{F}(X)$
and $t\in[0,1]$, if 
\[
g_{t}=(1-t)g+th
\]
is the sum of $(1-t)\cdot g$ and $t\cdot h$ in $\mathbf{F}(X)$,
then $d\left(g,g_{t}\right)=t\cdot d(g,h)$, thus $t\mapsto g_{t}$
is a geodesic segment from $g$ to $h$ in $\mathbf{F}(X)$. Note
also that 
\[
\left\Vert g_{t}\right\Vert ^{2}=(1-t)^{2}\left\Vert g\right\Vert ^{2}+t^{2}\left\Vert h\right\Vert ^{2}+2t(1-t)\left\langle g,h\right\rangle .
\]
For the CAT(0)-inequality, we finally have to show that for every
$f\in\mathbf{F}(X)$, 
\[
d(f,g_{t})^{2}+t(1-t)d(g,h)^{2}\leq(1-t)d(f,g)^{2}+td(f,h)^{2}.
\]
Given the previous formula for $\left\Vert g_{t}\right\Vert ^{2}$,
this amounts to 
\[
\left\langle f,g_{t}\right\rangle \geq(1-t)\left\langle f,g\right\rangle +t\left\langle f,h\right\rangle 
\]
which is the already established concavity of $\left\langle f,-\right\rangle $.
\end{proof}

\subsubsection{~}

Let $d_{\mathrm{Std}}:\mathbf{F}(X)\times\mathbf{F}(X)\rightarrow\mathbb{R}$
be the distance attached to the standard rank function $x\mapsto\height(x)$
on $X$. By~\ref{subsec:GrC}, there are constants $A>a>0$ such
that $a\leq\rank(y)-\rank(x)\leq A$ for every $x<y$ in $X$. It
then follows from~\ref{subsec:FormulaPairingApp} that there are
constants $B>b>0$ such that $b\,d_{\mathrm{Std}}(f,g)\leq d(f,g)\leq B\,d_{\mathrm{Std}}(f,g)$
for every $f,g\in\mathbf{F}(X)$. The topology induced by $d$ on
$\mathbf{F}(X)$ thus does not depend upon the chosen rank function.
We call it the \emph{canonical topology}. Being complete for the induced
distance, apartments and closed chambers are closed in $\mathbf{F}(X)$. 
\begin{prop}
The metric space $(\mathbf{F}(X),d)$ is complete.
\end{prop}
\begin{proof}
We may assume that $d=d_{\mathrm{Std}}$. The type function $\mathbf{t}:\mathbf{F}(X)\rightarrow\mathbb{R}_{\geq}^{r}$
defined in~\ref{subsec:typemapdef}~is then non-expanding for the
standard euclidean distance $d$ on $\mathbb{R}_{\geq}^{r}$: this
follows from~\ref{subsec:Funct4Scal} applied to $\height:X\rightarrow\{0,\cdots,r\}$.
In fact, for any maximal chain $C$ in any apartment $S$ of $X$,
the composition of the isometric embeddings
\[
\xyC{3pc}\xymatrix{\mathbf{F}(C)\ar@{^{(}->}[r] & \mathbf{F}(S)\ar@{^{(}->}[r]\sp(0.4){r_{c}} & \mathbf{F}(\Gr_{C}^{\bullet})\simeq\mathbb{R}^{r}}
\]
with the non-expanding type map $\mathbf{t}:\mathbf{F}(\Gr_{C}^{\bullet})\rightarrow\mathbb{R}_{\geq}^{r}$
is an isometry $\mathbf{F}(C)\simeq\mathbb{R}_{\geq}^{r}$. It follows
that for every pair of types $(t_{1},t_{2})$ in $\mathbb{R}_{\geq}^{r}$,
\[
\left\{ d(f_{1},f_{2})\left|\begin{array}{l}
f_{\nu}\in\mathbf{F}(S)\\
\mathbf{t}(f_{\nu})=t_{\nu}
\end{array}\right.\right\} \subset\left\{ d(f_{1},f_{2})\left|\begin{array}{l}
f_{\nu}\in\mathbf{F}(\Gr_{C}^{\bullet})\\
\mathbf{t}(f_{\nu})=t_{\nu}
\end{array}\right.\right\} 
\]
and both sets are finite with the same minimum $d(t_{1},t_{2})$,
thus also 
\[
\left\{ d(f_{1},f_{2})\left|\begin{array}{l}
f_{\nu}\in\mathbf{F}(X)\\
\mathbf{t}(f_{\nu})=t_{\nu}
\end{array}\right.\right\} \subset\left\{ d(f_{1},f_{2})\left|\begin{array}{l}
f_{\nu}\in\mathbf{F}(\Gr_{C}^{\bullet})\\
\mathbf{t}(f_{\nu})=t_{\nu}
\end{array}\right.\right\} 
\]
is finite with minimum $d(t_{1},t_{2})$. In particular, there is
a constant $\epsilon(t_{1},t_{2})>0$ such that for every $f_{1},f_{2}\in\mathbf{F}(X)$
with $\mathbf{t}(f_{1})=t_{1}$ and $\mathbf{t}(f_{2})=t_{2}$, 
\[
d(f_{1},f_{2})=d(t_{1},t_{2})\quad\mbox{or}\quad d(f_{1},f_{2})\geq d(t_{1},t_{2})+\epsilon(t_{1},t_{2}).
\]
Let now $(f_{n})_{n\geq0}$ be a Cauchy sequence in $\mathbf{F}(X)$.
Then $t_{n}=\mathbf{t}(f_{n})$ is a Cauchy sequence in $\mathbb{R}_{\geq}^{r}$,
so it converges to a type $t\in\mathbb{R}_{\geq}^{r}$. Fix $N\in\mathbb{N}$
such that 
\[
d(f_{n},f_{m})<{\textstyle \frac{1}{3}}\epsilon(t,t)\quad\mbox{and}\quad d(t_{n},t)\leq{\textstyle \frac{1}{3}}\epsilon(t,t)
\]
for all $n,m\geq N$. For each $n\geq N$, pick a maximal chain $C_{n}$
containing $f_{n}(\mathbb{R})$ and let $g_{n}$ be the unique element
of the closed chamber $\mathbf{F}(C_{n})$ such that $\mathbf{t}(g_{n})=t$.
Then $d(f_{n},g_{n})=d(t_{n},t)$ since $f_{n}$ and $g_{n}$ belong
to $\mathbf{F}(C_{n})$. Note that if $g'_{n}$ is any other element
of $\mathbf{F}(X)$ such that $\mathbf{t}(g'_{n})=t$ and $d(f_{n},g'_{n})=d(t_{n},t)$,
then 
\[
d(g_{n},g'_{n})\leq d(g_{n},f_{n})+d(f_{n},g'_{n})=2d(t_{n},t)\leq{\textstyle \frac{2}{3}}\epsilon(t,t)<\epsilon(t,t),
\]
therefore $g_{n}=g'_{n}$. Similarly for every $n,m\geq N$, 
\[
d(g_{n},g_{m})\leq d(g_{n},f_{n})+d(f_{n},f_{m})+d(f_{m},g_{m})<\epsilon(t,t)
\]
thus $g_{n}=g_{m}$. Call $g\in\mathbf{F}(X)$ this common value.
Then 
\[
d(f_{n},g)=d(f_{n},g_{n})=d(t_{n},t)
\]
thus $f_{n}\rightarrow g$ in $\mathbf{F}(X)$ since $t_{n}\rightarrow t$
in $\mathbb{R}_{\geq}^{r}$.
\end{proof}

\subsubsection{~}

Let $\deg:X\rightarrow\mathbb{R}$ be a degree function on $X$ and
let $\left\langle \star,-\right\rangle :\mathbf{F}(X)\rightarrow\mathbb{R}$
be its unique extension to a degree function on $\mathbf{F}(X)$,
as explained in~\ref{subsec:degfctXandF(X)}. 
\begin{prop}
\label{prop:Contdeg}Suppose that $\lim f_{n}=f$ in $\mathbf{F}(X)$.
Then 
\[
\lim\sup\left\langle \star,f_{n}\right\rangle \leq\left\langle \star,f\right\rangle .
\]
If moreover $\deg(X)$ is bounded, then $\left\langle \star,-\right\rangle :\mathbf{F}(X)\rightarrow\mathbb{R}$
is continuous.
\end{prop}
\begin{rem}
The first assertion says that $\left\langle \star,-\right\rangle $
is always upper semi-continuous.
\end{rem}
\begin{proof}
Let $C=f(\mathbb{R})=\{c_{0}<\cdots<c_{s}\}$. In the previous proof,
we have seen that for every sufficiently large $n$, any maximal chain
$C_{n}$ containing $f_{n}(\mathbb{R})$ also contains $C$. Since
our degree function is exact on the chain $C_{n}$, 
\[
\left\langle \star,f_{n}\right\rangle =\left\langle \star,r_{C}(f_{n})\right\rangle \quad\mbox{and}\quad\left\langle \star,f\right\rangle =\left\langle \star,r_{C}(f)\right\rangle .
\]
Since $d\left(r_{C}(f_{n}),r_{C}(f)\right)\leq d(f_{n},f)$, also
$\lim r_{C}(f_{n})=r_{C}(f)$ in $\mathbf{F}(X)$. Now on 
\[
\mathbf{F}(\Gr_{C}^{\bullet})=\prod_{i=1}^{s}\mathbf{F}(\Gr_{C}^{i})\quad\mbox{with}\quad\Gr_{C}^{i}=[c_{i-1},c_{i}]
\]
the distance and degree are respectively given by 
\[
d\left((a_{i}),(b_{i})\right)^{2}=\sum_{i=1}^{s}d_{i}(a_{i},b_{i})^{2}\quad\mbox{and}\quad\left\langle \star,(a_{i})\right\rangle =\sum_{i=1}^{s}\left\langle \star_{i},a_{i}\right\rangle 
\]
where $d_{i}$ and $\left\langle \star_{i},-\right\rangle $ are induced
by the corresponding rank and degree functions
\[
\rank_{i}(z)=\rank(z)-\rank(c_{i-1})\quad\mbox{and}\quad\deg_{i}(z)=\deg(z)-\deg(c_{i-1})
\]
for $z\in\Gr_{C}^{i}$. All this reduces us to the case where $f=X(\mu)$
for some $\mu\in\mathbb{R}$. Now
\[
\left\langle \star,f_{n}\right\rangle =\gamma_{n,1}\deg(1_{X})+\sum_{i=2}^{s_{n}}(\gamma_{n,i}-\gamma_{n,i-1})\deg\left(f_{n}(\gamma_{n,i})\right)
\]
with $\Jump(f_{n})=\{\gamma_{n,1}<\cdots<\gamma_{n,s_{n}}\}$. Since
$\lim\mathbf{t}(f_{n})=\mathbf{t}(f)=(\mu,\cdots,\mu)$ in $\mathbb{R}_{\geq}^{r}$,
\[
\lim\gamma_{n,1}=\mu\quad\mbox{and}\quad\lim\mbox{\ensuremath{\sup}}\left\{ \gamma_{n,i}-\gamma_{n,i-1}:2\leq i\leq s_{n}\right\} =0.
\]
Since finally $\deg(X)$ is bounded above, we obtain 
\[
\lim\sup\left\langle \star,f_{n}\right\rangle \leq\mu\deg(1_{X})=\left\langle \star,f\right\rangle 
\]
and $\lim\left\langle \star,f_{n}\right\rangle =\left\langle \star,f\right\rangle $
if $\deg(X)$ is also bounded below. 
\end{proof}

\subsubsection{~}

For $f\in\mathbf{F}(X)$, the degree function $\left\langle f,-\right\rangle :\mathbf{F}(X)\rightarrow\mathbb{R}$
is continuous since 
\[
\left\langle f,g\right\rangle =\frac{1}{2}\left(\left\Vert f\right\Vert ^{2}+\left\Vert g\right\Vert ^{2}-d(f,g)^{2}\right)=\frac{1}{2}\left(\left\Vert f\right\Vert ^{2}+\left\Vert g\right\Vert ^{2}-d(f,g)^{2}\right).
\]
This also follows from proposition~\ref{prop:Contdeg} since for
every $x\in X$, 
\[
\left|\deg_{f}(x)\right|=\left|\left\langle f,x\right\rangle \right|\leq\left\Vert f\right\Vert \left\Vert x\right\Vert 
\]
with $\left\Vert x\right\Vert ^{2}=\rank(x)\leq\rank(1_{X})$, but
a bit more is actually true:
\begin{prop}
The degree function $\left\langle f,-\right\rangle :\mathbf{F}(X)\rightarrow\mathbb{R}$
is $\left\Vert f\right\Vert $-Lipschitzian. 
\end{prop}
\begin{proof}
We have to show that $\left|\left\langle f,h\right\rangle -\left\langle f,g\right\rangle \right|\leq\left\Vert f\right\Vert \cdot d(g,h)$
for every $g,h\in\mathbf{F}(X)$. Pick an apartment $S$ of $X$ with
$g,h\in\mathbf{F}(S)$ and set $g_{t}=(1-t)g+th\in\mathbf{F}(S)$
for $t\in[0,1]$. Since $\mathbf{F}(S)$ is the union of finitely
many closed (convex) chambers, there is an integer $N>0$, a finite
sequence $0=t_{0}<\cdots<t_{N}=1$ and maximal chains $C_{1},\cdots,C_{N}$
in $S$ such that for every $1\leq i\leq N$ and $t\in[t_{i-1},t_{i}]$,
$g_{t}$ belongs the closed chamber $\mathbf{F}(C_{i})$. Set $g_{i}=g_{t_{i}}$
for $i\in\{0,\cdots,N\}$. Since
\[
\left|\left\langle f,h\right\rangle -\left\langle f,g\right\rangle \right|=\left|\sum_{i=1}^{N}\left\langle f,g_{i}\right\rangle -\left\langle f,g_{i-1}\right\rangle \right|\leq\sum_{i=1}^{N}\left|\left\langle f,g_{i}\right\rangle -\left\langle f,g_{i-1}\right\rangle \right|
\]
and $d(g,h)=\sum_{i=1}^{N}d(g_{i-1},g_{i})$, we may assume that $g,h\in\mathbf{F}(C)$
for some maximal chain $C$ in $X$. Now choose an apartment $S$
of $X$ containing $C$ and $f(\mathbb{R})$ and let $f',g',h'$ be
the images of $f,g,h$ under $r_{C}:\mathbf{F}(X)\rightarrow\mathbf{F}(\Gr_{C}^{\bullet})$.
Then
\[
\begin{array}{rcl}
\left\langle f,h\right\rangle  & = & \left\langle f',h'\right\rangle \\
\left\langle f,g\right\rangle  & = & \left\langle f',g'\right\rangle 
\end{array}\quad\mbox{and}\quad\begin{array}{rcl}
\left\Vert f\right\Vert  & = & \left\Vert f'\right\Vert \\
d(g,h) & = & d(g',h')
\end{array}
\]
since $f,g,h\in\mathbf{F}(S)$ with $C\subset S$. This reduces us
further to the case of a finite boolean lattice $X$, where $\mathbf{F}(X)$
is a euclidean space and our claim is trivial. 
\end{proof}

\subsection{HN-filtrations}

Suppose now that our modular lattice $X$ is also equipped with a
degree function $\deg:X\rightarrow\mathbb{R}$ and let $\left\langle \star,-\right\rangle :\mathbf{F}(X)\rightarrow\mathbb{R}$
be its unique extension to a degree function on $\mathbf{F}(X)$,
as explained in~\ref{subsec:degfctXandF(X)}. 

\subsubsection{~}

We say that $X$ is \emph{semi-stable of slope} $\mu\in\mathbb{R}$
if and only if for every $x\in X$, $\deg(x)\leq\mu\,\rank(x)$ with
equality for $x=1_{X}$. More generally for every $x\leq y$ in $X$,
we say that the interval $[x,y]$ is \emph{semi-stable of slope }$\mu$
if and only if it is semi-stable of slope $\mu$ for the induced rank
and degree functions, i.e.~for every $z\in[x,y]$, 
\[
\deg(z)\leq\mu\left(\rank(z)-\rank(y)\right)+\deg(y)
\]
with equality for $z=y$. Note that for $x=y$, $[x,y]=\{x\}$ is
semi-stable of slope $\mu$ for every $\mu\in\mathbb{R}$. For any
$x<y$, the \emph{slope} of $[x,y]$ is defined by 
\[
\mu([x,y])=\frac{\deg(y)-\deg(x)}{\rank(y)-\rank(x)}\in\mathbb{R}.
\]

\subsubsection{~}

For any $x,y,z\in X$ with $x<y<z$, we have 
\[
\mu([x,z])=\frac{\rank(z)-\rank(y)}{\rank(z)-\rank(x)}\mu([y,z])+\frac{\rank(y)-\rank(x)}{\rank(z)-\rank(x)}\mu([x,y])
\]
thus one of the following cases occurs: 
\[
\begin{array}{rl}
 & \mu([x,y])<\mu([x,z])<\mu([y,z]),\\
\mbox{or} & \mu([x,y])>\mu([x,z])>\mu([y,z]),\\
\mbox{or} & \mu([x,y])=\mu([x,z])=\mu([y,z]).
\end{array}
\]

\begin{lem}
\label{lem:LozStab}Suppose that $x\leq x'\leq y'$ and $x\leq y\leq y'$
with $[x,y]$ semi-stable of slope $\mu$ and $[x',y']$ semi-stable
of slope $\mu'$. If $\mu>\mu'$, then also $y\leq x'$.
\end{lem}
\begin{proof}
Suppose not, i.e.~$x'<y\vee x'$ and $y\wedge x'<y$. Then
\[
\mu\stackrel{(1)}{\leq}\mu([y\wedge x',y])\stackrel{(2)}{\leq}\mu([x',y\vee x'])\stackrel{(3)}{\leq}\mu'
\]
since $(1)$ $y\wedge x'$ belongs to $[x,y]$ which is semi-stable
of slope $\mu$, $(3)$ $y\vee x'$ belongs to $[x',y']$ which is
semi-stable of slope $\mu'$, and $(2)$ follows from the definition
of $\mu$. 
\end{proof}

\subsubsection{~}

The main result of this section is the following proposition. 
\begin{prop}
\label{prop:CarHN}For any $\mathcal{F}\in\mathbf{F}(X)$, the following
conditions are equivalent.

\begin{enumerate}
\item For every $f\in\mathbf{F}(X)$, $\left\Vert \mathcal{F}\right\Vert ^{2}-2\left\langle \star,\mathcal{F}\right\rangle \leq\left\Vert f\right\Vert ^{2}-2\left\langle \star,f\right\rangle $.
\item For every $f\in\mathbf{F}(X)$, $\left\langle \star,f\right\rangle \leq\left\langle \mathcal{F},f\right\rangle $
with equality for $f=\mathcal{F}$.
\item For every $\gamma\in\mathbb{R}$, $\Gr_{\mathcal{F}}^{\gamma}$ is
semi-stable of slope $\gamma$. 
\end{enumerate}
Moreover, there is a unique such $\mathcal{F}$, and $\left\Vert \mathcal{F}\right\Vert ^{2}=\left\langle \star,\mathcal{F}\right\rangle $.
\end{prop}
\begin{proof}
It is sufficient to establish $(1)\Rightarrow(2)\Rightarrow(3)$,
and the existence (resp.~uniqueness) of an $\mathcal{F\in\mathbf{F}}(X)$
satisfying $(1)$ (resp.~$(3)$). We start with the following claim. 

\emph{There is a constant $A>0$ such that $\left\langle \star,f\right\rangle \leq A\left\Vert f\right\Vert $.
}Indeed, pick any maximal chain $C$ in $X$. Then $\left\langle \star,f\right\rangle \leq\left\langle \star,r_{C}(f)\right\rangle $
and $\left\Vert f\right\Vert =\left\Vert r_{C}(f)\right\Vert $ for
every $f\in\mathbf{F}(X)$. But on the finite dimensional $\mathbb{R}$-vector
space $\mathbf{F}(\Gr_{C}^{\bullet})$, $\left\langle \star,-\right\rangle :\mathbf{F}(\Gr_{C}^{\bullet})\rightarrow\mathbb{R}$
is a linear form while $\left\Vert -\right\Vert :\mathbf{F}(\Gr_{C}^{\bullet})\rightarrow\mathbb{R}_{+}$
is a euclidean norm. Our claim easily follows.

\emph{Existence in $(1)$. }Since $\left\langle \star,f\right\rangle \leq A\left\Vert f\right\Vert $,
the function $f\mapsto\left\Vert f\right\Vert ^{2}-2\left\langle \star,f\right\rangle $
is bounded below. Let $(f_{n})$ be any sequence in $\mathbf{F}(X)$
such that $\left\Vert f_{n}\right\Vert ^{2}-2\left\langle \star,f_{n}\right\rangle $
converges to $I=\inf\left\{ \left\Vert f\right\Vert ^{2}-2\left\langle \star,f\right\rangle :f\in\mathbf{F}(X)\right\} $.
By the CAT(0)-inequality, 
\[
2\left\Vert {\textstyle \frac{1}{2}}f_{n}+{\textstyle \frac{1}{2}}f_{m}\right\Vert ^{2}+{\textstyle \frac{1}{2}}d(f_{n},f_{m})^{2}\le\left\Vert f_{n}\right\Vert ^{2}+\left\Vert f_{m}\right\Vert ^{2}.
\]
By concavity of $f\mapsto\left\langle \star,f\right\rangle $, 
\[
\left\langle \star,{\textstyle \frac{1}{2}}f_{n}+{\textstyle \frac{1}{2}}f_{m}\right\rangle \geq{\textstyle \frac{1}{2}}\left\langle \star,f_{n}\right\rangle +{\textstyle \frac{1}{2}}\left\langle \star,f_{m}\right\rangle .
\]
We thus obtain
\[
\begin{array}{c}
2I+\frac{1}{2}d(f_{n},f_{m})^{2}\leq2\left(\left\Vert \frac{1}{2}f_{n}+\frac{1}{2}f_{m}\right\Vert ^{2}-2\left\langle \star,{\textstyle \frac{1}{2}}f_{n}+{\textstyle \frac{1}{2}}f_{m}\right\rangle \right)+\frac{1}{2}d(f_{n},f_{m})^{2}\\
\leq\left(\left\Vert f_{n}\right\Vert ^{2}-2\left\langle \star,f_{n}\right\rangle \right)+\left(\left\Vert f_{m}\right\Vert ^{2}-2\left\langle \star,f_{m}\right\rangle \right).
\end{array}
\]
It follows that $(f_{n})$ is a Cauchy sequence in $\mathbf{F}(X)$,
and therefore converges to some $\mathcal{F}\in\mathbf{F}(X)$. Then
$\left\Vert f_{n}\right\Vert \rightarrow\left\Vert \mathcal{F}\right\Vert $
and $\left\langle \star,f_{n}\right\rangle \rightarrow\frac{1}{2}\left(\left\Vert \mathcal{F}\right\Vert ^{2}-I\right)$.
By~proposition~\ref{prop:Contdeg}, $\left\Vert \mathcal{F}\right\Vert ^{2}-2\left\langle \star,\mathcal{F}\right\rangle \leq I$
thus actually $\left\Vert \mathcal{F}\right\Vert ^{2}-2\left\langle \star,\mathcal{F}\right\rangle =I$
by definition of $I$. 

$(1)$\emph{ implies $(2)$. }Suppose $(1)$. Then for any $f\in\mathbf{F}(X)$
and $t\geq0$, 
\[
\left\Vert \mathcal{F}\right\Vert ^{2}-2\left\langle \star,\mathcal{F}\right\rangle \leq\left\Vert \mathcal{F}+tf\right\Vert ^{2}-2\left\langle \star,\mathcal{F}+tf\right\rangle .
\]
Since $\left\Vert \mathcal{F}+tf\right\Vert ^{2}=\left\Vert \mathcal{F}\right\Vert ^{2}+t^{2}\left\Vert f\right\Vert ^{2}+2t\left\langle \mathcal{F},f\right\rangle $
and $\left\langle \star,\mathcal{F}+tf\right\rangle \geq\left\langle \star,\mathcal{F}\right\rangle +t\left\langle \star,f\right\rangle $,
\[
0\leq t^{2}\left\Vert f\right\Vert ^{2}+2t\left(\left\langle \mathcal{F},f\right\rangle -\left\langle \star,f\right\rangle \right).
\]
Since this holds for every $t\geq0$, indeed $\left\langle \star,f\right\rangle \leq\left\langle \mathcal{F},f\right\rangle $.
On the other hand, 
\[
\left\Vert \mathcal{F}\right\Vert ^{2}-2\left\langle \star,\mathcal{F}\right\rangle \leq\left\Vert t\mathcal{F}\right\Vert ^{2}-2\left\langle \star,t\mathcal{F}\right\rangle =t^{2}\left\Vert \mathcal{F}\right\Vert ^{2}-2t\left\langle \star,\mathcal{F}\right\rangle 
\]
for all $t\geq0$, therefore also $\left\Vert \mathcal{F}\right\Vert ^{2}=\left\langle \star,\mathcal{F}\right\rangle $. 

$(2)$ \emph{implies $(3)$. }Suppose $(2)$. Let $s$ be the number
of jumps of $\mathcal{F}$ and set 
\[
\mathcal{F}(\mathbb{R})=\left\{ c_{0}<\cdots<c_{s}\right\} \quad\mbox{and}\quad\Jump(\mathcal{F})=\left\{ \gamma_{1}>\cdots>\gamma_{s}\right\} .
\]
For $i\in\{1,\cdots,s\}$ and $\theta$ sufficiently close to $\gamma_{i}$,
let $f_{i,\theta}$ be the unique $\mathbb{R}$-filtration on $X$
such that $f_{i,\theta}(\mathbb{R})=\mathcal{F}(\mathbb{R})$ and
$\Jump(f_{i,\theta})\setminus\{\theta\}=\Jump(\mathcal{F})\setminus\{\gamma_{i}\}$.
Then 
\begin{eqnarray*}
\left\langle \star,f_{i,\theta}\right\rangle -\theta\deg\left(\Gr_{\mathcal{F}}^{\gamma_{i}}\right) & = & \left\langle \star,\mathcal{F}\right\rangle -\gamma_{i}\deg\left(\Gr_{\mathcal{F}}^{\gamma_{i}}\right)\\
\mbox{and}\quad\left\langle \mathcal{F},f_{i,\theta}\right\rangle -\theta\gamma_{i}\,\rank\left(\Gr_{\mathcal{F}}^{\gamma_{i}}\right) & = & \left\langle \mathcal{F},\mathcal{F}\right\rangle -\gamma_{i}^{2}\,\rank\left(\Gr_{\mathcal{F}}^{\gamma_{i}}\right).
\end{eqnarray*}
Since $\left\langle \star,f_{i,\theta}\right\rangle \leq\left\langle \mathcal{F},f_{i,\theta}\right\rangle $
and $\left\langle \star,\mathcal{F}\right\rangle =\left\langle \mathcal{F},\mathcal{F}\right\rangle $,
it follows that
\[
(\theta-\gamma_{i})\left(\gamma_{i}\,\rank\left(\Gr_{\mathcal{F}}^{\gamma_{i}}\right)-\deg\left(\Gr_{\mathcal{F}}^{\gamma_{i}}\right)\right)\geq0.
\]
Since this holds for every $\theta$ close to $\gamma_{i}$, it must
be that $\gamma_{i}=\mu\left(\Gr_{\mathcal{F}}^{\gamma_{i}}\right).$
Now for any $c_{i-1}<z<c_{i}$ and a sufficiently small $\epsilon>0$,
let $f_{i,z,\epsilon}$ be the unique $\mathbb{R}$-filtration on
$X$ such that $f_{i,z,\epsilon}(\mathbb{R})=\mathcal{F}(\mathbb{R})\cup\{z\}$
and $\Jump(f_{i,z,\epsilon})=\Jump(\mathcal{F})\cup\{\gamma_{i}+\epsilon\}$.
Then
\begin{eqnarray*}
\left\langle \star,f_{i,z,\epsilon}\right\rangle  & = & \left\langle \star,\mathcal{F}\right\rangle +\epsilon\deg\left({\textstyle \frac{z}{c_{i-1}}}\right)\\
\mbox{and}\quad\left\langle \mathcal{F},f_{i,z,\epsilon}\right\rangle  & = & \left\langle \mathcal{F},\mathcal{F}\right\rangle +\epsilon\gamma_{i}\,\rank\left({\textstyle \frac{z}{c_{i-1}}}\right).
\end{eqnarray*}
Since again $\left\langle \star,f_{i,z,\epsilon}\right\rangle \leq\left\langle \mathcal{F},f_{i,z,\epsilon}\right\rangle $
and $\left\langle \star,\mathcal{F}\right\rangle =\left\langle \mathcal{F},\mathcal{F}\right\rangle $,
we obtain
\[
\deg\left({\textstyle \frac{z}{c_{i-1}}}\right)\leq\gamma_{i}\,\rank\left({\textstyle \frac{z_{i}}{c_{i-1}}}\right).
\]
Thus $\Gr_{\mathcal{F}}^{\gamma_{i}}$ is indeed semi-stable of slope
$\gamma_{i}$ for all $i\in\{1,\cdots,s\}$.

\emph{Unicity in $(3)$. }Suppose that $\mathcal{F}$ and $\mathcal{F}'$
both satisfy $(3)$ and set
\[
\left\{ \gamma_{1}>\cdots>\gamma_{s}\right\} =\Jump(\mathcal{F})\cup\Jump(\mathcal{F}'),\quad\gamma_{0}=\gamma_{1}+1.
\]
We show by ascending induction on $i\in\{0,\cdots,s\}$ and descending
induction on $j\in\{i,\cdots,s\}$ that $\mathcal{F}(\gamma_{i})\leq\mathcal{F}'(\gamma_{j})$.
For $i=0$ or $j=s$ there is nothing to prove since $\mathcal{F}(\gamma_{0})=0_{X}$
and $\mathcal{F}'(\gamma_{s})=1_{X}$. Suppose now that $1\leq i\leq j<s$
and we already now $\mathcal{F}(\gamma_{i-1})\leq\mathcal{F}'(\gamma_{i-1})$
and $\mathcal{F}(\gamma_{i})\leq\mathcal{F}'(\gamma_{j+1})$. Then
$\mathcal{F}(\gamma_{i})\leq\mathcal{F}'(\gamma_{j})$ by lemma~\ref{lem:LozStab}.
Thus $\mathcal{F}(\gamma_{i})\leq\mathcal{F}'(\gamma_{i})$ for all
$i\in\{1,\cdots,s\}$. By symmetry $\mathcal{F}=\mathcal{F}'$.
\end{proof}
\begin{defn}
We call $\mathcal{F}\in\mathbf{F}(X)$ the \emph{Harder-Narasimhan
filtration} of $(X,\deg)$.
\end{defn}

\subsubsection{Example\label{subsec:ExampleProj}}

For $f\in\mathbf{F}(X)$ and the degree function $\deg_{f}(x)=\left\langle f,x\right\rangle $
on $X$, the Harder-Narasimhan filtration $\mathcal{F}\in\mathbf{F}(X)$
of $(X,\deg_{f})$ minimizes 
\[
g\mapsto\left\Vert f\right\Vert ^{2}+\left\Vert g\right\Vert ^{2}-2\left\langle f,g\right\rangle =d(f,g)^{2}
\]
thus plainly $\mathcal{F}=f$. More generally suppose that $Y$ is
a $\{0,1\}$-sublattice of $X$ with the induced rank function. Then
$\mathbf{F}(Y)\hookrightarrow\mathbf{F}(X)$ is an isometric embedding,
with a non-expanding retraction, namely the convex projection $p:\mathbf{F}(X)\twoheadrightarrow\mathbf{F}(Y)$
of \cite[II.2.4]{BrHa99}. Then for any $f\in\mathbf{F}(X)$, $y\mapsto\left\langle f,y\right\rangle $
is a degree function on $Y$ and the corresponding Harder-Narasimhan
filtration $\mathcal{F}\in\mathbf{F}(Y)$ equals $p(f)$. In particular,
\[
\left\langle f,g\right\rangle \leq\left\langle p(f),g\right\rangle 
\]
for every $f\in\mathbf{F}(X)$ and $g\in\mathbf{F}(Y)$ with equality
for $g=p(f)$.

\subsubsection{~}

If $X$ is complemented and $\deg:X\rightarrow\mathbb{R}$ is exact,
the Harder-Narasimhan filtration may also be characterized by the
following weakening of condition $(2)$: 

\emph{$(2')$ For every $f\in\mathbf{F}(X)$, $\left\langle \star,f\right\rangle \leq\left\langle \mathcal{F},f\right\rangle $. }

\noindent We have to show that for any $\mathcal{F}\in\mathbf{F}(X)$
satisfying $(2')$, $\left\langle \star,\mathcal{F}\right\rangle \geq\left\langle \mathcal{F},\mathcal{F}\right\rangle $.
Since $X$ is complemented, there is an $\mathbb{R}$-filtration $\mathcal{F}'$
on $X$ which is opposed to $\mathcal{F}$. Since $\deg$ is exact,
$f\mapsto\left\langle \star,f\right\rangle $ is additive, thus $\left\langle \star,\mathcal{F}\right\rangle +\left\langle \star,\mathcal{F}'\right\rangle =\left\langle \star,\mathcal{F}+\mathcal{F}'\right\rangle =0$
and indeed
\[
\left\langle \star,\mathcal{F}\right\rangle =-\left\langle \star,\mathcal{F}'\right\rangle \geq-\left\langle \mathcal{F},\mathcal{F}'\right\rangle =\left\langle \mathcal{F},\mathcal{F}\right\rangle .
\]
This also shows that then $\left\langle \star,\mathcal{F}'\right\rangle =\left\langle \mathcal{F},\mathcal{F}'\right\rangle $
for any $\mathcal{F}'\in\mathbf{F}(X)$ opposed to $\mathcal{F}$. 

\section{The Harder-Narasimhan formalism for categories (after André)\label{sec:HN4Cat}}

\subsection{Basic notions}

Let $\C$ be a category with a null object $0$, with kernels and
cokernels. Let $\sk\,\C$ be the \emph{skeleton} of $\C$: the isomorphism
classes of objects in $\C$. 

\subsubsection{~}

Let $X$ be an object of $\C$. Recall that a \emph{subobject} of
$X$ is an isomorphism class of monomorphisms with codomain $X$.
We write $x\hookrightarrow X$ for the subobject itself or any monomorphism
in its class. We say that $f:x\hookrightarrow X$ is \emph{strict}
if $f$ is a kernel. Equivalently, $f$ is strict if and only if $f=\im(f)$.
Dually, we have the notions of \emph{quotients} and \emph{strict quotients},
and $f\mapsto\coker f$ yields a bijection between strict subobjects
and strict quotients of $X$, written $x\mapsto X/x$. A \emph{short
exact sequence} is a pair of composable morphisms $f$ and $g$ such
that $f=\ker g$ and $g=\coker f$: it is thus of the form $0\rightarrow x\rightarrow X\rightarrow X/x\rightarrow0$
for some strict subobject $x$ of $X$. 

\subsubsection{~}

The class of all strict subobjects of $X$ will be denoted by $\Sub(X)$.
It is partially ordered: $(f:x\hookrightarrow X)\leq(f':x'\hookrightarrow X)$
if and only if there is a morphism $h:x\rightarrow x'$ such that
$f=f'\circ h$. Note that the morphism $h$ is then unique, and is
itself a strict monomorphism, realizing $x$ as a strict subobject
of $x'$. Conversely, a strict subobject $x$ of $x'$ yields a subobject
of $X$ which is not necessarily strict. 

\subsubsection{~\label{subsec:pullbackofstrictexists}}

The pull-back of a strict monomorphism $x\hookrightarrow X$ by any
morphism $Y\rightarrow X$ exists, and it is a strict monomorphism
$y\hookrightarrow Y$: it is the kernel of $Y\rightarrow X\rightarrow X/x$.
Dually, the push-out of a strict epimorphism $X\twoheadrightarrow X/x$
by any morphism $X\rightarrow Y$ exists, and it is a strict epimorphism
$Y\rightarrow Y/y$: it is the cokernel of $x\hookrightarrow X\rightarrow Y$. 

\subsubsection{~\label{subsec:joinmeet}}

Suppose that $\C$ is essentially small and the fiber product of any
pair of strict monomorphisms $x\hookrightarrow X$ and $y\hookrightarrow X$
(which exists by~\ref{subsec:pullbackofstrictexists}) induces a
\emph{strict} monomorphism $x\times_{X}y\rightarrow X$. Then $\Sub(X)$
is a set and $(\Sub(X),\leq)$ is a bounded lattice, with maximal
element $X$ and minimal element $0$. The meet of $x,y\in\Sub(X)$
is the image of $x\times_{X}y\rightarrow X$, also given by the less
symmetric formulas 
\[
x\wedge y=\ker(x\rightarrow X/y)=\ker(y\rightarrow X/x).
\]
The join of $x,y$ is the kernel of the morphism from $X$ to the
amalgamated sum of $X\rightarrow X/x$ and $X\rightarrow X/y$, also
given by the less symmetric formulas
\[
x\vee y=\ker\left(X\rightarrow\coker\left(x\rightarrow X/y\right)\right)=\ker\left(X\rightarrow\coker\left(y\rightarrow X/x\right)\right).
\]

\subsubsection{~\label{subsec:defrkdegonC}}

A \emph{degree function} on $\C$ is a function $\deg:\sk\,\C\rightarrow\mathbb{R}$
which is additive on short exact sequences and such that if $f:X\rightarrow Y$
is any morphism in $\C$, then $\deg(\coim f)\leq\deg(\im f)$. It
is \emph{exact} if $-\deg:\sk\,\C\rightarrow\mathbb{R}$ is also a
degree function on $\C$. A \emph{rank function} on $\C$ is an exact
degree function $\rank:\sk\,\C\rightarrow\mathbb{R}_{+}$ such that
for every $X\in\sk\,\C$, $\rank(X)=0$ if and only if $X=0$.

\subsubsection{~}

Under the assumptions of \ref{subsec:joinmeet}, if $\deg:\sk\,\C\rightarrow\mathbb{R}$
is a degree function on $\C$, then for every object $X$ of $\C$,
$x\mapsto\deg(x)$ is a degree function on $\Sub(X)$. Indeed, for
every $x,y\in\Sub(X)$, we have a commutative diagram with exact rows
\[
\xyC{1.5pc}\xyR{1.5pc}\xymatrix{0\ar[r] & x\wedge y\ar[r] & x\ar[r]\ar[d]^{f} & x/x\wedge y\ar[r]\ar[d]^{\overline{f}} & 0\\
0 & Q\ar[l] & X/y\ar[l]^{\pi} & I\ar[l] & 0\ar[l]\\
0 & X/x\vee y\ar[l]\ar[u]^{\overline{g}} & X\ar[l]\ar[u]^{g} & x\vee y\ar[l] & 0\ar[l]
}
\]
where $I=\im(f)$ and $Q=\coker(f)=\im(\pi\circ g)$ with $x/x\wedge y=\coim(f)$
and $X/x\vee y=\coim(\pi\circ g)$. It follows that
\[
\begin{array}{lrcl}
 & \deg(x)-\deg(x\wedge y)=\deg(x/x\wedge y) & \leq & \deg I=\deg(X/y)-\deg(Q)\\
\mbox{and\quad} & \deg(X)-\deg(x\vee y)=\deg(X/x\vee y) & \leq & \deg Q
\end{array}
\]
thus since also $\deg(X/y)=\deg(X)+\deg(y)$, 
\[
\deg(x)+\deg(y)\leq\deg(x\wedge y)+\deg(x\vee y).
\]
If $\deg:\sk\,\C\rightarrow\mathbb{R}$ is exact, so is $\deg:\Sub(X)\rightarrow\mathbb{R}$.
If $\rank:\sk\,\C\rightarrow\mathbb{R}_{+}$ is a rank function, then
so is $\rank:\Sub(X)\rightarrow\mathbb{R}_{+}$. 

\subsubsection{~\label{subsec:propofCwithrank}}

Suppose that $\C$ satisfies the assumptions of \ref{subsec:joinmeet}
and admits an integer-valued rank function $\rank:\sk\,\C\rightarrow\mathbb{N}$.
We then have the following properties:
\begin{itemize}
\item \emph{$\C$ is modular of finite length in the following sense: for
every object $X$ of $\C$, the lattice $\left(\Sub(X),\leq\right)$
of strict subobjects of $X$ is modular of finite length.}
\end{itemize}
This follows from~\ref{subsec:RankImpliesModular}. We write $\length(X)$
for the length of $\Sub(X)$. 
\begin{itemize}
\item \emph{For every $X\in\C$ and any $x$ in $\Sub(X)$, the following
maps are mutually inverse rank-preserving isomorphisms of lattices:}\emph{\small{}
\[
\xyC{1pc}\xyR{0.5ex}\xymatrix{[0,x]\ar@{<->}[r] & \Sub(x) &  & [x,X]\ar@{<->}[r] & \Sub(X/x)\\
y\ar@{|->}[r] & y & \mbox{and} & y\ar@{|->}[r] & \im(y\rightarrow X/x)\\
\im(z\rightarrow X) & z\ar@{|->}[l] &  & \ker\left(X\rightarrow(X/x)/z\right) & z\ar@{|->}[l]
}
\]
}{\small \par}
\item \emph{For any $f:Z\rightarrow Y$ in $\C$ with trivial kernel and
cokernel, the following maps are rank-preserving mutually inverse
isomorphisms of lattices:
\[
\xyC{1pc}\xyR{0.5ex}\xymatrix{\Sub(Y)\ar@{<->}[r] & \Sub(Z)\\
y\ar@{|->}[r] & \ker(Z\rightarrow Y/y)\\
\im(z\rightarrow Y) & z\ar@{|->}[l]
}
\]
} 
\end{itemize}
Write $(\alpha,\beta)$ for any of these pairs of maps. One checks
that for $y$ and $z$ as above, 
\[
\beta\circ\alpha(y)\leq y\quad\mbox{and}\quad z\leq\alpha\circ\beta(z).
\]
It is therefore sufficient to establish that all of our maps are rank-preserving
(the rank on $[x,X]$ maps $y$ to $\rank(y)-\rank(x)=\rank(y/x)$).
Writing $(\alpha_{i},\beta_{i})$ for the $i$-th pair, this is obvious
for $\alpha_{1}$; for $\beta_{1}$, $\im(z\rightarrow X)$ and $z=\coim(z\rightarrow X)$
have the same rank; for $\alpha_{2}$, $\im(y\rightarrow X/x)$ and
$y/x=\coim(y\rightarrow X/x)$ have the same rank; for $\beta_{2}$,
$X\rightarrow(X/x)/z$ is an epimorphism, its coimage $X/\beta_{2}(z)$
and image $(X/x)/z$ thus have the same rank, and so do $\beta_{2}(z)/x$
and $z$; for $\alpha_{3}$, the cokernel of $Z\rightarrow Y/y$ is
trivial, thus $Y/y=\im(Z\rightarrow Y/y)$ and $Z/\alpha_{3}(y)=\coim(Z\rightarrow Y/y)$
have the same rank, and so do $y$ and $\alpha_{3}(y)$ since also
$\rank(Z)=\rank(Y)$; for $\beta_{3}$, the kernel of $z\rightarrow Y$
is trivial, thus $z=\coim(z)$ and $\beta_{3}(z)=\im(z\rightarrow Y)$
have the same rank.
\begin{itemize}
\item \emph{The composition of two strict monomorphism (resp.~epimorphisms)
is a strict monomorphism (resp.~epimorphisms). }
\item \emph{For every $X\in\C$ and $a\leq b$ in $\Sub(X)$, the following
maps are mutually inverse rank-preserving isomorphisms of lattices
\[
\xyC{1pc}\xyR{0.5ex}\xymatrix{[a,b]\ar@{<->}[r] & \Sub(b/a)\\
x\ar@{|->}[r] & \im(x\rightarrow b/a)\\
\ker(b\rightarrow(b/a)/y) & y\ar@{|->}[l]
}
\]
}
\end{itemize}
This follows easily from the previous statements.
\begin{itemize}
\item \emph{For any morphism $f:X\rightarrow Y$, the induced morphism $\overline{f}:\coim(f)\rightarrow\im(f)$
has trivial kernel and cokernel.}
\end{itemize}
The kernel of $\overline{f}$ always pulls-back through $X\rightarrow\coim(f)$
to the kernel of $f$, so it now also has to be the image of that
kernel, which is trivial by definition of $\coim(f)$. Similarly,
the image of $\overline{f}$ always pushes-out through $\im(f)\rightarrow Y$
to the image of $f$, so it now has to be this image, i.e.~$\coker(\overline{f})=0$. 
\begin{itemize}
\item \emph{The length function $\length:\sk\,\C\rightarrow\mathbb{N}$
is an integer-valued rank function.}
\end{itemize}
Indeed, for a short exact sequence $0\rightarrow x\rightarrow X\rightarrow X/x\rightarrow0$
in $\C$, 
\begin{eqnarray*}
\length(X) & = & \length(\Sub(X))\\
 & = & \length([0,x])+\length([x,X])\\
 & = & \length(\Sub(x))+\length(\Sub(X/x))\\
 & = & \length(x)+\length(X/x)
\end{eqnarray*}
and for any morphism $f:X\rightarrow Y$, since $\ker(\overline{f})=0=\coker(\overline{f})$,
\[
\length(\coim(f))=\length(\Sub(\coim(f))=\length(\Sub(\im(f))=\length(\im(f)).
\]

\subsubsection{~}

Suppose that $\C$ is a proto-abelian category in the sense of André~\cite[\S 2]{An09}:
$(1)$ every morphism with zero kernel (resp.~cokernel) is a monomorphism
(resp.~an epimorphism) and $(2)$ the pull-back of a strict epimorphism
by a strict monomorphism is a strict epimorphism and the push-out
of a strict monomorphism by a strict epimorphism is a strict monomorphism.
In this case, a degree function on $\C$ is a function $\deg:\sk\,\C\rightarrow\mathbb{R}$
which is additive on short exact sequences and non-decreasing on mono-epi's
(=morphisms which are simultaneously monomorphisms and epimorphisms).
Our definitions for rank and degree functions on such a category $\C$
are thus more restrictive than those of André (beyond the differences
between the allowed codomains of these functions): he only requires
the slope $\mu=\deg/\rank$ to be non-decreasing on mono-epi's, while
we simultaneously require the denominator to be constant and the numerator
to be non-decreasing on mono-epi's. In all the examples we know, the
rank functions satisfy our assumptions.

\subsection{HN-filtrations}

Let $\C$ be an essentially small category with null objects, kernels
and cokernels, such that the fiber product of strict subobjects $x,y\hookrightarrow X$
is a strict subobject $x\wedge y\hookrightarrow X$, and let $\rank:\sk\,\C\rightarrow\mathbb{N}$
be a fixed, integer-valued rank function on $\C$. 

\subsubsection{~}

For every object $X$ of $\C$, write $\mathbf{F}(X)$ for the set
of $\mathbb{R}$-filtrations on the modular lattice $\Sub(X)$. Thus
$\mathbf{F}(X)=\mathbf{F}(\Sub(X))$ is the set of ``$\mathbb{R}$-filtrations
on $X$ by strict subobjects''. It is equipped with its scalar multiplication,
symmetric addition, its collection of apartments and facet decomposition.
The rank function on $\C$ moreover induces a rank function on $\Sub(X)$,
which equips $\mathbf{F}(X)$ with a scalar product $\left\langle -,-\right\rangle $,
a norm $\left\Vert -\right\Vert $, a complete CAT(0)-distance $d(-,-)$,
the underlying topology, and the standard degree function $\deg:\mathbf{F}(X)\rightarrow\mathbb{R}$
which maps $\mathcal{F}$ to 
\[
\deg(\mathcal{F})=\left\langle X(1),\mathcal{F}\right\rangle =\sum_{\gamma\in\mathbb{R}}\gamma\,\rank\left(\Gr_{\mathcal{F}}^{\gamma}\right).
\]
Here $X(\mu)$ is the $\mathbb{R}$-filtration on $X$ with a single
jump at $\mu$ and we may either view $\Gr_{\mathcal{F}}^{\gamma}$
as an interval in $\Sub(X)$, or as the corresponding strict subquotient
of $X$. For a strict subquotient $y/x$ of $X$ and $\mathcal{F}\in\mathbf{F}(X)$,
we denote by $\mathcal{F}_{y/x}$ the induced $\mathbb{R}$-filtration
on $y/x$, given by $\mathcal{F}_{y/x}(\gamma)=(\mathcal{F}(\gamma)\wedge y)\vee x/x=(\mathcal{F}(\gamma)\vee x)\wedge y/x$. 

\subsubsection{~\label{subsec:CatF(C)}}

We denote by $\F(\C)$ the category whose objects are pairs $(X,\mathcal{F})$
with $X\in\C$ and $\mathcal{F}\in\mathbf{F}(X)$. A \emph{morphism}
$(X,\mathcal{F})\rightarrow(Y,\mathcal{G})$ in $\F(\C)$ is a morphism
$f:X\rightarrow Y$ in $\C$ such that for any $\gamma\in\mathbb{R}$,
$f(\mathcal{F}(\gamma))\subseteq\mathcal{G}(\gamma)$. Here $f:\Sub(X)\rightarrow\Sub(Y)$
maps $x$ to 
\[
\im(x\hookrightarrow X\stackrel{f}{\longrightarrow}Y)
\]
and we have switched to the notation $\subseteq$ for the partial
order $\leq$ on $\Sub(-)$. The category $\F(\C)$ is essentially
small, and it also has a zero object, kernels and cokernels. For the
above morphism, they are respectively given by $(\ker(f),\mathcal{F}_{\ker(f)})$
and $(\coker(f),\mathcal{G}_{\coker(f)})$. The fiber product of strict
monomorphisms is a strict monomorphism. The forgetful functor $\omega:\F(\C)\rightarrow\C$
which takes $(X,\mathcal{F})$ to $X$ is exact and induces a lattice
isomorphism $\Sub(X,\mathcal{F})\simeq\Sub(X)$, whose inverse maps
$x$ to $(x,\mathcal{F}_{x})$. The category $\F(\C)$ is equipped
with rank and degree functions, 
\[
\rank(X,\mathcal{F})\eqd\rank(X)\quad\mbox{and}\quad\deg(X,\mathcal{F})\eqd\deg(\mathcal{F}).
\]
Indeed, the first formula plainly defines an integer valued rank function
on $\F(\C)$, which thus satisfies all the properties of \ref{subsec:propofCwithrank}.
For any exact sequence 
\[
0\rightarrow(x,\mathcal{F}_{x})\rightarrow(X,\mathcal{F})\rightarrow(X/x,\mathcal{F}_{X/x})\rightarrow0
\]
in $\F(\C)$, there is an apartment $S$ of $\Sub(X)$ containing
$\mathcal{F}(\mathbb{R})$ and $C=\{0,x,1_{X}\}$; the corresponding
apartment of $\mathbf{F}(X)$ contains $X(1)$ and $\mathcal{F}$,
thus by \ref{subsec:ScalAndrC}
\begin{eqnarray*}
\deg(X,\mathcal{F}) & = & \left\langle X(1),\mathcal{F}\right\rangle \\
 & = & \left\langle r_{C}\left(X(1)\right),r_{C}(\mathcal{F})\right\rangle \\
 & = & \left\langle x(1),\mathcal{F}_{x}\right\rangle +\left\langle X/x(1),\mathcal{F}_{X/x}\right\rangle \\
 & = & \deg(x,\mathcal{F}_{x})+\deg(X/x,\mathcal{F}_{X/x}).
\end{eqnarray*}
For a morphism $f:(X,\mathcal{F})\rightarrow(Y,\mathcal{G})$ with
trivial kernel and cokernel, the induced map $f:\Sub(X)\rightarrow\Sub(Y)$
is a rank preserving lattice isomorphism, thus 
\begin{eqnarray*}
\deg(X,\mathcal{F}) & = & \gamma_{1}\cdot\rank(X)+{\textstyle \sum_{i=2}^{s}}(\gamma_{i}-\gamma_{i-1})\cdot\rank\left(\mathcal{F}(\gamma_{i})\right)\\
 & = & \gamma_{1}\cdot\rank(Y)+{\textstyle \sum_{i=2}^{s}}(\gamma_{i}-\gamma_{i-1})\cdot\rank\left(f(\mathcal{F}(\gamma_{i}))\right)\\
 & \leq & \gamma_{1}\cdot\rank(Y)+{\textstyle \sum_{i=2}^{s}}(\gamma_{i}-\gamma_{i-1})\cdot\rank\left(\mathcal{G}(\gamma_{i})\right)\\
 & = & \deg(Y,\mathcal{G}).
\end{eqnarray*}
where $\left\{ \gamma_{1}<\cdots<\gamma_{s}\right\} =\Jump(\mathcal{F})\cup\Jump(\mathcal{G})$.
This shows that $\deg:\sk\,\F(\C)\rightarrow\mathbb{R}$ is indeed
a degree function on $\F(\C)$. Note also that with notations as above,
we have $\deg(X,\mathcal{F})=\deg(X,\mathcal{G})$ if and only if
$\mathcal{G}(\gamma)=\im(\mathcal{F}(\gamma)\rightarrow Y)$ for every
$\gamma\in\mathbb{R}$. 

\subsubsection{~}

A degree function $\deg:\sk\,\C\rightarrow\mathbb{R}$ on $\C$ gives
rise to a degree function on $\Sub(X)$ for every $X\in\C$, which
yields an \emph{Harder-Narasimhan} $\mathbb{R}$-filtration $\mathcal{F}_{HN}(X)\in\mathbf{F}(X)$
on $X$: the unique $\mathbb{R}$-filtration $\mathcal{F}$ on $X$
(by strict subobjects) such that $\Gr_{\mathcal{F}}^{\gamma}$ is
semi-stable of slope $\gamma$ for every $\gamma\in\mathbb{R}$. Here
semi-stability may either refer to the lattice notion of semi-stable
intervals in $\Sub(X)$, as defined earlier, or to the corresponding
categorical notion: an object $Y$ of $\C$ is semi-stable of slope
$\mu\in\mathbb{R}$ if and only if $\deg(Y)=\mu\,\rank(Y)$ and $\deg(y)\leq\mu\,\rank(y)$
for every strict subobject $y$ of $Y$. This is equivalent to: $\deg(Y)=\mu\,\rank(Y)$
and $\deg(Y/y)\geq\mu\,\rank(y)$ for every strict subobject $y$
of $Y$. Note that $Y=0$ is semi-stable of slope $\mu$ for every
$\mu\in\mathbb{R}$. In general, the slope of a nonzero object $X$
of $\C$ is given by 
\[
\mu(X)=\frac{\deg(X)}{\rank(X)}\in\mathbb{R}.
\]
For any $x\in\Sub(X)$ with $x\neq0$ and $X/x\neq0$, 
\[
\mu(X)=\frac{\rank(x)}{\rank(X)}\mu(x)+\frac{\rank(X/x)}{\rank(X)}\mu(X/x)
\]
thus either one of the following cases occur:
\[
\begin{array}{rl}
 & \mu(x)<\mu(X)<\mu(X/x),\\
\mbox{or} & \mu(x)>\mu(X)>\mu(X/x),\\
\mbox{or} & \mu(x)=\mu(X)=\mu(X/x).
\end{array}
\]

\subsubsection{~}

We claim that the Harder-Narasimhan filtration $X\mapsto\mathcal{F}_{HN}(X)$
is functorial. This easily follows from the next classical lemma,
a categorical variant of lemma~\ref{lem:LozStab}.
\begin{lem}
\label{lem:Stab}Suppose that $A$ and $B$ are semi-stable of slope
$a>b$. Then 
\[
\Hom_{\C}(A,B)=0.
\]
\end{lem}
\begin{proof}
Suppose $f:A\rightarrow B$ is nonzero, i.e.~$\coim(f)\neq0$ and
$\im(f)\neq0$. Then 
\[
a\stackrel{(1)}{\leq}\mu(\coim(f))\stackrel{(2)}{\leq}\mu(\im(f))\stackrel{(3)}{\leq}b
\]
since $(1)$ $A$ is semi-stable of slope $a$, $(3)$ $B$ is semi-stable
of slope $b$, and $(2)$ follows from the definition of $\mu$. This
is a contradiction, thus $f=0$.
\end{proof}

\subsubsection{~}

We thus obtain a \emph{Harder-Narasimhan functor} 
\[
\mathcal{F}_{HN}:\C\rightarrow\F(\C)
\]
which is a section of the forgetful functor $\omega:\F(\C)\rightarrow\C$.
The original degree function on $\C$ may be retrieved from the associated
functor $\mathcal{F}_{HN}$ by composing it with the standard degree
function on $\F(\C)$ which takes $(X,\mathcal{F})$ to $\deg(\mathcal{F})$.
The above construction thus yields an injective map from the set of
all degree functions on $\C$ to the set of all sections $\C\rightarrow\F(\C)$
of $\omega:\F(\C)\rightarrow\C$. A functor in the image of this map
is what André calls a \emph{slope filtration on} $\C$ \cite[\S 4]{An09}. 
\begin{rem}
For the rank and degree functions on $\C'=\F(\C)$ defined in section~\ref{subsec:CatF(C)},
the Harder-Narasimhan filtration is tautological: $\mathcal{F}_{HN}(X,\mathcal{F})=\mathcal{F}$
in 
\[
\mathbf{F}(X)=\mathbf{F}(\Sub(X))=\mathbf{F}(\Sub(X,\mathcal{F}))=\mathbf{F}(X,\mathcal{F}).
\]
\end{rem}

\subsubsection{~}

As mentioned in the introduction, our Harder-Narasimhan formalism
for categories is closely related to André's formalism in \cite{An09},
which indeed was our main source of inspiration. The formalism used
by Fargues in~\cite{Fa10} is a specialization of André's, with a
set-up closer to what we will have in the next section. Other formalisms
have been proposed, dealing with categories equipped with auxilliary
structures: triangulations in~\cite{Br07}, exact sequences and geometric
structures in~\cite{Ch10}.

\section{The Harder-Narasimhan formalism on quasi-Tannakian categories\label{sec:HN4QT}}

\subsection{Tannakian categories\label{subsec:HypOnA}}

Let $k$ be a field and let $\A$ be a $k$-linear tannakian category
\cite{De90} with unit $1_{\A}$ and ground field $k_{A}=\End_{\A}(1_{A})$,
an extension of $k$. Let also $G$ be a reductive group over $k$.
We denote by $\Rep(G)$ the $k$-linear tannakian category of algebraic
representations of $G$ on finite dimensional $k$-vector spaces.
Finally, let $\omega_{G,\A}:\Rep(G)\rightarrow\A$ be a fixed exact
and faithful $k$-linear $\otimes$-functor. 

\subsubsection{~}

The category $\A$ is equipped with a natural integer-valued rank
function
\[
\rank_{\A}:\sk\,\A\rightarrow\mathbb{N}.
\]
Indeed, recall that a fiber functor on $\A$ is an exact faithful
$k_{\A}$-linear $\otimes$-functor 
\[
\omega_{\A,\ell}:\A\rightarrow\Vect_{\ell}
\]
for some extension $\ell$ of $k_{\A}$. The existence of such fiber
functors is part of the definition of tannakian categories, and any
two such functors $\omega_{\A,\ell_{1}}$ and $\omega_{\A,\ell_{2}}$
become isomorphic over some common extension $\ell_{3}$ of $\ell_{1}$
and $\ell_{2}$~\cite[\S 1.10]{De90}: we may thus set
\[
\forall X\in\sk\,\A:\qquad\rank_{\A}(X)\eqd\dim_{\ell}\left(\omega_{\A,\ell}(X)\right).
\]
This equips $\Sub(X)$ with a natural rank function and $\mathbf{F}(X)=\mathbf{F}(\Sub(X))$
with a natural norm, CAT(0)-distance and scalar product \textendash{}
for every object $X$ of $\A$. 

\subsubsection{~}

The category $\F(\A)$ is a quasi-abelian $k_{\A}$-linear rigid $\otimes$-category,
with 
\[
(X_{1},\mathcal{F}_{1})\otimes(X_{2},\mathcal{F}_{2})\eqd(X_{1}\otimes X_{2},\mathcal{F}_{1}\otimes\mathcal{F}_{2})\quad\mbox{and}\quad(X,\mathcal{F})^{\ast}\eqd(X^{\ast},\mathcal{F}^{\ast})
\]
where $\mathcal{F}_{1}\otimes\mathcal{F}_{2}\in\mathbf{F}(X_{1}\otimes X_{2})$
and $\mathcal{F}^{\ast}\in\mathbf{F}(X^{\ast})$ are respectively
given by 
\[
(\mathcal{F}_{1}\otimes\mathcal{F}_{2})(\gamma)\eqd{\textstyle \sum_{\gamma_{1}+\gamma_{2}=\gamma}}\mathcal{F}_{1}(\gamma_{1})\otimes\mathcal{F}_{2}(\gamma_{2})\quad\mbox{and}\quad\mathcal{F}^{\ast}(\gamma)\eqd\left(X/\mathcal{F}_{+}(\gamma)\right)^{\ast}.
\]
Note that the formula defining $\mathcal{F}_{1}\otimes\mathcal{F}_{2}$
indeed makes sense, since $\mathcal{F}_{1}(\mathbb{R})$ and $\mathcal{F}_{2}(\mathbb{R})$
are finite subsets of $\Sub(X_{1})$ and $\Sub(X_{2})$, and the $\otimes$-product
is exact. For the standard degree function $\deg_{\A}:\sk\,\F(\A)\rightarrow\mathbb{R}$
of section~\ref{subsec:CatF(C)}, 
\[
\deg_{\A}\left(\mathcal{F}_{1}\otimes\mathcal{F}_{2}\right)=\rank_{\A}(X_{1})\cdot\deg_{\A}(\mathcal{F}_{2})+\rank_{\A}(X_{2})\cdot\deg_{\A}(\mathcal{F}_{1})
\]
and $\deg_{\A}(\mathcal{F}^{\ast})=-\deg_{\A}(\mathcal{F})$. This
can be checked after applying some fiber functor $\omega_{\A,\ell}:\A\rightarrow\Vect_{\ell}$
as above: the formulas are easily established in $\Vect_{\ell}$. 

\subsubsection{~}

We denote by $\mathbf{F}(\omega_{G,\A})$ the set of all factorizations
\[
\omega_{G,\A}:\Rep(G)\stackrel{\mathcal{F}}{\longrightarrow}\F(\A)\stackrel{\omega}{\longrightarrow}\A
\]
of our given exact $\otimes$-functor $\omega_{G,\A}$ through a $k$-linear
exact $\otimes$-functor 
\[
\mathcal{F}:\Rep(G)\rightarrow\F(\A).
\]
Thus for every $\tau\in\Rep(G)$, we have an evaluation map 
\[
\mathbf{F}(\omega_{G,\A})\rightarrow\mathbf{F}(\omega_{G,\A}(\tau)),\qquad\mathcal{F}\mapsto\mathcal{F}(\tau).
\]
For instance, the trivial filtration $0\in\mathbf{F}(\omega_{G,\A})$
maps $\tau\in\Rep(G)$ to the $\mathbb{R}$-filtration on $\omega_{G,\A}(\tau)$
with a single jump at $\gamma=0$, i.e.~$0(\tau)=\omega_{G,\A}(\tau)(0)$. 
\begin{thm}
The set $\mathbf{F}(\omega_{G,\A})$ is equipped with a scalar multiplication
and a symmetric addition map given by the following formulas: for
every $\tau\in\Rep(G)$, 
\[
(\lambda\cdot\mathcal{F})(\tau)\eqd\lambda\cdot\mathcal{F}(\tau)\quad\mbox{and}\quad(\mathcal{F}+\mathcal{G})(\tau)\eqd\mathcal{F}(\tau)+\mathcal{G}(\tau).
\]
The choice of a faithful representation $\tau$ of $G$ equips $\mathbf{F}(\omega_{G,\A})$
with a norm, a distance, and a scalar product given by the following
formulas: for $\mathcal{F},\mathcal{G}$ in $\mathbf{F}(\omega_{G,\A})$,
\[
\left\Vert \mathcal{F}\right\Vert _{\tau}\eqd\left\Vert \mathcal{F}(\tau)\right\Vert ,\quad d_{\tau}(\mathcal{F},\mathcal{G})\eqd d(\mathcal{F}(\tau),\mathcal{G}(\tau))\quad\mbox{and}\quad\left\langle \mathcal{F},\mathcal{G}\right\rangle _{\tau}\eqd\left\langle \mathcal{F}(\tau),\mathcal{G}(\tau)\right\rangle .
\]
The resulting metric space $(\mathbf{F}(\omega_{G,\A}),d_{\tau})$
is CAT(0) and complete. The underlying metrizable topology on $\mathbf{F}(\omega_{G,\A})$
does not depend upon the chosen $\tau$.
\end{thm}
\begin{proof}
If $\A=\Vect_{k_{\A}}$ and $\omega_{G,\A}$ is the standard fiber
functor $\omega_{G,k_{\A}}$ which maps a representation $\tau$ of
$G$ on the $k$-vector space $V(\tau)$ to the $k_{\A}$-vector space
$V(\tau)\otimes k_{\A}$, then $\mathbf{F}(\omega_{G,k_{\A}})$ is
the vectorial Tits building of $G_{k_{\A}}$ studied in \cite[Chapter 4]{Co14}
where everything can be found. For the general case, pick an extension
$\ell$ of $k_{\A}$ and a fiber functor $\omega_{\A,\ell}:\A\rightarrow\Vect_{\ell}$
such that $\omega_{\A,\ell}\circ\omega_{G,\A}$ is $\otimes$-isomorphic
to the standard fiber functor $\omega_{G,\ell}$. Then, for every
$\tau\in\Rep(G)$, we obtain a commutative diagram
\[
\xyC{2pc}\xyR{2pc}\xymatrix{\mathbf{F}\left(\omega_{G,\A}\right)\ar@{^{(}->}[r]\ar[d] & \mathbf{F}\left(\omega_{G,\ell}\right)\ar[d]\\
\mathbf{F}\left(\omega_{G,\A}(\tau)\right)\ar@{^{(}->}[r] & \mathbf{F}\left(\omega_{G,\ell}(\tau)\right)
}
\]
The horizontal maps are injective since $\omega_{\A,\ell}$ is exact
and faithful. The second vertical map is continuous, and so is therefore
also the first one (for the induced topologies). Moreover, both vertical
maps are injective if $\tau$ is a faithful representation of $G$
by \cite[Corollary 87]{Co14}. For the first claims, we have to show
that the functors $\Rep(G)\rightarrow\F(\A)$ defined by the formulas
for $\lambda\cdot\mathcal{F}$ and $\mathcal{F}+\mathcal{G}$ are
exact and compatible with tensor products: this can be checked after
post-composition with the fiber functor $\omega_{\A,\ell}$, see \cite[Section 3.11.10]{Co14}.
It follows that for any faithful $\tau$, $\mathbf{F}(\omega_{G,\A})$
is a convex subset of $\mathbf{F}(\omega_{G,\A}(\tau))$ and $\mathbf{F}(\omega_{G,\ell})$,
the function $d_{\tau}$ is a CAT(0)-distance on $\mathbf{F}(\omega_{G,\A})$
and the resulting topology does not depend upon the chosen $\tau$~\cite[Section 4.2.11]{Co14}.
It remains to establish that $\left(\mathbf{F}(\omega_{G,\A}),d_{\tau}\right)$
is complete, and this amounts to showing that $\mathbf{F}(\omega_{G,\A})$
is closed in $\mathbf{F}(\omega_{G,\ell})$. But if $\mathcal{F}_{n}\in\mathbf{F}(\omega_{G,\A})$
converges to $\mathcal{F}\in\mathbf{F}(\omega_{G,\ell})$, then for
every $\tau\in\Rep(G)$, $\mathcal{F}_{n}(\tau)\in\mathbf{F}(\omega_{G,\A}(\tau))$
converges to $\mathcal{F}(\tau)\in\mathbf{F}(\omega_{G,\ell}(\tau))$,
thus actually $\mathcal{F}(\tau)\in\mathbf{F}(\omega_{G,\A}(\tau))$
since $\mathbf{F}(\omega_{G,\A}(\tau))$ is (complete thus) closed
in $\mathbf{F}(\omega_{G,\ell}(\tau))$, therefore indeed $\mathcal{F}\in\mathbf{F}(\omega_{G,\A})$. 
\end{proof}

\subsubsection{~\label{subsec:ProjAttau}}

For a faithful representation $\tau$ of $G$, we have just seen that
evaluation at $\tau$ identifies $\mathbf{F}(\omega_{G,\A})$ with
a closed convex subset $\mathbf{F}(\omega_{G,\A})(\tau)$ of $\mathbf{F}(\omega_{G,\A}(\tau))$.
Let 
\[
p:\mathbf{F}(\omega_{G,\A}(\tau))\twoheadrightarrow\mathbf{F}(\omega_{G,\A})(\tau)
\]
be the corresponding convex projection with respect to the natural
distance $d$ on $\mathbf{F}(\omega_{G,\A}(\tau))$. For every $\mathcal{F}\in\mathbf{F}(\omega_{G,\A})$
and $f,g\in\mathbf{F}(\omega_{G,\A}(\tau))$, we have 
\[
d\left(p(f),p(g)\right)\leq d\left(f,g\right),\quad\left\Vert p(f)\right\Vert \leq\left\Vert f\right\Vert \quad\mbox{and}\quad\left\langle \mathcal{F}(\tau),f\right\rangle \leq\left\langle \mathcal{F}(\tau),p(f)\right\rangle .
\]
The first formula comes from~\cite[II.2.4]{BrHa99}. The second follows,
with $g=p(g)=0(\tau)$. The third formula can be proved as in section~\ref{subsec:ExampleProj},
see also \cite[Section 5.7.7]{Co14}.

\subsection{Quasi-Tannakian categories\label{subsec:HypOnC}}

Let now $\C$ be an essentially small $k$-linear quasi-abelian $\otimes$-category
with a faithful exact $k$-linear $\otimes$-functor $\omega_{\C,\A}:\C\rightarrow\A$
such that for every object $X$ of $\C$, $\omega_{\C,\A}$ induces
a bijection between strict subobjects of $X$ in $\C$ and (strict)
subobjects of $\omega_{\C,\A}(X)$ in $\A$. We add to this data a
degree function $\deg_{\C}:\sk\,\C\rightarrow\mathbb{R}$, i.e. a
function which is additive on short exact sequences and non-decreasing
on mono-epis. Together with the rank function 
\[
\rank_{\C}(X)\eqd\rank_{\A}(\omega_{\C,\A}(X)),
\]
it yields a Harder-Narasimhan filtration on $\C$, which we view as
a functor over $\A$, 
\[
\mathcal{F}_{HN}:\C\rightarrow\F(\A),\qquad\omega\circ\mathcal{F}_{HN}=\omega_{\C,\A}.
\]
Note that this functor $\mathcal{F}_{HN}$ is usually neither exact,
nor a $\otimes$-functor.

\subsubsection{\label{subsec:CarHNQuAb}~}

We denote by $\C(X)$ the fiber of $\omega_{\C,\A}:\C\rightarrow\A$
over an object $X$ of $\A$, and for $x\in\C(X)$, we denote by $\left\langle x,-\right\rangle :\mathbf{F}(X)\rightarrow\mathbb{R}$
the concave degree function on 
\[
\mathbf{F}(X)=\mathbf{F}(\Sub(X))=\mathbf{F}(\Sub(x))=\mathbf{F}(x)
\]
induced by our given degree function on $\C$, thereby obtaining a
pairing
\[
\left\langle -,-\right\rangle :\C(X)\times\mathbf{F}(X)\rightarrow\mathbb{R}.
\]
By proposition~\ref{prop:CarHN}, the Harder-Narasimhan filtration
$\mathcal{F}_{HN}(x)$ of $x$ is the unique element $\mathcal{F}\in\mathbf{F}(X)$
with the following equivalent properties: 
\begin{enumerate}
\item For every $f\in\mathbf{F}(X)$, $\left\Vert \mathcal{F}\right\Vert ^{2}-2\left\langle x,\mathcal{F}\right\rangle \leq\left\Vert f\right\Vert ^{2}-2\left\langle x,f\right\rangle $.
\item For every $f\in\mathbf{F}(X)$, $\left\langle x,f\right\rangle \leq\left\langle \mathcal{F},f\right\rangle $
with equality for $f=\mathcal{F}$. 
\item For every $\gamma\in\mathbb{R}$, $\Gr_{\mathcal{F}}^{\gamma}(x)$
is semi-stable of slope $\gamma$.
\end{enumerate}
In $(3)$, $\Gr_{\mathcal{F}}^{\gamma}(x)=\mathcal{F}^{\gamma}(x)/\mathcal{F}_{+}^{\gamma}(x)$
where $\mathcal{F}^{\gamma}(x)$ and $\mathcal{F}_{+}^{\gamma}(x)$
are the strict subobjects of $x$ corresponding to the (strict) subobjects
$\mathcal{F}(\gamma)$ and $\mathcal{F}_{+}(\gamma)$ of $X=\omega_{\C,\A}(x)$. 

\subsubsection{~\label{subsec:CompWithTensor}}

We denote by $\C^{\otimes}(\omega_{G,\A})$ the set of all factorizations
\[
\omega_{G,\A}:\Rep(G)\stackrel{x}{\longrightarrow}\C\stackrel{\omega_{\C,\A}}{\longrightarrow}\A
\]
of our given exact $\otimes$-functor $\omega_{G,\A}$ through a $k$-linear
exact $\otimes$-functor 
\[
x:\Rep(G)\rightarrow\C.
\]
Thus for every $\tau\in\Rep(G)$, we have an evaluation map 
\[
\C^{\otimes}(\omega_{G,\A})\rightarrow\mathbf{\C}(\omega_{G,\A}(\tau)),\qquad x\mapsto x(\tau)
\]
and the corresponding pairing 
\[
\left\langle -,-\right\rangle _{\tau}:\C^{\otimes}(\omega_{G,\A})\times\mathbf{F}(\omega_{G,\A})\rightarrow\mathbb{R},\qquad\left\langle x,\mathcal{F}\right\rangle _{\tau}=\left\langle x(\tau),\mathcal{F}(\tau)\right\rangle .
\]
Note that the latter is concave in the second variable.
\begin{prop}
For $x\in\C^{\otimes}(\omega_{G,\A})$ and any faithful representation
$\tau$ of $G$, there is a unique $\mathcal{F}$ in $\mathbf{F}(\omega_{G,\A})$
which satisfies the following equivalent conditions:

\begin{enumerate}
\item For every $f\in\mathbf{F}(\omega_{G,\A})$, $\left\Vert \mathcal{F}\right\Vert _{\tau}^{2}-2\left\langle x,\mathcal{F}\right\rangle _{\tau}\leq\left\Vert f\right\Vert _{\tau}^{2}-2\left\langle x,f\right\rangle _{\tau}$.
\item For every $f\in\mathbf{F}(\omega_{G,\A})$, $\left\langle x,f\right\rangle _{\tau}\leq\left\langle \mathcal{F},f\right\rangle _{\tau}$
with equality for $f=\mathcal{F}$.
\end{enumerate}
Suppose moreover that for every $f\in\mathbf{F}(\omega_{G,\A}(\tau))$
with projection $p(f)\in\mathbf{F}(\omega_{G,\A})(\tau)$,
\[
\left\langle x(\tau),f\right\rangle \leq\left\langle x(\tau),p(f)\right\rangle .
\]
Then $\mathcal{F}(\tau)=\mathcal{F}_{HN}(x(\tau))$.
\end{prop}
\begin{proof}
For the first claim, it is sufficient to establish the implication
$(1)\Rightarrow(2)$ for any $\mathcal{F}\in\mathbf{F}(\omega_{G,\A})$,
the existence of an $\mathcal{F}$ satisfying $(1)$, and the uniqueness
of any $\mathcal{F}$ satisfying $(2)$. The first two of these are
proved as in proposition~\ref{prop:CarHN}, replacing everywhere
the complete CAT(0)-space $\mathbf{F}(X)$ by $\mathbf{F}(\omega_{G,\A})$
and the concave function $\left\langle \star,-\right\rangle $ by
$\left\langle x,-\right\rangle _{\tau}$. As for uniqueness, if $\mathcal{F}$
and $\mathcal{G}$ both satisfy $(2)$, then 
\[
\left\Vert \mathcal{F}\right\Vert _{\tau}^{2}=\left\langle x,\mathcal{F}\right\rangle _{\tau}\leq\left\langle \mathcal{G},\mathcal{F}\right\rangle _{\tau}\quad\mbox{and}\quad\left\Vert \mathcal{G}\right\Vert _{\tau}^{2}=\left\langle x,\mathcal{G}\right\rangle _{\tau}\leq\left\langle \mathcal{F},\mathcal{G}\right\rangle _{\tau}
\]
therefore $d_{\tau}(\mathcal{F},\mathcal{G})^{2}=\left\Vert \mathcal{F}\right\Vert _{\tau}^{2}+\left\Vert \mathcal{G}\right\Vert _{\tau}^{2}-2\left\langle \mathcal{F},\mathcal{G}\right\rangle _{\tau}\leq0$
and $\mathcal{F}=\mathcal{G}$. For the last claim, 
\[
\left\Vert \mathcal{F}(\tau)\right\Vert ^{2}-2\left\langle x(\tau),\mathcal{F}(\tau)\right\rangle \leq\left\Vert p(f)\right\Vert ^{2}-2\left\langle x(\tau),p(f)\right\rangle \leq\left\Vert f\right\Vert ^{2}-2\left\langle x(\tau),f\right\rangle 
\]
for every $f\in\mathbf{F}(\omega_{G,\A}(\tau))$ by the first characterization
of $\mathcal{F}$, the assumption on $(x,\tau)$ and the inequality
$\left\Vert p(f)\right\Vert \leq\left\Vert f\right\Vert $. Thus indeed
$\mathcal{F}(\tau)=\mathcal{F}_{HN}(x(\tau))$ by \ref{subsec:CarHNQuAb}.
\end{proof}
\begin{prop}
\label{prop:CaractGood}Fix $x\in\C^{\otimes}(\omega_{G,\A})$. Suppose
that for every faithful representation $\tau$ of $G$ and every $f\in\mathbf{F}(\omega_{G,\A}(\tau))$
with projection $p(f)\in\mathbf{F}(\omega_{G,\A})(\tau)$, we have
\[
\left\langle x(\tau),f\right\rangle \leq\left\langle x(\tau),p(f)\right\rangle .
\]
Then $\mathcal{F}_{HN}(x):=\mathcal{F}_{HN}\circ x$ is an exact $\otimes$-functor
$\mathcal{F}_{HN}(x):\Rep(G)\rightarrow\F(\A)$ and for every faithful
representation $\tau$ of $G$, $\mathcal{F}_{HN}(x)$ is the unique
element $\mathcal{F}$ of $\mathbf{F}(\omega_{G,\A})$ which satisfies
the following equivalent conditions:

\begin{enumerate}
\item For every $f\in\mathbf{F}(\omega_{G,\A})$, $\left\Vert \mathcal{F}\right\Vert _{\tau}^{2}-2\left\langle x,\mathcal{F}\right\rangle _{\tau}\leq\left\Vert f\right\Vert _{\tau}^{2}-2\left\langle x,f\right\rangle _{\tau}$.
\item For every $f\in\mathbf{F}(\omega_{G,\A})$, $\left\langle x,f\right\rangle _{\tau}\leq\left\langle \mathcal{F},f\right\rangle _{\tau}$
with equality for $f=\mathcal{F}$.
\item For every $\gamma\in\mathbb{R}$, $\Gr_{\mathcal{F}(\tau)}^{\gamma}(x(\tau))$
is semi-stable of slope $\gamma$. 
\end{enumerate}
\end{prop}
\begin{proof}
By the previous proposition, for any faithful $\tau$, the three conditions
are equivalent and determine a unique $\mathcal{F}_{\tau}\in\mathbf{F}(\omega_{G,\A})$
with $\mathcal{F}_{\tau}(\tau)=\mathcal{F}_{HN}(x)(\tau)$. For any
$\sigma\in\Rep(G)$, $\tau'=\tau\oplus\sigma$ is also faithful. By
additivity of $\mathcal{F}_{\tau'}$ and $\mathcal{F}_{HN}(x)$, 
\[
\mathcal{F}_{\tau'}(\tau)\oplus\mathcal{F}_{\tau'}(\sigma)=\mathcal{F}_{\tau'}(\tau')=\mathcal{F}_{HN}(x)(\tau')=\mathcal{F}_{\tau}(\tau)\oplus\mathcal{F}_{HN}\left(x\right)(\sigma)
\]
inside $\mathbf{F}(\omega_{G,\A}(\tau))\times\mathbf{F}(\omega_{G,\A}(\sigma))\subset\mathbf{F}(\omega_{G,\A}(\tau'))$,
therefore 
\[
\mathcal{F}_{\tau}(\tau)=\mathcal{F}_{\tau'}(\tau)\quad\mbox{and}\quad\mathcal{F}_{HN}\left(x\right)(\sigma)=\mathcal{F}_{\tau'}(\sigma).
\]
Since evaluation at $\tau$ is injective, $\mathcal{F}_{\tau}=\mathcal{F}_{\tau'}$
and $\mathcal{F}_{HN}(x)(\sigma)=\mathcal{F}_{\tau}(\sigma)$ for
every $\sigma\in\Rep(G)$. In particular, $\mathcal{F}=\mathcal{F}_{\tau}$
does not depend upon $\tau$ and $\mathcal{F}_{HN}(x)=\mathcal{F}$
is indeed an exact $\otimes$-functor. This proves the proposition.
\end{proof}

\subsection{Compatibility with $\otimes$-products\label{subsec:CompatibilityTensorAxioGood}}

Let us now slightly change our set-up. We keep $k$ and $\A$ fixed,
view $\C$, $\omega_{\C,\A}:\C\rightarrow\A$ and $\deg_{\C}:\sk\,\C\rightarrow\mathbb{R}$
as auxiliary data, and we do not fix $G$ or $\omega_{G,\A}$. 

\subsubsection{~}

A faithful exact $k$-linear $\otimes$-functor $x:\Rep(G)\rightarrow\C$
is \emph{good} if it satisfies the assumption of the previous proposition,
when we view it as an element of $\C^{\otimes}(\omega_{G,\A})$ with
$\omega_{G,\A}=\omega_{\C,\A}\circ x$. Then $\mathcal{F}_{HN}(x):=\mathcal{F}_{HN}\circ x$
is an exact $k$-linear $\otimes$-functor 
\[
\mathcal{F}_{HN}(x):\Rep(G)\rightarrow\F(\A).
\]
We say that a pair of objects $(x_{1},x_{2})$ in $\C$ is \emph{good
}if the following holds. For $i\in\{1,2\}$, set $d_{i}=\rank_{\C}(x_{i})$
and let $\tau_{i}$ and $1_{i}$ be respectively the tautological
and trivial representations of $GL(d_{i})$ on $V(\tau_{i})=k^{d_{i}}$
and $V(1_{i})=k$. We require the existence of a \emph{good} exact
$k$-linear $\otimes$-functor 
\[
x:\Rep\left(GL(d_{1})\times GL(d_{2})\right)\rightarrow\C
\]
mapping $\tau'_{1}=\tau_{1}\boxtimes1_{2}$ to $x_{1}$ and $\tau'_{2}=1_{1}\boxtimes\tau_{2}$
to $x_{2}$. Then 
\[
\mathcal{F}_{HN}(x_{1}\otimes x_{2})=\mathcal{F}_{HN}(x_{1})\otimes\mathcal{F}_{HN}(x_{2}).
\]
We say that $(\C,\deg_{\C})$ is \emph{good} if every pair of objects
in $\C$ is good.
\begin{cor}
\label{cor:GoodIpliesTensor}If $(\C,\deg_{\C})$ is good, then $\mathcal{F}_{HN}:\C\rightarrow\F(\A)$
is a $\otimes$-functor.
\end{cor}

\subsubsection{~\label{subsec:FiberProductGood}}

Suppose that $(\omega_{i}:\C_{i}\rightarrow\A,\deg_{i})_{i\in I}$
is a finite collection of data as above. Let $\omega:\C\rightarrow\A$
be the fibered product of the $\omega_{i}$'s, with fiber $\C(X)=\prod\C_{i}(X)$
over any object $X$ of $\A$ and with homomorphisms given by 
\[
\Hom_{\C}((x_{i}),(y_{i}))\eqd\cap_{i}\Hom_{\C_{i}}(x_{i},y_{i})\quad\mbox{in}\quad\Hom_{\A}(X,Y)
\]
for $(x_{i})\in\C(X)$, $(y_{i})\in\C(Y)$. Then $\C$ is yet another
essentially small quasi-abelian $k$-linear $\otimes$-category equipped
with a faithful exact $k$-linear $\otimes$-functor $\omega:\C\rightarrow\A$
which identifies $\Sub((x_{i}))$ and $\Sub(X)$ for every $(x_{i})\in\C(X)$.
Fix $\lambda=(\lambda_{i})\in\mathbb{R}^{I}$ with $\lambda_{i}>0$
and for every object $x=(x_{i})$ of $\C$, set $\deg_{\lambda}(x):=\sum\lambda_{i}\deg_{i}(x_{i})$.
Then 
\[
\deg_{\lambda}:\sk\,\C\rightarrow\mathbb{R}
\]
is a degree function on $\C$ and for every $X\in\A$, $x=(x_{i})\in\C(X)$
and $\mathcal{F}\in\mathbf{F}(X)$, 
\[
\left\langle x,\mathcal{F}\right\rangle =\sum\lambda_{i}\left\langle x_{i},\mathcal{F}\right\rangle .
\]
Thus an exact $k$-linear $\otimes$-functor $x:\Rep(G)\rightarrow\C$
is good if it has good components $x_{i}:\Rep(G)\rightarrow\C_{i}$,
a pair $((x_{i}),(y_{i}))$ in $\C$ is good if it has good components
$(x_{i},y_{i})$ in $\C_{i}$, and $(\C,\deg_{\lambda})$ is good
if the $(\C_{i},\deg_{i})$'s are, in which case the Harder-Narasimhan
filtration $\mathcal{F}_{HN}:\C\rightarrow\F(\A)$ is compatible with
tensor products.

\subsubsection{~}

Our use of an auxiliary reductive group $G$ to establish the compatibility
of Harder-Narasimhan filtrations with tensor products may obscure
the main idea, which goes back to at least Totaro's \cite{To96}:
once the Harder-Narasimhan filtration has been characterized as the
(unique) solution of an optimization problem on a space of $\mathbb{R}$-filtrations,
the desired compatibility $\mathcal{F}_{HN}(x_{1}\otimes x_{2})=\mathcal{F}_{HN}(x_{1})\otimes\mathcal{F}_{HN}(x_{2})$
follows from an inequality of the form $\left\langle x_{1}\otimes x_{2},f\right\rangle \leq\left\langle x_{1}\otimes x_{2},p(f)\right\rangle $,
for every $\mathbb{R}$-filtration $f\in\mathbf{F}(x_{1}\otimes x_{2})$,
where $p$ is the convex projection of $\mathbf{F}(x_{1}\otimes x_{2})$
onto the image of the tensor product map $\otimes:\mathbf{F}(x_{1})\times\mathbf{F}(x_{2})\rightarrow\mathbf{F}(x_{1}\otimes x_{2})$.
Note that $p(f)$ is itself the (unique) solution of a different and
easier optimization problem. For a strict subobject $z$ of $x_{1}\otimes x_{2}$
mapping to some $f$ in $\mathbf{F}(x_{1}\otimes x_{2})$ under the
embedding of section~\ref{subsec:embeddingX2F(X)}, a pair of $\mathbb{R}$-filtrations
$(\mathcal{F}_{1},\mathcal{F}_{2})\in\mathbf{F}(x_{1})\times\mathbf{F}(x_{2})$
with the property that $\mathcal{F}_{1}\otimes\mathcal{F}_{2}=p(f)$
in $\mathbf{F}(x_{1}\otimes x_{2})$ is what would be called a Kempf
filtration in~\cite{To96} or \cite{LeWE16}. In our set-up, the
tensor product map is the evaluation map $\mathbf{F}(\omega_{G,\A})\rightarrow\mathbf{F}(\omega_{G,\A}(\tau))$
induced by the tensor product representation $\tau$ of $G:=GL(d_{1})\times GL(d_{2})$
(with $d_{i}=\rank(x_{i})$). It turns out that in all the examples
we know, the proofs of the desired inequalities work equally well
for arbitrary $G$ and $\tau$, and the final results thus obtained
are stronger: in addition to their compatibility with $\otimes$-products,
our Harder-Narasimhan filtrations also have some exactness properties,
a feature that usually required further arguments, most notably Haboush's
theorem \cite{Ha75}. Of course, our set-up is also tailor-made for
the applications that we have in mind.

\section{Examples of good $\C$'s\label{sec:Examples}}

\subsection{Filtered vector spaces\label{subsec:FilteredVectorSpaces}}

\subsubsection{~}

We consider the following set-up: $k$ is a field, $\ell$ is an extension
of $k$ and 
\[
\A=\Vect_{k}\quad\mbox{and}\quad\C=\Fil_{k}^{\ell}\quad\mbox{with}\quad\left\{ \begin{array}{rcl}
\omega(V,\mathcal{F}) & = & V,\\
\rank(V,\mathcal{F}) & = & \dim_{k}V,\\
\deg(V,\mathcal{F}) & = & \deg(\mathcal{F}).
\end{array}\right.
\]
Here $\Fil_{k}^{\ell}$ is the category of all pairs $(V,\mathcal{F})$
where $V$ is a finite dimensional $k$-vector space and $\mathcal{F}$
is an $\mathbb{R}$-filtration on $V_{\ell}:=V\otimes_{k}\ell$, i.e.~a
collection $\mathcal{F}=(\mathcal{F}^{\gamma})_{\gamma\in\mathbb{R}}$
of $\ell$-subspaces of $V_{\ell}$ such that $\mathcal{F}^{\gamma}\subset\mathcal{F}^{\gamma\prime}$
if $\gamma'\leq\gamma$, $\mathcal{F}^{\gamma}=V_{\ell}$ for $\gamma\ll0$,
$\mathcal{F}^{\gamma}=0$ for $\gamma\gg0$ and $\mathcal{F}^{\gamma}=\cap_{\gamma'<\gamma}\mathcal{F}^{\gamma'}$
for every $\gamma\in\mathbb{R}$. A \emph{morphism} $f:(V_{1},\mathcal{F}_{1})\rightarrow(V_{2},\mathcal{F}_{2})$
is a $k$-linear morphism $f:V_{1}\rightarrow V_{2}$ such that $f_{\ell}(\mathcal{F}_{1}^{\gamma})\subset\mathcal{F}_{2}^{\gamma}$
for every $\gamma\in\mathbb{R}$, where $f_{\ell}:V_{1,\ell}\rightarrow V_{2,\ell}$
is the $\ell$-linear extension of $f$. The kernel and cokernel of
$f$ are given by $(\ker f,\mathcal{F}_{1,\ker f})$ and $(\coker f,\mathcal{F}_{2,\coker f})$
where $\mathcal{F}_{1,\ker f}^{\gamma}$ and $\mathcal{F}_{2,\coker f}^{\gamma}$
are respectively the inverse and direct images of $\mathcal{F}_{1}^{\gamma}$
and $\mathcal{F}_{2}^{\gamma}$ under $(\ker f)_{\ell}\hookrightarrow V_{1,\ell}$
and $V_{2,\ell}\twoheadrightarrow(\coker f)_{\ell}$. The morphism
$f$ is strict if and only if $\mathcal{F}_{2}^{\gamma}\cap f_{\ell}(V_{1,\ell})=f_{\ell}(\mathcal{F}_{1}^{\gamma})$
for every $\gamma\in\mathbb{R}$. It is a mono-epi if and only if
the underlying map $f:V_{1}\rightarrow V_{2}$ is an isomorphism.
The category $\Fil_{k}^{\ell}$ is quasi-abelian, the rank and degree
functions are additive on short exact sequences, and they are respectively
constant and non-decreasing on mono-epis. More precisely if $f:(V_{1},\mathcal{F}_{1})\rightarrow(V_{2},\mathcal{F}_{2})$
is a mono-epi, then $\deg\mathcal{F}_{1}\leq\deg\mathcal{F}_{2}$
with equality if and only if $f$ is an isomorphism. We thus obtain
a HN-formalism on $\Fil_{k}^{\ell}$. There is also a tensor product,
given by 
\begin{eqnarray*}
(V_{1},\mathcal{F}_{1})\otimes(V_{2},\mathcal{F}_{2}) & \eqd & (V_{1}\otimes_{k}V_{2},\mathcal{F}_{1}\otimes\mathcal{F}_{2}),\\
\mbox{with}\quad(\mathcal{F}_{1}\otimes\mathcal{F}_{2})^{\gamma} & \eqd & \sum_{\gamma_{1}+\gamma_{2}=\gamma}\mathcal{F}_{1}^{\gamma_{1}}\otimes_{\ell}\mathcal{F}_{2}^{\gamma_{2}}.
\end{eqnarray*}
We will show that if $\ell$ is a separable extension of $k$, the
HN-filtration is compatible with $\otimes$-products. This has been
known for some time, see for instance~\cite[I.2]{DaOrRa10}, where
a counter-example is also given when $\ell$ is a finite inseparable
extension of $k$. For $k=\ell$, we simplify our notations to $\Fil_{k}:=\Fil_{k}^{k}=\F(\Vect_{k})$. 

\subsubsection{~}

Let $\mathbb{F}(G)$ be the smooth $k$-scheme denoted by $\mathbb{F}^{\mathbb{R}}(G)$
in \cite{Co14}. Thus
\[
\mathbf{F}(G,\ell)\eqd\mathbb{F}(G)(\ell)=\mathbf{F}(\omega_{G,\ell})=(\Fil_{k}^{\ell})^{\otimes}(\omega_{G,k})
\]
is the vectorial Tits building of $G_{\ell}$, where $\omega_{G,\ell}:\Rep(G)\rightarrow\Vect_{\ell}$
is the standard fiber functor. The choice of a finite dimensional
faithful representation $\tau$ of $G$ equips these buildings with
compatible complete CAT(0)-metrics $d_{\tau}$ whose induced topologies
do not depend upon the chosen $\tau$. These constructions are covariantly
functorial in $G$, compatible with products and closed immersions,
and covariantly functorial in $\ell$. We thus obtain a (strictly)
commutative diagram of functors \xyC{4pc}
\[
\xymatrix{\Red(k)\times\Ext(k)\ar[r]\sp(0,6){\mathbf{F}(-,-)} & \Top\\
\Red(G)\times\Ext(k)\ar[r]\sp(0,6){\left(\mathbf{F}(-,-),d_{\tau}\right)}\ar@{^{(}->}[u] & \CCAT\ar@{^{(}->}[u]
}
\]
where $\Red(k)$ is the category of reductive groups over $k$, $\Red(G)$
is the poset of all (closed) reductive subgroups $H$ of $G$ viewed
as a subcategory of $\Red(k)$, $\Ext(k)$ is the category of field
extensions $\ell$ of $k$, $\Top$ is the category of topological
spaces and continuous maps, and $\CCAT$ is the category of complete
CAT(0)-metric spaces and distance preserving maps. For $\tau,$ $H$
and $\ell$ as above, the commutative diagram\xyC{3pc} 
\[
\xymatrix{\left(\mathbf{F}(H,k),d_{\tau}\right)\ar@{^{(}->}[r]\ar@{^{(}->}[d] & \left(\mathbf{F}(G,k),d_{\tau}\right)\ar@{^{(}->}[d]\\
\left(\mathbf{F}(H,\ell),d_{\tau}\right)\ar@{^{(}->}[r] & \left(\mathbf{F}(G,\ell),d_{\tau}\right)
}
\]
is cartesian in $\CCAT$ since $\mathbb{F}(H)(k)=\mathbb{F}(H)(\ell)\cap\mathbb{F}(G)(k)$
inside $\mathbb{F}(G)(\ell)$. Using~\cite[II.2.4]{BrHa99}, we obtain
a usually non-commutative diagram of non-expanding retractions
\[
\xymatrix{\left(\mathbf{F}(H,k),d_{\tau}\right) & \left(\mathbf{F}(G,k),d_{\tau}\right)\ar@{->>}[l]_{p_{k}}\\
\left(\mathbf{F}(H,\ell),d_{\tau}\right)\ar@{->>}[u]^{\pi_{H}} & \left(\mathbf{F}(G,\ell),d_{\tau}\right)\ar@{->>}[u]^{\pi_{G}}\ar@{->>}[l]_{p_{\ell}}
}
\]
where each map sends a point in its source to the unique closest point
in its target. 
\begin{thm}
\label{thm:FunctFil}If $\ell$ is a separable extension of $k$,
the diagrams\xyC{3pc}
\[
\xymatrix{\mathbf{F}(H,k)\ar@{^{(}->}[d] & \mathbf{F}(G,k)\ar@{^{(}->}[d]\ar@{->>}[l]_{p_{k}}\\
\mathbf{F}(H,\ell) & \mathbf{F}(G,\ell)\ar@{->>}[l]_{p_{\ell}}
}
\quad\mbox{and}\quad\xymatrix{\mathbf{F}(H,k)\ar@{^{(}->}[r] & \mathbf{F}(G,k)\\
\mathbf{F}(H,\ell)\ar@{^{(}->}[r]\ar@{->>}[u]^{\pi_{H}} & \mathbf{F}(G,\ell)\ar@{->>}[u]^{\pi_{G}}
}
\]
are commutative, moreover $\pi_{G}$ does not depend upon $\tau$
and defines a retraction
\[
\pi:\mathbf{F}(-,\ell)\twoheadrightarrow\mathbf{F}(-,k)
\]
of the embedding $\mathbf{F}(-,k)\hookrightarrow\mathbf{F}(-,\ell)$
of functors from $\Red(k)$ to $\Top$. Finally, 
\[
\forall(x,y)\in\mathbf{F}(H,\ell)\times\mathbf{F}(G,k):\qquad\left\langle x,y\right\rangle _{\tau}\leq\left\langle x,p_{k}(y)\right\rangle _{\tau}
\]
\end{thm}
\begin{proof}
This is essentially formal. 

\emph{Commutativity of the first diagram.} We have to show that for
every $x\in\mathbf{F}(G,k)$, $y=p_{\ell}(x)$ belongs to $\mathbf{F}(H,k)\subset\mathbf{F}(H,\ell)$
\textendash{} for then indeed $y=p_{k}(x)$. Since $\mathbf{F}(H,\ell)=\mathbb{F}(H)(\ell)$
and $\mathbb{F}(H)$ is locally of finite type over $k$, there is
a finitely generated subextension $\ell'$ of $\ell/k$ such that
$y$ belongs to $\mathbb{F}(H)(\ell')=\mathbf{F}(H,\ell')$. Plainly
$y=p_{\ell'}(x)$, and we may thus assume that $\ell=\ell'$ is a
finitely generated separable extension of $\ell$. Then~\cite[V, \S $16$, $n^\circ 7$, Corollaire of Théorème $5$]{BoAl4a7}
reduces us to the following cases: $(1)$ $\ell=k(t)$ is a purely
transcendental extension of $k$ or $(2)$ $\ell$ is a separable
algebraic extension of $k$. Note that in any case, $y$ is fixed
by the automorphism group $\Gamma$ of $\ell/k$. Indeed, $\Gamma$
acts by isometries on $\mathbf{F}(G,\ell)$ and $\mathbf{F}(H,\ell)$,
thus $p_{\ell}$ is $\Gamma$-equivariant and $\Gamma$ fixes $y=p_{\ell}(x)$
since it fixes $x\in\mathbf{F}(G,k)$. This settles the following
sub-cases, where $k$ is the subfield of $\ell$ fixed by $\Gamma$:
$(1')$ $\ell=k(t)$ with $k$ infinite (where $\Gamma=PGL_{2}(k)$),
and $(2')$ $\ell$ is Galois over $k$ (where $\Gamma=\Gal(\ell/k)$).
If $\ell$ is merely algebraic and separable over $k$, let $\ell'$
be its Galois closure in a suitable algebraic extension. Then $\ell'/\ell$
and $\ell'/k$ are Galois, thus $p_{\ell}(x)=p_{\ell'}(x)=p_{k}(x)$
by $(2')$, which settles case $(2)$. Finally if $\ell=k(t)$ with
$k=\mathbb{F}_{q}$ finite, the Frobenius $\sigma(t)=t^{q}$, also
not bijective on $\ell$, still induces a distance preserving map
on $\mathbf{F}(G,\ell)$ and $\mathbf{F}(H,\ell)$. Thus $d_{\tau}(x,y)=d_{\tau}(x,\sigma y)$
since $\sigma x=x$, but then $\sigma y=y$ by definition of $y=p_{\ell}(x)$,
and $y\in\mathbf{F}(G,k)$ as desired.

\emph{Final inequality. }For $x,y\in\mathbf{F}(H,\ell)\times\mathbf{F}(G,\ell)$,
$\left\langle x,y\right\rangle _{\tau}\leq\left\langle x,p_{\ell}(y)\right\rangle _{\tau}$
by \cite[5.7.7]{Co14} and for $y\in\mathbf{F}(G,k)$, also $p_{\ell}(y)=p_{k}(y)$
by commutativity of the first diagram.

\emph{Commutativity of the second diagram. }For $x\in\mathbf{F}(H,\ell)$
and $y=\pi_{G}(x)\in\mathbf{F}(G,k)$, 
\[
d_{\tau}\left(x,y\right)\geq d_{\tau}\left(p_{\ell}(x),p_{\ell}(y)\right)=d_{\tau}\left(x,p_{k}(y)\right)
\]
since $p_{\ell}$ is non-expanding, equal to the identity on $\mathbf{F}(H,\ell)$
and to $p_{k}$ on $\mathbf{F}(G,k)$ by commutativity of the first
diagram. Since $p_{k}(y)\in\mathbf{F}(H,k)\subset\mathbf{F}(G,k)$,
it follows that $p_{k}(y)=y$ by definition of $y$. In particular
$y\in\mathbf{F}(H,k)$, thus also $y=\pi_{H}(x)$.

\emph{Independence of $\tau$ and functoriality. }Let $G_{1}$ and
$G_{2}$ be reductive groups over $k$ with faithful representations
$\tau_{1}$ and $\tau_{2}$. Set $\tau_{3}=\tau_{1}\boxplus\tau_{2}$,
a faithful representation of $G_{3}=G_{1}\times G_{2}$. Then $\mathbb{F}(G_{3})=\mathbb{F}(G_{1})\times_{k}\mathbb{F}(G_{2})$
and for every extension $m$ of $k$, 
\[
\left(\mathbf{F}(G_{3},m),d_{\tau_{3}}\right)=\left(\mathbf{F}(G_{1},m),d_{\tau_{1}}\right)\times\left(\mathbf{F}(G_{2},m),d_{\tau_{2}}\right)
\]
in $\CCAT$. This actually means that for $x_{3}=(x_{1},x_{2})$ and
$y_{3}=(y_{1},y_{2})$ in 
\[
\mathbf{F}(G_{3},m)=\mathbf{F}(G_{1},m)\times\mathbf{F}(G_{2},m)
\]
we have the usual Pythagorean formula 
\[
d_{\tau_{3}}(x_{3},y_{3})=\sqrt{d_{\tau_{1}}(x_{1},y_{1})^{2}+d_{\tau_{2}}(x_{2},y_{2})^{2}}.
\]
It immediately follows that 
\[
\left(\mathbf{F}(G_{3},\ell)\stackrel{\pi_{3}}{\twoheadrightarrow}\mathbf{F}(G_{3},k)\right)=\left(\mathbf{F}(G_{1},\ell)\times\mathbf{F}(G_{2},\ell)\stackrel{(\pi_{1},\pi_{2})}{\twoheadrightarrow}\mathbf{F}(G_{1},k)\times\mathbf{F}(G_{2},k)\right)
\]
where $\pi_{i}=\pi_{G_{i}}$ is the retraction attached to $\tau_{i}$.
Applying this to $G_{1}=G_{2}=G$ and using the commutativity of our
second diagram for the diagonal embedding $\Delta:G\hookrightarrow G\times G$,
we obtain $\Delta\circ\pi_{3}=(\pi_{1},\pi_{2})\circ\Delta$, where
$\pi_{3}$ is now the retraction $\pi_{G}$ attached to the faithful
representation $\tau_{1}\oplus\tau_{2}=\Delta^{\ast}(\tau_{3})$ of
$G$. Thus $\pi_{1}=\pi_{3}=\pi_{2}$, i.e.~$\pi_{G}$ does not depend
upon the choice of $\tau$. Using the commutativity of our second
diagram for the graph embedding $\Delta_{f}:G_{1}\hookrightarrow G_{1}\times G_{2}$
of a morphism $f:G_{1}\rightarrow G_{2}$, we similarly obtain the
functoriality of $G\mapsto\pi_{G}$.
\end{proof}

\subsubsection{~}

For $G=GL(V)$, evaluation at the tautological representation $\tau$
of $G$ on $V$ identifies $\mathbf{F}(G,-)$ with $\mathbf{F}(V\otimes_{k}-)$.
For any reductive group $G$ with a faithful representation $\tau$
on $V=V(\tau)$, the projection $p:\mathbf{F}(V)\twoheadrightarrow\mathbf{F}(G,k)$
of proposition~\ref{prop:CaractGood} becomes the projection $p_{k}:\mathbf{F}(GL(V),k)\twoheadrightarrow\mathbf{F}(G,k)$
of the previous theorem for the embedding $\tau:G\hookrightarrow GL(V)$.
Thus if $\ell$ is a separable extension of $k$, then every $x\in\mathbf{F}(G,\ell)$
is good. Similarly for every pair $x_{1}=(V_{1},\mathcal{F}_{1})$
and $x_{2}=(V_{2},\mathcal{F}_{2})$ of objects in $\Fil_{k}^{\ell}$,
$\mathbf{F}(GL(V_{1})\times GL(V_{2}),\ell)\simeq\mathbf{F}(V_{1}\otimes_{k}\ell)\times\mathbf{F}(V_{2}\otimes_{k}\ell)$
contains $(\mathcal{F}_{1},\mathcal{F}_{2})$, which implies that
then also $\left(\Fil_{k}^{\ell},\deg\right)$ is good. We thus obtain:
\begin{prop}
Suppose that $\ell$ is a separable extension of $k$. Then
\[
\mathcal{F}_{HN}:\Fil_{k}^{\ell}\rightarrow\Fil_{k}\mbox{ is a \ensuremath{\otimes}-functor}.
\]
For every $x\in\mathbf{F}(G,\ell)$, $\mathcal{F}_{HN}(x):=\mathcal{F}_{HN}\circ x$
belongs to $\mathbf{F}(G,k)$, i.e.
\[
\mathcal{F}_{HN}(x):\Rep(G)\rightarrow\Fil_{k}\mbox{\,\ is an exact \ensuremath{\otimes}-functor.}
\]
Moreover, $\mathcal{F}_{HN}(x)=\pi_{G}(x)$ in $\mathbf{F}(G,k)$.
\end{prop}
\begin{proof}
The last assertion follows either from~proposition~\ref{prop:CaractGood}
(both $\mathcal{F}_{HN}(x)$ and $\pi_{G}(x)$ minimize $f\mapsto d_{\tau}(x,f)^{2}=\left\Vert x\right\Vert _{\tau}^{2}+\left\Vert f\right\Vert _{\tau}^{2}-2\left\langle x,f\right\rangle _{\tau}$
on $\mathbf{F}(G,k)$) or from the functoriality of $\pi_{G}$ (for
every $\sigma\in\Rep(G)$, $\pi_{G}(x)(\sigma)=\mathcal{F}_{HN}(x)(\sigma)$
by \ref{subsec:ExampleProj}).
\end{proof}
\noindent Once we know that the projection $\pi_{G}:\mathbf{F}(G,\ell)\twoheadrightarrow\mathbf{F}(G,k)$
computes the Harder-Narasimhan filtrations, the compatibility of the
latter with tensor product constructions also directly follows from
the functoriality of $G\mapsto\pi_{G}$:
\begin{prop}
The Harder-Narasimhan functor $\mathcal{F}_{HN}:\Fil_{k}^{\ell}\rightarrow\Fil_{k}$
is compatible with tensor products, symmetric and exterior powers,
and duals.
\end{prop}
\begin{proof}
Apply the functoriality of $G\mapsto\pi_{G}$ to $GL(V_{1})\times GL(V_{2})\rightarrow GL(V_{1}\otimes V_{2})$,
$GL(V)\rightarrow GL(\Sym^{r}V)$, $GL(V)\rightarrow GL(\Lambda^{r}V)$
and $GL(V)\rightarrow GL(V^{\ast})$.
\end{proof}

\subsection{Normed vector spaces\label{subsec:NormedVectorSpaces}}

\subsubsection{~}

Let $K$ be a field with a non-archimedean absolute value $\left|-\right|:K\rightarrow\mathbb{R}_{+}$
whose valuation ring $\mathcal{O}=\left\{ x\in K:\left|x\right|\leq1\right\} $
is Henselian with residue field $\ell$. A $K$\emph{-norm} on a finite
dimensional $K$-vector space $\mathcal{V}$ is a function $\alpha:\mathcal{V}\rightarrow\mathbb{R}_{+}$
such that $\alpha(v)=0\Leftrightarrow v=0$, $\alpha(v_{1}+v_{2})\leq\max\left\{ \alpha(v_{1}),\alpha(v_{2})\right\} $
and $\alpha(\lambda v)=\left|\lambda\right|\alpha(v)$ for every $v,v_{1},v_{2}\in\mathcal{V}$
and $\lambda\in K$. It is \emph{splittable} if and only if there
exists a $K$-basis $\underline{e}=(e_{1},\cdots,e_{r})$ of $\mathcal{V}$
such that $\alpha(v)=\max\left\{ \left|\lambda_{i}\right|\alpha(e_{i})\right\} $
for all $v=\sum\lambda_{i}e_{i}$ in $\mathcal{V}$; we then say that
$\alpha$ and $\underline{e}$ are \emph{adapted, }or that $\underline{e}$
is an orthogonal basis of $(\mathcal{V},\alpha)$. We denote by $\mathbf{B}(\mathcal{V})$
the set of all splittable $K$-norms on $\mathcal{V}$: it is the
extended Bruhat-Tits building of $GL(\mathcal{V})$. If $K$ is locally
compact, then every $K$-norm is splittable \cite[Proposition 1.1]{GoIw63}.
Given two splittable $K$-norms $\alpha$ and $\beta$ on $\mathcal{V}$,
there is a $K$-basis $\underline{e}$ of $\mathcal{V}$ which is
adapted to both~(\cite[Appendice]{BrTi84b} or~\cite{Pa99}), we
may furthermore assume that $\lambda_{i}=\log\alpha(e_{i})-\log\beta(e_{i})$
is non-increasing, and then \cite[6.1 \& 5.2.8]{Co14} 
\[
\mathbf{d}(\alpha,\beta)\eqd\left(\lambda_{1},\cdots,\lambda_{r}\right)\in\mathbb{R}_{\geq}^{r}\quad\mbox{and}\quad\nu(\alpha,\beta)\eqd\lambda_{1}+\cdots+\lambda_{r}\in\mathbb{R}
\]
do not depend upon the chosen adapted basis $\underline{e}$ of $\mathcal{V}$.
The functions 
\[
\mathbf{d}:\mathbf{B}(\mathcal{V})\times\mathbf{B}(\mathcal{V})\rightarrow\mathbb{R}_{\geq}^{r}\quad\mbox{and}\quad\nu:\mathbf{B}(\mathcal{V})\times\mathbf{B}(\mathcal{V})\rightarrow\mathbb{R}
\]
satisfy the following properties \cite[6.1 \& 5.2.8]{Co14}: for every
$\alpha,\beta,\gamma\in\mathbf{B}(\mathcal{V})$, 
\[
\mathbf{d}(\alpha,\gamma)\leq\mathbf{d}(\alpha,\beta)+\mathbf{d}(\beta,\gamma)\quad\mbox{and}\quad\nu(\alpha,\gamma)=\nu(\alpha,\beta)+\nu(\beta,\gamma)
\]
where the inequality is with respect to the usual dominance order
on the convex cone $\mathbb{R}_{\geq}^{r}$. A splittable $K$-norm
$\alpha$ on $\mathcal{V}$ induces a splittable $K$-norm $\alpha_{\mathcal{X}}$
on every subquotient $\mathcal{X}=\mathcal{Y}/\mathcal{Z}$ of $\mathcal{V}$,
given by the following formula: for every $x\in\mathcal{X}$, 
\[
\alpha_{\mathcal{X}}(x)\eqd\inf\left\{ \alpha(y):\mathcal{Y}\ni y\mapsto x\in\mathcal{X}\right\} =\min\left\{ \alpha(y):\mathcal{Y}\ni y\mapsto x\in\mathcal{X}\right\} .
\]
For a $K$-subspace $\mathcal{W}$ of $\mathcal{V}$ and any $\alpha,\beta\in\mathbf{B}(\mathcal{V})$,
we then have \cite[6.3.3 \& 5.2.10]{Co14}
\begin{eqnarray*}
\mathbf{d}(\alpha,\beta) & \geq & \mathbf{d}(\alpha_{\mathcal{W}},\beta_{\mathcal{W}})\ast\mathbf{d}(\alpha_{\mathcal{V}/\mathcal{W}},\beta_{\mathcal{V}/\mathcal{W}})\\
\mbox{and}\quad\nu(\alpha,\beta) & = & \nu(\alpha_{\mathcal{W}},\beta_{\mathcal{W}})+\nu(\alpha_{\mathcal{V}/\mathcal{W}},\gamma_{\mathcal{V}/\mathcal{W}})
\end{eqnarray*}
where the $\ast$-operation just re-orders the components.

\subsubsection{~}

We denote by $\Norm_{K}$ the quasi-abelian $\otimes$-category of
pairs $(\mathcal{V},\alpha)$ where $\mathcal{V}$ is a finite dimensional
$K$-vector space and $\alpha$ is a splittable $K$-norm on $\mathcal{V}$
\cite[6.4]{Co14}. A \emph{morphism} $f:(\mathcal{V}_{1},\alpha_{1})\rightarrow(\mathcal{V}_{2},\alpha_{2})$
is a $K$-linear morphism $f:\mathcal{V}_{1}\rightarrow\mathcal{V}_{2}$
such that $\alpha_{2}(f(x))\leq\alpha_{1}(x)$ for every $x\in\mathcal{V}_{1}$.
Its kernel and cokernels are given by $(\ker(f),\alpha_{1,\ker(f)})$
and $(\coker(f),\alpha_{2,\coker(f)})$. The morphism is strict if
and only if 
\[
\alpha_{2}(y)=\inf\left\{ \alpha_{1}(x):f(x)=y\right\} =\min\left\{ \alpha_{1}(x):f(x)=y\right\} 
\]
for every $y\in f(\mathcal{V}_{1})$. It is a mono-epi if and only
if $f:\mathcal{V}_{1}\rightarrow\mathcal{V}_{2}$ is an isomorphism,
in which case $\nu(f_{\ast}(\alpha_{1}),\alpha_{2})\geq0$ with equality
if and only if $f$ is an isomorphism in $\Norm_{K}$, where $f_{\ast}(\alpha_{1})$
is the splittable $K$-norm on $\mathcal{V}_{2}$ with $f_{\ast}(\alpha_{1})(f(x))=\alpha_{1}(x)$.
The tensor product of $\Norm_{K}$ is given by the formula
\[
(\mathcal{V}_{1},\alpha_{1})\otimes(\mathcal{V}_{2},\alpha_{2})\eqd(\mathcal{V}_{1}\otimes_{K}\mathcal{V}_{2},\alpha_{1}\otimes\mathcal{\alpha}_{2})
\]
where for every $v\in\mathcal{V}_{1}\otimes_{K}\mathcal{V}_{2}$,
\[
(\alpha_{1}\otimes\alpha_{2})(v)\eqd\min\left\{ \max\left\{ \alpha_{1}(v_{1,i})\alpha_{2}(v_{2,i}):i\right\} \left|\begin{array}{l}
v={\textstyle \sum_{i}}v_{1,i}\otimes v_{2,i}\\
v_{1,i}\in\mathcal{V}_{1},\,v_{2,i}\in\mathcal{V}_{2}
\end{array}\right.\right\} .
\]
This formula indeed defines a splittable $K$-norm on $\mathcal{V}_{1}\otimes\mathcal{V}_{2}$
by \cite[1.11]{BrTi84b}.

\subsubsection{~}

A \emph{lattice}\footnote{\emph{Not to be confused with the eponymous notion from section~\ref{subsec:definitions4lattices}}}
in $\mathcal{V}$ is a finitely generated $\mathcal{O}$-submodule
$L$ of $\mathcal{V}$ which spans $\mathcal{V}$ over $K$. Any such
lattice is actually finite and free over $\mathcal{O}$. The \emph{gauge
norm} of $L$ is the splittable $K$-norm $\alpha_{L}:\mathcal{V}\rightarrow\mathbb{R}_{+}$
defined by 
\[
\alpha_{L}(v)\eqd\inf\left\{ \left|\lambda\right|:v\in\lambda L\right\} .
\]
This construction defines a faithful exact $\mathcal{O}$-linear $\otimes$-functor
\[
\alpha_{-}:\Bun_{\mathcal{O}}\rightarrow\Norm_{K}
\]
where $\Bun_{\mathcal{O}}$ is the quasi-abelian $\mathcal{O}$-linear
$\otimes$-category of finite free $\mathcal{O}$-modules. A normed
$K$-vector space $(\mathcal{V},\alpha)$ belongs to the essential
image of this functor if and only if $\alpha(\mathcal{V})\subset\left|K\right|$.
This essential image is stable under strict subobjects and quotients,
and the functor is an equivalence of categories if $\left|K\right|=\mathbb{R}_{+}$. 

\subsubsection{~}

Suppose that $k$ is a subfield of $\mathcal{O}$. Thus $\left|k^{\times}\right|=1$
and $\ell$ is an extension of $k$. We denote by $\Norm_{k}^{K}$
the quasi-abelian $k$-linear $\otimes$-category of pairs $(V,\alpha)$
where $V$ is a finite dimensional $k$-vector space and $\alpha$
is a splittable $K$-norm on $V_{K}:=V\otimes_{k}K$. A \emph{morphism}
$f:(V_{1},\alpha_{1})\rightarrow(V_{2},\alpha_{2})$ is a $k$-linear
morphism $f:V_{1}\rightarrow V_{2}$ inducing a morphism $f_{K}:(V_{1,K},\alpha_{1})\rightarrow(V_{2,K},\alpha_{2})$
in $\Norm_{K}$. Its kernel and cokernel are given by the obvious
formulas, the morphism is strict if and only if $f_{K}$ is so, it
is a mono-epi if and only if $f:V_{1}\rightarrow V_{2}$ is an isomorphism,
in which case $\nu(f_{K,\ast}(\alpha_{1}),\alpha_{2})\geq0$ with
equality if and only if $f$ is an isomorphism in $\Norm_{k}^{K}$.
The tensor product in $\Norm_{k}^{K}$ is given by $(V_{1},\alpha_{1})\otimes(V_{2},\alpha_{2}):=(V_{1}\otimes V_{2},\alpha_{1}\otimes\alpha_{2})$
and the forgetful functor $\omega:\Norm_{k}^{K}\rightarrow\Vect_{k}$
is a faithful exact $k$-linear $\otimes$-functor which identifies
the poset $\Sub(V,\alpha)$ of strict subobjects of $(V,\alpha)$
in $\Norm_{k}^{K}$ with the poset $\Sub(V)$ of $k$-subspaces of
$V=\omega(V,\alpha)$. In addition, there are two exact $\otimes$-functors
\[
\Norm_{k}^{K}\rightarrow\Norm_{K},\quad(V,\alpha)\mapsto(V_{K},\alpha)\mbox{ or }(V_{K},\alpha_{V\otimes\mathcal{O}})
\]
where $V\otimes\mathcal{O}=V\otimes_{k}\mathcal{O}$ is the standard
$\mathcal{O}$-lattice in $V_{K}=V\otimes_{k}K$. We set 
\[
\rank(V,\alpha)\eqd\dim_{k}V\quad\mbox{and}\quad\deg(V,\alpha)\eqd\nu(\alpha_{V\otimes\mathcal{O}},\alpha).
\]
These functions are both plainly additive on short exact sequences
and respectively constant and non-decreasing on mono-epis. More precisely,
if $f:(V_{1},\alpha_{1})\rightarrow(V_{2},\alpha_{2})$ is a mono-epi,
then $f:V_{1}\rightarrow V_{2}$ is an isomorphism, $f_{K,\ast}(\alpha_{V_{1}\otimes\mathcal{O}})=\alpha_{V_{2}\otimes\mathcal{O}}$
and
\[
\begin{array}{rclcl}
\deg(V_{1},\alpha_{1}) & = & \nu(\alpha_{V_{1}\otimes\mathcal{O}},\alpha_{1})\\
 & = & \nu(\alpha_{V_{2}\otimes\mathcal{O}},f_{K,\ast}(\alpha_{1}))\\
 & = & \nu(\alpha_{V_{2}\otimes\mathcal{O}},\alpha_{2})-\nu(f_{K,\ast}(\alpha_{1}),\alpha_{2}) & \leq & \deg(V_{2},\alpha_{2})
\end{array}
\]
with equality if and only if $f$ is an isomorphism in $\Norm_{k}^{K}$. 

\subsubsection{~\label{subsec:setup4Norms}}

We may thus consider the following set-up
\[
\A=\Vect_{k}\quad\mbox{and}\quad\C=\Norm_{k}^{K}\quad\mbox{with}\quad\left\{ \begin{array}{rcl}
\omega(V,\alpha) & = & V,\\
\rank(V,\alpha) & = & \dim_{k}V,\\
\deg(V,\alpha) & = & \nu(\alpha_{V\otimes\mathcal{O}},\alpha),
\end{array}\right.
\]
giving rise to a HN-formalism on $\Norm_{k}^{K}$, with HN-filtration
\[
\mathcal{F}_{HN}:\Norm_{k}^{K}\rightarrow\Fil_{k}.
\]
We will show that if $\ell$ is a separable extension of $k$, then
for any reductive group $G$ over $k$, sufficiently many $\alpha$'s
in $(\Norm_{k}^{K})^{\otimes}(\omega_{G,k})$ are good for the pair
$(\Norm_{k}^{K},\deg)$ itself to be good. In particular, $\mathcal{F}_{HN}$
is then a $\otimes$-functor.

\subsubsection{~}

\emph{A variant}. Let $\Bun_{k}^{K}$ be the category of pairs $(V,L)$
where $V$ is a finite dimensional $k$-vector space and $L$ is an
$\mathcal{O}$-lattice in $V_{K}$. With the obvious morphisms and
tensor products, this is yet another quasi-abelian $k$-linear $\otimes$-category,
and the $k$-linear exact $\otimes$-functor $(V,L)\mapsto(V,\alpha_{L})$
identifies $\Bun_{k}^{K}$ with a full subcategory of $\Norm_{k}^{K}$,
made of those $(V,\alpha)$ such that $\alpha(V_{K})\subset\left|K\right|$,
which is stable under strict subobjects and quotients. The above rank
and degree functions on $\Norm_{k}^{K}$ therefore induce a HN-formalism
on $\Bun_{k}^{K}$ whose corresponding HN-filtration 
\[
\mathcal{F}_{HN}:\Bun_{k}^{K}\rightarrow\Fil_{k}
\]
is a $\otimes$-functor if $\ell$ is a separable extension of $k$.
Note that
\[
\deg(V,L)=\sum_{i=1}^{r}\log\left|\lambda_{i}\right|\quad\mbox{if}\quad V\otimes_{k}\mathcal{O}=\oplus_{i=1}^{r}\mathcal{O}e_{i}\mbox{ and }L=\oplus_{i=1}^{r}\mathcal{O}\lambda_{i}e_{i}.
\]
If $K$ is discretely valued, it is convenient to either normalize
its valuation so that $\log\left|K^{\times}\right|=\mathbb{Z}$, or
to renormalize the degree function on $\Norm_{k}^{K}$, so that its
restriction to $\Bun_{k}^{K}$ takes values in $\mathbb{Z}$. The
HN-filtration on $\Bun_{k}^{K}$ is then a $\mathbb{Q}$-filtration.

\subsubsection{~}

For a reductive group $G$ over $\mathcal{O}$, let $\mathbf{B}^{e}(G_{K})$
be the extended Bruhat-Tits building of $G_{K}$. There is a canonical
injective and functorial map \cite[Theorem 132]{Co14} 
\[
\boldsymbol{\alpha}:\mathbf{B}^{e}(G_{K})\hookrightarrow\Norm_{K}^{\otimes}(\omega_{G,K})
\]
from the building $\mathbf{B}^{e}(G_{K})$ to the set $\Norm_{K}^{\otimes}(\omega_{G,K})$
of all factorizations 
\[
\omega_{G,K}:\Rep(G)\stackrel{\alpha}{\longrightarrow}\Norm_{K}\stackrel{\omega}{\longrightarrow}\Vect_{K}
\]
of the standard fiber functor $\omega_{G,K}:\Rep(G)\rightarrow\Vect_{K}$
through an exact $\otimes$-functor 
\[
\alpha:\Rep(G)\longrightarrow\Norm_{K}.
\]
Here $\Rep(G)$ is the quasi-abelian $\otimes$-category of algebraic
representations of $G$ on finite free $\mathcal{O}$-modules. We
shall refer to $\alpha$ as a \emph{$K$-norm on} $\omega_{G,K}$. 

\subsubsection{~}

For a reductive group $G$ over $k$, we set $\mathbf{B}^{e}(G,K)=\mathbf{B}^{e}(G_{K})$.
Pre-composition with the base change functor $\Rep(G)\rightarrow\Rep(G_{\mathcal{O}})$
then yields a map 
\[
\Norm_{K}^{\otimes}(\omega_{G_{\mathcal{O}},K})\rightarrow\Norm_{K}^{\otimes}(\omega_{G,K})
\]
which is injective: a $K$-norm on $\omega_{G_{\mathcal{O}},K}$ is
uniquely determined by its values on arbitrary large finite free subrepresentations
of the representation of $G_{\mathcal{O}}$ on its ring of regular
functions $\mathcal{A}(G_{\mathcal{O}})=\mathcal{A}(G)\otimes_{k}\mathcal{O}$
\cite[6.4.17]{Co14}, and those coming from finite dimensional subrepresentations
of $\mathcal{A}(G)$ form a cofinal system. Note that
\[
\Norm_{K}^{\otimes}(\omega_{G,K})=(\Norm_{k}^{K})^{\otimes}(\omega_{G,k}).
\]
We thus obtain a canonical, functorial injective map
\[
\boldsymbol{\alpha}:\mathbf{B}^{e}(G,K)\hookrightarrow(\Norm_{k}^{K})^{\otimes}(\omega_{G,k}).
\]
We will show that if $\ell$ is a separable extension of $k$, then
any $\alpha$ in 
\[
\mathbf{B}(\omega_{G},K)=\boldsymbol{\alpha}(\mathbf{B}^{e}(G,K))\subset(\Norm_{k}^{K})^{\otimes}(\omega_{G,k})
\]
is good in the sense of section~\ref{subsec:CompatibilityTensorAxioGood}. 

\subsubsection{~}

For a reductive group $G$ over $k$, the extended Bruhat-Tits building
$\mathbf{B}^{e}(G,K)$ of $G_{K}$ is equipped with an an action of
$G(K)$, a $G(K)$-equivariant addition map 
\[
+:\mathbf{B}^{e}(G,K)\times\mathbf{F}(G,K)\twoheadrightarrow\mathbf{B}^{e}(G,K),
\]
a distinguished point $\circ$ fixed by $G(\mathcal{O})$, and the
corresponding localization map 
\[
\loc:\mathbf{B}^{e}(G,K)\twoheadrightarrow\mathbf{F}(G,\ell).
\]
For $f\in\mathbf{F}(G,k)\subset\mathbf{F}(G,K)$, $\loc(\circ+f)=f$
in $\mathbf{F}(G,k)\subset\mathbf{F}(G,\ell)$, i.e.~
\[
\xymatrix{\mathbf{F}(G,k)\ar[r]^{\circ+} & \mathbf{B}^{e}(G,K)\ar[r]^{\loc} & \mathbf{F}(G,\ell)}
\]
is the base change map $\mathbf{F}(G,k)\hookrightarrow\mathbf{F}(G,\ell)$.
For $G=GL(V)$, the composition 
\[
\mathbf{B}^{e}(G,K)\stackrel{\boldsymbol{\alpha}}{\longrightarrow}\mathbf{B}(\omega_{G},K)\stackrel{\mathrm{ev}}{\longrightarrow}\Norm_{k}^{K}(V)=\mathbf{B}(V_{K})
\]
of the isomorphism $\boldsymbol{\alpha}$ with evaluation at the tautological
representation of $G$ on $V$ is a bijection from $\mathbf{B}^{e}(G,K)$
to the set $\mathbf{B}(V_{K})$ of all splittable $K$-norms on $V_{K}$.
The distinguished point is the gauge norm of $V\otimes\mathcal{O}$,
the addition map is given by 
\[
(\alpha+\mathcal{F})(v)\eqd\min\left\{ \max\left\{ e^{-\gamma}\alpha(v_{\gamma}):\gamma\in\mathbb{R}\right\} :v=\sum v_{\gamma},\,v_{\gamma}\in\mathcal{F}^{\gamma}\right\} ,
\]
and the localization map $\loc:\mathbf{B}(V_{K})\rightarrow\mathbf{F}(V_{\ell})$
sends $\alpha$ to the $\mathbb{R}$-filtration 
\[
\loc(\alpha)^{\gamma}\eqd\frac{\left\{ v\in V\otimes\mathcal{O}:\alpha(v)\leq e^{-\gamma}\right\} +V\otimes\mathfrak{m}}{V\otimes\mathfrak{m}}\subseteq V_{\ell}=\frac{V\otimes\mathcal{O}}{V\otimes\mathfrak{m}}
\]
where $\mathfrak{m}=\left\{ \lambda\in K:\left|\lambda\right|<1\right\} $
is the maximal ideal of $\mathcal{O}$. For a general reductive group
$G$ over $k$, the corresponding addition map, distinguished point
and localization map on $\mathbf{B}(\omega_{G},K)$ are given by the
following formulas: for $\tau\in\Rep(G)$, 
\[
(\alpha+\mathcal{F})(\tau)\eqd\alpha(\tau)+\mathcal{F}(\tau),\quad\boldsymbol{\alpha}(\circ)(\tau)\eqd\alpha_{\omega_{G}(\tau)\otimes\mathcal{O}}\quad\mbox{and}\quad\loc(\alpha)(\tau)\eqd\loc(\alpha(\tau)).
\]

\begin{lem}
\label{lem:B(omega)isBig}If $\mathcal{O}$ is strictly Henselian,
then $\mathbf{B}(\omega_{G},K)$ contains the image of 
\[
\Bun_{\mathcal{O}}^{\otimes}(\omega_{G,K})\hookrightarrow\Norm_{K}^{\otimes}(\omega_{G,K}).
\]
If moreover $\left|K\right|=\mathbb{R}_{+}$, then $\boldsymbol{\alpha}:\mathbf{B}^{e}(G,K)\rightarrow\Norm_{K}^{\otimes}(\omega_{G,K})$
is a bijection.
\end{lem}
\begin{proof}
Plainly $\omega_{G,\mathcal{O}}\in\Bun_{\mathcal{O}}^{\otimes}(\omega_{G,K})$
maps to $\boldsymbol{\alpha}(\circ)\in\mathbf{B}(\omega_{G},K)$,
and since all of our maps are equivariant under $G(K)=\Aut^{\otimes}(\omega_{G,K})$,
it is sufficient to establish that $G(K)$ acts transitively on $\Bun_{\mathcal{O}}^{\otimes}(\omega_{G,K})$.
Any $L\in\Bun_{\mathcal{O}}^{\otimes}(\omega_{G,K})$ is a faithful
exact $\otimes$-functor $L:\Rep(G)\rightarrow\Bun_{\mathcal{O}}$.
The groupoid of all such functors is equivalent to the groupoid of
all $G$-bundles over $\Spec(\mathcal{O})$, and the latter are classified
by the étale cohomology group $H_{et}^{1}(\Spec(\mathcal{O}),G)$,
which is isomorphic to $H_{et}^{1}(\Spec(\ell),G)$ by~\cite[XXIV 8.1]{SGA3.3r},
which is trivial since $\ell$ is separably closed. It follows that
all $L$'s are isomorphic, i.e.~indeed conjugated under $G(K)=\Aut^{\otimes}(\omega_{G,K})$.
If also $\left|K\right|=\mathbb{R}_{+}$, then $\Bun_{\mathcal{O}}\rightarrow\Norm_{K}$
is an equivalence of categories, $\Bun_{\mathcal{O}}^{\otimes}(\omega_{G,K})\rightarrow\Norm_{K}^{\otimes}(\omega_{G,K})$
is a bijection, and thus $\mathbf{B}(\omega_{G},K)=\Norm_{K}^{\otimes}(\omega_{G,K})$. 
\end{proof}

\subsubsection{~}

The choice of a faithful representation $\tau$ of $G$ yields a distance
$d_{\tau}$ on $\mathbf{B}^{e}(G,K)$ \cite[5.2.9]{Co14}, defined
by $d_{\tau}(x,y):=\left\Vert \mathcal{F}\right\Vert _{\tau}$ if
$y=x+\mathcal{F}$ in $\mathbf{B}^{e}(G,K)$, where $\left\Vert -\right\Vert _{\tau}:\mathbf{F}(G,K)\rightarrow\mathbb{R}_{+}$
is the length function attached to $\tau$. The resulting metric space
is CAT(0) \cite[Lemma 112]{Co14}, complete when $(K,\left|-\right|)$
is discrete~\cite[Lemma 114]{Co14}, the addition map is non-expanding
in both variables \cite[5.2.8]{Co14}, the localization map is non-expanding~\cite[6.4.13 \& 5.5.9]{Co14},
and the induced topology on $\mathbf{B}^{e}(G,K)$ does not depend
upon the chosen $\tau$. These constructions are covariantly functorial
in $G$, compatible with products and embeddings, and covariantly
functorial in $(K,\left|-\right|)$. In particular, we thus obtain
a (strictly) commutative diagram of functors \xyC{4pc}
\[
\xymatrix{\Red(k)\times\HV(k)\ar[r]\sp(0,6){\mathbf{B}^{e}(-,-)} & \Top\\
\Red(G)\times\HV(k)\ar[r]\sp(0,6){\left(\mathbf{B}^{e}(-,-),d_{\tau}\right)}\ar@{^{(}->}[u] & \CAT\ar@{^{(}->}[u]
}
\]
where $\HV(k)$ is the category of Henselian valued extensions $(K,\left|-\right|)$
of $k$ and $\CAT$ is the category of CAT(0) metric spaces with distance
preserving maps. 

\subsubsection{~}

For a closed subgroup $H$ of $G$, the commutative diagram of CAT(0)-spaces\xyC{3pc}
\[
\xymatrix{\left(\mathbf{F}(H,k),d_{\tau}\right)\ar@{^{(}->}[r]\ar@{^{(}->}[d] & \left(\mathbf{F}(G,k),d_{\tau}\right)\ar@{^{(}->}[d]\\
\left(\mathbf{B}^{e}(H,K),d_{\tau}\right)\ar@{^{(}->}[r] & \left(\mathbf{B}^{e}(G,K),d_{\tau}\right)
}
\]
is cartesian: for $\mathcal{F}\in\mathbf{F}(G,k)$ such that $\circ+\mathcal{F}\in\mathbf{B}^{e}(H,K)$,
$\loc(\circ+\mathcal{F})=\mathcal{F}$ belongs to $\mathbf{F}(H,\ell)$,
thus $\mathcal{F}$ belongs to $\mathbf{F}(H,k)=\mathbf{F}(G,k)\cap\mathbf{F}(H,\ell)$.
The corresponding (a priori non-commutative) diagram of non-expanding
retractions
\[
\xymatrix{\left(\mathbf{F}(H,k),d_{\tau}\right) & \left(\mathbf{F}(G,k),d_{\tau}\right)\ar@{->>}[l]_{p_{k}}\\
\left(\mathbf{B}^{e}(H,K),d_{\tau}\right)\ar@{->>}[u]^{\varpi_{H}} & \left(\mathbf{B}^{e}(G,K),d_{\tau}\right)\ar@{->>}[u]^{\varpi_{G}}\ar@{.>>}[l]_{p_{K}}
}
\]
has a caveat: since $\left(\mathbf{B}^{e}(H,K),d_{\tau}\right)$ may
not be complete (and $\mathbf{B}^{e}(H,K)$ perhaps not even closed
in $\mathbf{B}^{e}(G,K)$), we can not directly appeal to \cite[II.2.4]{BrHa99},
but its proof shows that a non-expanding retraction $p_{K}$ is at
least well-defined on the subset 
\[
\mathbf{B}^{e}(G,K)'\eqd\left\{ x\in\mathbf{B}^{e}(G,K)\left|\begin{array}{c}
\exists y\in\mathbf{B}^{e}(H,K)\mbox{ such that }\\
d_{\tau}(x,y)=\inf\left\{ d_{\tau}(x,y'):y'\in\mathbf{B}^{e}(H,K)\right\} 
\end{array}\right.\right\} .
\]
Of course $\mathbf{B}^{e}(H,K)\subset\mathbf{B}^{e}(G,K)'$ and $\mathbf{B}^{e}(G,K)'=\mathbf{B}^{e}(G,K)$
if $\mathbf{B}^{e}(H,K)$ is complete, for instance if $H$ is a torus
or if $(K,\left|-\right|)$ is discrete \cite[5.3.2]{Co14}.
\begin{thm}
\label{thm:FunctFakeHN}If $\ell$ is a separable extension of $k$,
then 
\[
\mathbf{B}^{e}(G,K)'\mbox{ contains }\circ+\mathbf{F}(G,k).
\]
Moreover, the diagrams \xyC{3pc}
\[
\xymatrix{\mathbf{F}(H,k)\ar@{^{(}->}[d] & \mathbf{F}(G,k)\ar@{^{(}->}[d]\ar@{->>}[l]_{p_{k}}\\
\mathbf{B}^{e}(H,K) & \mathbf{B}^{e}(G,K)'\ar@{->>}[l]_{p_{K}}
}
\quad\mbox{and}\quad\xymatrix{\mathbf{F}(H,k)\ar@{^{(}->}[r] & \mathbf{F}(G,k)\\
\mathbf{B}^{e}(H,K)\ar@{^{(}->}[r]\ar@{->>}[u]^{\varpi_{H}} & \mathbf{B}^{e}(G,K)\ar@{->>}[u]^{\varpi_{G}}
}
\]
are commutative, $\varpi_{G}$ does not depend upon $\tau$ and defines
a retraction
\[
\varpi:\mathbf{B}^{e}(-,K)\twoheadrightarrow\mathbf{F}(-,k)
\]
of the embedding $\mathbf{F}(-,k)\hookrightarrow\mathbf{B}^{e}(-,K)$
of functors from $\Red(k)$ to $\Top$. 
\end{thm}
\begin{proof}
This is again essentially formal.

\emph{First claim and commutativity of the first diagram. }For $\mathcal{F}\in\mathbf{F}(G,k)$
and any element $y\in\mathbf{B}^{e}(H,K)$, 
\[
d_{\tau}(\circ+\mathcal{F},y)\geq d_{\tau}\left(\mathcal{F},\loc(y)\right)\geq d_{\tau}\left(\mathcal{F},p_{\ell}(\mathcal{F})\right)=d_{\tau}\left(\mathcal{F},p_{k}(\mathcal{F})\right)
\]
since $\loc$ is non-expanding and $p_{\ell}=p_{k}$ on $\mathbf{F}(G,k)$
by theorem~\ref{thm:FunctFil}, therefore 
\[
d_{\tau}\left(\mathcal{F},p_{k}(\mathcal{F})\right)=d_{\tau}\left(\circ+\mathcal{F},\circ+p_{k}(\mathcal{F})\right)=\inf\left\{ d_{\tau}(\circ+\mathcal{F},y):y\in\mathbf{B}^{e}(H,K)\right\} .
\]
This says that $\circ+\mathcal{F}\in\mathbf{B}^{e}(G,K)'$ with $p_{K}(\circ+\mathcal{F})=\circ+p_{k}(\mathcal{F})$. 

\emph{Commutativity of the second diagram. }For $x\in\mathbf{B}^{e}(H,K)$
and $\mathcal{F}:=\varpi_{G}(x)$ in $\mathbf{F}(G,k)$, $x$ and
$\circ+\mathcal{F}$ belong to $\mathbf{B}^{e}(G,K)'$, moreover 
\[
d_{\tau}(x,\circ+\mathcal{F})\geq d_{\tau}\left(p_{K}(x),p_{K}(\circ+\mathcal{F})\right)=d_{\tau}\left(x,\circ+p_{k}(\mathcal{F})\right)
\]
by commutativity of the first diagram, thus $\mathcal{F}=p_{k}(\mathcal{F})$
by definition of $\mathcal{F}=\omega_{G}(x)$, in particular $\mathcal{F}$
belongs to $\mathbf{F}(H,k)$, from which easily follows that also
$\mathcal{F}=\varpi_{H}(x)$.

\emph{Independence of $\tau$ and functoriality. }Let $G_{1}$ and
$G_{2}$ be reductive groups over $k$ with faithful representations
$\tau_{1}$ and $\tau_{2}$. Set $\tau_{3}:=\tau_{1}\boxplus\tau_{2}$,
a faithful representation of $G_{3}:=G_{1}\times G_{2}$. Then 
\[
\left(\mathbf{B}^{e}(G_{3},K),d_{\tau_{3}}\right)=\left(\mathbf{B}^{e}(G_{1},K),d_{\tau_{1}}\right)\times\left(\mathbf{B}^{e}(G_{2},K),d_{\tau_{2}}\right)
\]
in $\CAT$. This actually means that for $x_{3}=(x_{1},x_{2})$ and
$y_{3}=(y_{1},y_{2})$ in 
\[
\mathbf{B}^{e}(G_{3},K)=\mathbf{B}^{e}(G_{1},K)\times\mathbf{B}^{e}(G_{2},K)
\]
we have the usual Pythagorean formula 
\[
d_{\tau_{3}}(x_{3},y_{3})=\sqrt{d_{\tau_{1}}(x_{1},y_{1})^{2}+d_{\tau_{2}}(x_{2},y_{2})^{2}}.
\]
It immediately follows that {\small{}
\[
\left(\mathbf{B}^{e}(G_{3},K)\stackrel{\varpi_{3}}{\twoheadrightarrow}\mathbf{F}(G_{3},k)\right)=\left(\mathbf{B}^{e}(G_{1},K)\times\mathbf{B}^{e}(G_{2},K)\stackrel{(\varpi_{1},\varpi_{2})}{\twoheadrightarrow}\mathbf{F}(G_{1},k)\times\mathbf{F}(G_{2},k)\right)
\]
}where $\varpi_{i}:=\varpi_{G_{i}}$ is the retraction attached to
$\tau_{i}$. Applying this to $G_{1}=G_{2}=G$ and using the commutativity
of our second diagram for the diagonal embedding $\Delta:G\hookrightarrow G\times G$,
we obtain $\Delta\circ\varpi_{3}=(\varpi_{1},\varpi_{2})\circ\Delta$,
where $\varpi_{3}$ is now the retraction $\varpi_{G}$ attached to
the faithful representation $\tau_{1}\oplus\tau_{2}=\Delta^{\ast}(\tau_{3})$
of $G$. Thus $\varpi_{1}=\varpi_{3}=\varpi_{2}$, i.e.~$\varpi_{G}$
does not depend upon the choice of $\tau$. Using the commutativity
of our second diagram for the graph embedding $\Delta_{f}:G_{1}\hookrightarrow G_{1}\times G_{2}$
of a morphism $f:G_{1}\rightarrow G_{2}$, we similarly obtain the
functoriality of $G\mapsto\varpi_{G}$. 
\end{proof}

\subsubsection{~}

With notations as above, the \emph{Busemann scalar product} is the
function 
\[
\left\langle -,-\right\rangle _{\tau}:\mathbf{B}^{e}(G,K)^{2}\times\mathbf{F}(G,K)\rightarrow\mathbb{R}
\]
which maps $(x,y,\mathcal{F})$ to 
\[
\left\langle \overrightarrow{xy},\mathcal{F}\right\rangle _{\tau}\eqd\left\Vert \mathcal{F}\right\Vert _{\tau}\cdot\lim_{t\rightarrow\infty}\left(d_{\tau}(x,z+t\mathcal{F})-d_{\tau}(y,z+t\mathcal{F})\right).
\]
Here $z$ is any fixed point in $\mathbf{B}^{e}(G,K)$: the limit
exists and does not depend upon the chosen $z$~\cite[5.5.8]{Co14}.
For every $x,y,z\in\mathbf{B}^{e}(G,K)$, $\mathcal{F}\in\mathbf{F}(G,K)$
and $t\geq0$, 
\[
\left\langle \overrightarrow{xz},\mathcal{F}\right\rangle _{\tau}=\left\langle \overrightarrow{xy},\mathcal{F}\right\rangle _{\tau}+\left\langle \overrightarrow{yz},\mathcal{F}\right\rangle _{\tau}\quad\mbox{and}\quad\left\langle \overrightarrow{xy},t\mathcal{F}\right\rangle _{\tau}=t\left\langle \overrightarrow{xy},\mathcal{F}\right\rangle _{\tau}.
\]
As a function of $x$, $\left\langle \overrightarrow{xy},\mathcal{F}\right\rangle _{\tau}$
is convex and $\left\Vert \mathcal{F}\right\Vert _{\tau}$-Lipschitzian;
as a function of $y$, it is concave and $\left\Vert \mathcal{F}\right\Vert _{\tau}$-Lipschitzian;
as a function of $\mathcal{F}$, it is usually neither convex nor
concave, but it is $d_{\tau}(x,y)$-Lipschitzian~\cite[5.5.11]{Co14};
as a function of $\tau$, it is additive: if $\tau'$ is another faithful
representation of $G$, then 
\[
\left\langle \overrightarrow{xy},\mathcal{F}\right\rangle _{\tau\oplus\tau'}=\left\langle \overrightarrow{xy},\mathcal{F}\right\rangle _{\tau}+\left\langle \overrightarrow{xy},\mathcal{F}\right\rangle _{\tau'}.
\]
For any $x\in\mathbf{B}^{e}(G,K)$ and $\mathcal{F}\in\mathbf{F}(G,k)$,
we have the following inequality \cite[5.5.9]{Co14}:
\[
\left\langle \overrightarrow{\circ x},\mathcal{F}\right\rangle _{\tau}\leq\left\langle \loc(x),\mathcal{F}\right\rangle _{\tau}.
\]
This is an equality when $x$ belongs to $\mathbf{F}(G,k)\simeq\circ+\mathbf{F}(G,k)$. 
\begin{prop}
\label{prop:CompBusemProj}Suppose that $\ell$ is a separable extension
of $k$. Let $H$ be a reductive subgroup of $G$. Then for every
$x\in\mathbf{B}^{e}(H,K)$ and $\mathcal{F}\in\mathbf{F}(G,k)$, 
\[
\left\langle \overrightarrow{\circ x},\mathcal{F}\right\rangle _{\tau}\leq\left\langle \overrightarrow{\circ x},p_{k}(\mathcal{F})\right\rangle _{\tau}
\]
where $p_{k}:\mathbf{F}(G,k)\twoheadrightarrow\mathbf{F}(H,k)$ is
the convex projection attached to $d_{\tau}$.
\end{prop}
\begin{proof}
Set $\mathcal{G}=p_{k}(\mathcal{F})\in\mathbf{F}(H,k)$ and pick a
splitting of $\mathcal{G}$~\cite[Cor. 63]{Co14}, corresponding
to an $\mathbb{R}$-filtration $\mathcal{G}'\in\mathbf{F}(H,k)$ opposed
to $\mathcal{G}$: for any representation $\sigma$ of $H$, 
\[
\omega_{H,k}(\sigma)=\oplus_{\gamma\in\mathbb{R}}\mathcal{G}(\sigma)^{\gamma}\cap\mathcal{G}'(\sigma)^{-\gamma}.
\]
Let $Q_{\mathcal{G}}\subset P_{\mathcal{G}}$ and $Q_{\mathcal{G}'}\subset P_{\mathcal{G}'}$
be the stabilizers of $\mathcal{G}$ and $\mathcal{G}'$ in $H$ and
$G$, so that $(Q_{\mathcal{G}},Q_{\mathcal{G}'})$ and $(P_{\mathcal{G}},P_{\mathcal{G}'})$
are pairs of opposed parabolic subgroups of $H$ and $G$, with Levi
subgroups $H':=Q_{\mathcal{G}}\cap Q_{\mathcal{G}'}$ and $G':=P_{\mathcal{G}}\cap P_{\mathcal{G}'}$.
Let $R^{u}(-)$ denote the unipotent radical. Then for $\star\in\{k,\ell,K\}$,
$\mathbf{B}^{e}(H',K)$, $\mathbf{B}^{e}(G',K)$, $\mathbf{F}(H',\star)$
and $\mathbf{F}(G',\star)$ are fundamental domains for the actions
of $R^{u}Q_{\mathcal{G}}(K)$, $R^{u}P_{\mathcal{G}}(K)$, $R^{u}Q_{\mathcal{G}}(\star)$
and $R^{u}P_{\mathcal{G}}(\star)$ on respectively $\mathbf{B}^{e}(H,K)$,
$\mathbf{B}^{e}(G,K)$, $\mathbf{F}(H,\star)$ and $\mathbf{F}(G,\star)$
\cite[5.2.10]{Co14}. We denote by the same letter $r$ the corresponding
retractions. They are all non-expanding, and the following diagrams
are commutative:
\[
\xymatrix{\mathbf{B}^{e}(H,K)\ar@{^{(}->}[r]\ar@{->>}[d]_{r} & \mathbf{B}^{e}(G,K)\ar@{->>}[d]^{r} & \mathbf{F}(H,\star)\ar@{^{(}->}[r]\ar@{->>}[d]_{r} & \mathbf{F}(G,\star)\ar@{->>}[d]^{r}\\
\mathbf{B}^{e}(H',K)\ar@{^{(}->}[r] & \mathbf{B}^{e}(G',K) & \mathbf{F}(H',\star)\ar@{^{(}->}[r] & \mathbf{F}(G',\star)
}
\]
Let $x':=r(x)$ and $\mathcal{F}':=r(\mathcal{F})$, so that $x'\in\mathbf{B}^{e}(H',K)$,
$\mathcal{F}'\in\mathbf{F}(G',k)$. Note that already $\mathcal{G},\mathcal{G}'\in\mathbf{F}(H',k)$.
We will establish the following inequalities:
\[
\left\langle \overrightarrow{\circ x},\mathcal{F}\right\rangle _{\tau}\stackrel{(1)}{\leq}\left\langle \overrightarrow{\circ x}',\mathcal{F}'\right\rangle _{\tau}\stackrel{(2)}{\leq}\left\langle \loc(x'),\mathcal{F}'\right\rangle _{\tau}\stackrel{(3)}{\leq}\left\langle \loc(x'),\mathcal{G}\right\rangle _{\tau}\stackrel{(4)}{=}\left\langle \overrightarrow{\circ x}',\mathcal{G}\right\rangle _{\tau}\stackrel{(5)}{=}\left\langle \overrightarrow{\circ x},\mathcal{G}\right\rangle _{\tau}.
\]
The second inequality was already mentioned just before the proposition.

\emph{Proof of $(1)$}. Since $\mathcal{F}'=r(\mathcal{F})$, there
is a $u\in R^{u}P_{\mathcal{G}}(k)$ such that $\mathcal{F}'=u\mathcal{F}$.
Since $u\in G(k)$ and all of our distances, norms etc... are $G(k)$-invariant,
it follows that $\left\Vert \mathcal{F}\right\Vert _{\tau}=\left\Vert \mathcal{F}'\right\Vert _{\tau}$.
Since $u\in G(\mathcal{O})$ fixes $\circ$, $u(\circ+t\mathcal{F})=\circ+t\mathcal{F}'$
belongs to $\mathbf{B}^{e}(G',K)$. Since $u\in R^{u}P_{\mathcal{G}}(K)$,
$r(\circ+t\mathcal{F})=\circ+t\mathcal{F}'$ for all $t\geq0$. Thus
\begin{eqnarray*}
\left\langle \overrightarrow{\circ x},\mathcal{F}\right\rangle _{\tau} & = & \left\Vert \mathcal{F}\right\Vert _{\tau}\lim_{t\rightarrow\infty}\left(t\left\Vert \mathcal{F}\right\Vert _{\tau}-d_{\tau}(x,\circ+t\mathcal{F})\right)\\
 & \leq & \left\Vert \mathcal{F}'\right\Vert _{\tau}\lim_{t\rightarrow\infty}\left(t\left\Vert \mathcal{F}'\right\Vert _{\tau}-d_{\tau}(x',\circ+t\mathcal{F}')\right)\\
 & = & \left\langle \overrightarrow{\circ x}',\mathcal{F}'\right\rangle _{\tau}
\end{eqnarray*}
since $r:\mathbf{B}^{e}(G,K)\twoheadrightarrow\mathbf{B}^{e}(G',K)$
is non-expanding.

\emph{Proof of }$(3)$. Note that $\loc(x')\in\mathbf{F}(H',\ell)$
and $\mathcal{F}'\in\mathbf{F}(G',k)$. By the last assertion of theorem~\ref{thm:FunctFil},
it is sufficient to establish that $p'_{k}(\mathcal{F}')=\mathcal{G}$
for the convex projection $p'_{k}:\mathbf{F}(G',k)\twoheadrightarrow\mathbf{F}(H',k)$
\textendash{} which is usually not equal to the restriction of $p_{k}:\mathbf{F}(G,k)\twoheadrightarrow\mathbf{F}(H,k)$
to $\mathbf{F}(G',k)$. For $t\gg0$, $\mathcal{F}+t\mathcal{G}=\mathcal{F}'+t\mathcal{G}$
by \cite[5.6.2]{Co14}. In particular $\mathcal{F}+t\mathcal{G}$
belongs to $\mathbf{F}(G',k)$ since $\mathcal{F}'$ and $\mathcal{G}$
do. On the other hand, 
\[
p_{k}(\mathcal{F}+t\mathcal{G})=(1+t)p_{k}\left({\textstyle \frac{1}{1+t}}\mathcal{F}+{\textstyle \frac{t}{1+t}}\mathcal{G}\right)=(1+t)\mathcal{G}
\]
using~\cite[II.2.4]{BrHa99} for the second equality. Since this
belongs to $\mathbf{F}(H',k)$, actually 
\[
p'_{k}(\mathcal{F}'+t\mathcal{G})=p'_{k}(\mathcal{F}+t\mathcal{G})=p_{k}(\mathcal{F}+t\mathcal{G})=(1+t)\mathcal{G}.
\]
Now observe that $\mathcal{H}\mapsto\mathcal{H}+t\mathcal{G}$ and
$\mathcal{H}\mapsto\mathcal{H}+t\mathcal{G}'$ are mutually inverse
isometries of $\mathbf{F}(G',k)$ and $\mathbf{F}(H',k)$, thus $p'_{k}$
commutes with both of them and 
\[
p'_{k}(\mathcal{F}')=p'_{k}(\mathcal{F}'+t\mathcal{G})+t\mathcal{G}'=(1+t)\mathcal{G}+t\mathcal{G}'=\mathcal{G}.
\]

\emph{Proof of $(4)$. }This follows from~\cite[5.5.3]{Co14}.

\emph{Proof of $(5)$. }Since $x'=r(x)$, there is a $u\in R^{u}Q_{\mathcal{G}}(K)$
such that $ux=x'$. For $t\gg0$, $u$ fixes $\circ+t\mathcal{G}$
by \cite[5.4.6]{Co14}. Then $d_{\tau}\left(x',\circ+t\mathcal{G}\right)=d_{\tau}\left(x,\circ+t\mathcal{G}\right)$
and
\begin{eqnarray*}
\left\langle \overrightarrow{\circ x}',\mathcal{G}\right\rangle _{\tau} & = & \left\Vert \mathcal{G}\right\Vert _{\tau}\lim_{t\rightarrow\infty}\left(t\left\Vert \mathcal{G}\right\Vert _{\tau}-d_{\tau}(x',\circ+t\mathcal{G})\right)\\
 & = & \left\Vert \mathcal{G}\right\Vert _{\tau}\lim_{t\rightarrow\infty}\left(t\left\Vert \mathcal{G}\right\Vert _{\tau}-d_{\tau}(x,\circ+t\mathcal{G})\right)\\
 & = & \left\langle \overrightarrow{\circ x},\mathcal{G}\right\rangle _{\tau}.
\end{eqnarray*}
This finishes the proof of the proposition.
\end{proof}
\begin{cor}
\label{cor:ConcaveBusemann}For every $x\in\mathbf{B}^{e}(G,K)$,
$\mathcal{F}\mapsto\left\langle \overrightarrow{\circ x},\mathcal{F}\right\rangle _{\tau}$
is concave on $\mathbf{F}(G,k)$.
\end{cor}
\begin{proof}
We have to show that for any $x\in\mathbf{B}^{e}(G,K)$ and $\mathcal{F},\mathcal{G}\in\mathbf{F}(G,k)$,
\[
\left\langle \overrightarrow{\circ x},\mathcal{F}\right\rangle _{\tau}+\left\langle \overrightarrow{\circ x},\mathcal{G}\right\rangle _{\tau}\leq\left\langle \overrightarrow{\circ x},\mathcal{F}+\mathcal{G}\right\rangle _{\tau}.
\]
For the diagonal embedding $\Delta:G\hookrightarrow G\times G$, the
proposition gives
\[
\left\langle \overrightarrow{\circ x},\mathcal{H}\right\rangle _{\tau\boxplus\tau}\leq\left\langle \overrightarrow{\circ x},p_{k}(\mathcal{H})\right\rangle _{\tau\oplus\tau}=2\left\langle \overrightarrow{\circ x},p_{k}(\mathcal{H})\right\rangle _{\tau}
\]
for every $\mathcal{H}$ in $\mathbf{F}(G\times G,k)=\mathbf{F}(G,k)\times\mathbf{F}(G,k)$.
For $\mathcal{H}=(\mathcal{F},\mathcal{G})$, we have 
\[
\left\langle \overrightarrow{\circ x},\mathcal{H}\right\rangle _{\tau\boxplus\tau}=\left\langle \overrightarrow{\circ x},\mathcal{F}\right\rangle _{\tau}+\left\langle \overrightarrow{\circ x},\mathcal{G}\right\rangle _{\tau}
\]
and $p_{k}(\mathcal{H})$ is the point closest to $(\mathcal{F},\mathcal{G})$
in the diagonally embedded $\mathbf{F}(G,k)$: the middle point $\frac{1}{2}(\mathcal{F}+\mathcal{G})=\frac{1}{2}\mathcal{F}+\frac{1}{2}\mathcal{G}$
of the geodesic segment $[\mathcal{F},\mathcal{G}]$ of $\mathbf{F}(G,k)$.
Thus
\[
\left\langle \overrightarrow{\circ x},\mathcal{F}\right\rangle _{\tau}+\left\langle \overrightarrow{\circ x},\mathcal{G}\right\rangle _{\tau}\leq2\left\langle \overrightarrow{\circ x},{\textstyle \frac{1}{2}}(\mathcal{F}+\mathcal{G})\right\rangle _{\tau}=\left\langle \overrightarrow{\circ x},\mathcal{F}+\mathcal{G}\right\rangle _{\tau},
\]
which proves the corollary. 
\end{proof}

\subsubsection{~\label{subsec:Formula4Busemann}}

For $\mathcal{V}\in\Vect_{K}$ and for the canonical metric on $\mathbf{F}(\mathcal{V})$,
there is an explicit formula for the corresponding Busemann scalar
product
\[
\left\langle -,-\right\rangle :\mathbf{B}(\mathcal{V})^{2}\times\mathbf{F}(\mathcal{V})\rightarrow\mathbb{R}.
\]
which maps $(\alpha,\beta,\mathcal{F})$ to 
\[
\left\langle \overrightarrow{\alpha\beta},\mathcal{F}\right\rangle =\left\Vert \mathcal{F}\right\Vert \cdot\lim_{t\rightarrow\infty}\left(d(\alpha,\gamma+t\mathcal{F})-d(\beta,\gamma+t\mathcal{F})\right).
\]
By~\cite[6.4.15]{Co14}, the latter may indeed be computed as 
\[
\left\langle \overrightarrow{\alpha\beta},\mathcal{F}\right\rangle =\sum_{\gamma}\gamma\,\nu\left(\Gr_{\mathcal{F}}^{\gamma}(\alpha),\Gr_{\mathcal{F}}^{\gamma}(\beta)\right)
\]
where $\Gr_{\mathcal{F}}^{\gamma}(\alpha)$ and $\Gr_{\mathcal{F}}^{\gamma}(\beta)$
are the splittable $K$-norms on $\Gr_{\mathcal{F}}^{\gamma}(\mathcal{V})$
induced by $\alpha$ and $\beta$. If $\mathcal{V}=V_{K}$ and $\mathcal{F}=f_{K}$
for some $V\in\Vect_{k}$ and $f\in\mathbf{F}(V)$, then $\Gr_{\mathcal{F}}^{\gamma}(\mathcal{V})$
equals $\Gr_{f}^{\gamma}(V)\otimes_{k}K$; if moreover $\alpha$ is
the gauge norm of $V\otimes_{k}\mathcal{O}$, then $\Gr_{\mathcal{F}}^{\gamma}(\alpha)$
is the gauge norm of $\Gr_{f}^{\gamma}(V)\otimes_{k}\mathcal{O}$.
In particular, the pairing of section~\ref{subsec:CarHNQuAb}, 
\[
\left\langle -,-\right\rangle :\Norm_{k}^{K}(V)\times\mathbf{F}(V)\rightarrow\mathbb{R},\quad\left\langle \alpha,f\right\rangle =\sum\gamma\,\deg\Gr_{f}^{\gamma}(\alpha)
\]
is related to the Busemann scalar product by the formula
\[
\left\langle \alpha,f\right\rangle =\left\langle \overrightarrow{\alpha_{V\otimes\mathcal{O}}\alpha},f_{K}\right\rangle .
\]

\subsubsection{~}

The previous formula yields another proof of corollary~\ref{cor:ConcaveBusemann},
which now works without any assumption on the extension $\ell$ of
$k$: for every $x\in\mathbf{B}^{e}(G,K)$, the function $\mathcal{F}\mapsto\left\langle \overrightarrow{\circ x},\mathcal{F}\right\rangle _{\tau}$
is concave on $\mathbf{F}(G,k)$ since for $\alpha:=\boldsymbol{\alpha}(x)\in\mathbf{B}(\omega_{G},K)$,
\[
\left\langle \overrightarrow{\circ x},\mathcal{F}\right\rangle _{\tau}=\left\langle \overrightarrow{\circ(\tau)x(\tau)},\mathcal{F}(\tau)\right\rangle =\left\langle \alpha(\tau),\mathcal{F}(\tau)\right\rangle 
\]
and $f\mapsto\left\langle \alpha(\tau),f\right\rangle $ is a degree
function on $\mathbf{F}(\omega_{G,k}(\tau))$. If $\ell$ is a separable
extension of $k$, proposition~\ref{prop:CompBusemProj} implies
that every $\alpha\in\mathbf{B}(\omega_{G},K)$ is good. On the other
hand for every pair of objects $(V_{1},\alpha_{1})$ and $(V_{2},\alpha_{2})$
in $\Norm_{k}^{K}$ and $G:=GL(V_{1})\times GL(V_{2})$, 
\[
\mathbf{B}^{e}(G,K)\simeq\mathbf{B}(\omega_{G},K)\simeq\mathbf{B}(V_{1,K})\times\mathbf{B}(V_{2,K})
\]
contains $(\alpha_{1},\alpha_{2})$, therefore $(\Norm_{k}^{K},\deg)$
is then also good. We obtain:
\begin{thm}
Suppose that $\ell$ is a separable extension of $k$. Then
\[
\mathcal{F}_{HN}:\Norm_{k}^{K}\rightarrow\Fil_{k}\mbox{ is a \ensuremath{\otimes}-functor}.
\]
For every $\alpha\in\mathbf{B}(\omega_{G},K)$, $\mathcal{F}_{HN}(\alpha):=\mathcal{F}_{HN}\circ\alpha$
belongs to $\mathbf{F}(G,k)$, i.e.
\[
\mathcal{F}_{HN}(\alpha):\Rep(G)\rightarrow\Fil_{k}\mbox{\,\ is an exact \ensuremath{\otimes}-functor.}
\]
For any faithful representation $\tau$ of $G$ and $x\in\mathbf{B}^{e}(G,K)$,
$\pi_{G}(x):=\mathcal{F}_{HN}(\boldsymbol{\alpha}(x))$ is the unique
element $\mathcal{F}$ of $\mathbf{F}(G,k)$ which satisfies the following
equivalent conditions:

\begin{enumerate}
\item For every $f\in\mathbf{F}(G,k)$, $\left\Vert \mathcal{F}\right\Vert _{\tau}^{2}-2\left\langle \overrightarrow{\circ x},\mathcal{F}\right\rangle _{\tau}\leq\left\Vert f\right\Vert _{\tau}^{2}-2\left\langle \overrightarrow{\circ x},f\right\rangle _{\tau}$.
\item For every $f\in\mathbf{F}(G,k)$, \textup{$\left\langle \overrightarrow{\circ x},f\right\rangle _{\tau}\leq\left\langle \mathcal{F},f\right\rangle _{\tau}$}\textup{\emph{
with equality for $f=\mathcal{F}$. }}
\item For every $\gamma\in\mathbb{R}$, $\Gr_{\mathcal{F}}^{\gamma}(\boldsymbol{\alpha}(x))(\tau)$
is semi-stable of slope $\gamma$.
\end{enumerate}
The function $x\mapsto\pi_{G}(x)$ is non-expanding for $d_{\tau}$
and defines a retraction 
\[
\pi:\mathbf{B}^{e}(-,K)\twoheadrightarrow\mathbf{F}(-,k)
\]
of the embedding $\mathbf{F}(-,k)\hookrightarrow\mathbf{B}^{e}(-,K)$
of functors from $\Red(k)$ to $\Top$.
\end{thm}
\begin{proof}
Everything follows from proposition~\ref{prop:CaractGood} except
the last sentence, which still requires a proof. For $x,y\in\mathbf{B}^{e}(G,K)$,
set $\mathcal{F}:=\pi_{G}(x)$ and $\mathcal{G}:=\pi_{G}(y)$. Then
\begin{eqnarray*}
d_{\tau}(\mathcal{F},\mathcal{G})^{2} & = & \left\Vert \mathcal{F}\right\Vert _{\tau}^{2}+\left\Vert \mathcal{G}\right\Vert _{\tau}^{2}-\left\langle \mathcal{F},\mathcal{G}\right\rangle _{\tau}-\left\langle \mathcal{G},\mathcal{F}\right\rangle _{\tau}\\
 & \leq & \left\langle \overrightarrow{\circ x},\mathcal{F}\right\rangle _{\tau}+\left\langle \overrightarrow{\circ y},\mathcal{G}\right\rangle _{\tau}-\left\langle \overrightarrow{\circ x},\mathcal{G}\right\rangle _{\tau}-\left\langle \overrightarrow{\circ y},\mathcal{F}\right\rangle _{\tau}\\
 & = & \left\langle \overrightarrow{xy},\mathcal{G}\right\rangle _{\tau}-\left\langle \overrightarrow{xy},\mathcal{F}\right\rangle _{\tau}\\
 & \leq & d_{\tau}(x,y)\cdot d_{\tau}(\mathcal{F},\mathcal{G})
\end{eqnarray*}
thus $d_{\tau}(\mathcal{F},\mathcal{G})\leq d_{\tau}(x,y)$, i.e.~$\pi_{G}:\mathbf{B}^{e}(G,K)\rightarrow\mathbf{F}(G,k)$
is indeed non-expanding for $d_{\tau}$. It is plainly functorial
in $G$. For $\mathcal{F},f\in\mathbf{F}(G,k)$ and $x:=\circ+\mathcal{F}$,
we have 
\[
\left\langle \overrightarrow{\circ x},f\right\rangle _{\tau}=\left\langle \mathcal{F},f\right\rangle _{\tau}
\]
thus $\pi_{G}(x)=\mathcal{F}$, i.e.~$\pi$ is indeed a retraction
of $\mathbf{F}(-,k)\hookrightarrow\mathbf{B}^{e}(-,K)$.
\end{proof}
\noindent Once we know that the projection $\pi_{G}:\mathbf{B}^{e}(G,K)\twoheadrightarrow\mathbf{F}(G,k)$
computes the Harder-Narasimhan filtrations, the compatibility of the
latter with tensor product constructions again directly follows from
the functoriality of $G\mapsto\pi_{G}$:
\begin{prop}
The Harder-Narasimhan functor $\mathcal{F}_{HN}:\Norm_{k}^{K}\rightarrow\Fil_{k}$
is compatible with tensor products, symmetric and exterior powers,
and duals.
\end{prop}
\begin{proof}
Apply the functoriality of $G\mapsto\pi_{G}$ to $GL(V_{1})\times GL(V_{2})\rightarrow GL(V_{1}\otimes V_{2})$,
$GL(V)\rightarrow GL(\Sym^{r}V)$, $GL(V)\rightarrow GL(\Lambda^{r}V)$
and $GL(V)\rightarrow GL(V^{\ast})$.
\end{proof}
\begin{rem}
We now have three non-expanding retractions of $\mathbf{F}(-,k)\hookrightarrow\mathbf{B}^{e}(-,K)$:
$(1)$ the composition $\pi\circ\loc$ where $\pi:\mathbf{F}(-,\ell)\twoheadrightarrow\mathbf{F}(-,k)$
is the convex projection from~theorem~\ref{thm:FunctFil}, which
computes the Harder-Narasimhan filtration on $\Fil_{k}^{\ell}$\emph{;
$(2)$ }the convex projection $\omega:\mathbf{B}^{e}(-,K)\twoheadrightarrow\mathbf{F}(-,k)$
from theorem~\ref{thm:FunctFakeHN}; $(3)$ the retraction $\pi:\mathbf{B}^{e}(-,K)\twoheadrightarrow\mathbf{F}(-,k)$
that we have just defined, which computes the Harder-Narasimhan filtration
on $\Norm_{k}^{K}$. We leave it to the reader to verify that already
for $G=PGL(2)$, these three retractions are pairwise distinct.
\end{rem}

\subsection{Normed $\varphi$-modules\label{subsec:NormedPhiModules}}

\subsubsection{~}

Let $k=\mathbb{F}_{q}$ be a finite field, $K$ an extension of $k$,
$\left|-\right|:K\rightarrow\mathbb{R}_{+}$ a non-archimedean absolute
value such that the local $k$-algebra $\mathcal{O}=\left\{ x\in K:\left|x\right|\leq1\right\} $
is Henselian with residue field $\ell$, $K^{s}$ a fixed separable
closure of $K$ with Galois group $\Gal_{K}=\Gal(K^{s}/K)$. The category
$\Rep_{k}(\Gal_{K})$ of continuous (i.e.~with open kernels) representations
$(V,\rho)$ of $\Gal_{K}$ on finite dimensional $k$-vector spaces
is a $k$-linear neutral tannakian category which is equivalent to
the category $\Vect_{K}^{\varphi}$ of étale $\varphi$-modules $(\mathcal{V},\varphi_{\mathcal{V}})$
over $K$. Here $\varphi(x)=x^{q}$ is the Frobenius of $K$, $\mathcal{V}$
is a finite dimensional $K$-vector space and $\varphi_{\mathcal{V}}:\varphi^{\ast}\mathcal{V}\rightarrow\mathcal{V}$
is a $K$-linear isomorphism where $\varphi^{\ast}\mathcal{V}=\mathcal{V}\otimes_{K,\varphi}K$.
The equivalence of categories is given by 
\begin{eqnarray*}
(V,\rho) & \rightarrow & \left((V\otimes_{k}K^{s})^{\Gal_{K}},\mathrm{Id}_{V}\otimes\varphi\right)\\
\left(\left(\mathcal{V}\otimes_{K}K^{s}\right)^{\varphi_{\mathcal{V}}\otimes\varphi=\mathrm{Id}},\gamma\mapsto\mathrm{Id}\otimes\gamma\right) & \leftarrow & (\mathcal{V},\varphi_{V})
\end{eqnarray*}

\subsubsection{~\label{subsec:CatNorm^phi_K}}

We denote by $\Norm_{K}^{\varphi}$ the quasi-abelian $k$-linear
$\otimes$-category of all triples $(\mathcal{V},\varphi_{\mathcal{V}},\alpha)$
where $(\mathcal{V},\varphi_{\mathcal{V}})$ is an étale $\varphi$-module
and $\alpha$ is a splittable $K$-norm on $\mathcal{V}$, with the
obvious morphisms and $\otimes$-products. It comes with two exact
$\otimes$-functors 
\[
\Norm_{K}^{\varphi}\rightarrow\Norm_{K},\quad(\mathcal{V},\varphi_{\mathcal{V}},\alpha)\mapsto(\mathcal{V},\alpha)\mbox{ or }(\mathcal{V},\varphi_{\mathcal{V}}(\alpha))
\]
where $\varphi_{\mathcal{V}}(\alpha)$ is the splittable $K$-norm
on $\mathcal{V}$ defined by 
\begin{eqnarray*}
(\varphi_{\mathcal{V}}(\alpha))(v) & \eqd & (\varphi^{\ast}\alpha)(\varphi_{\mathcal{V}}^{-1}(v))\\
\mbox{with}\quad(\varphi^{\ast}\alpha)(v') & \eqd & \min\left\{ \max\left\{ \left|\lambda_{i}\right|\alpha(v_{i})^{q}\right\} :\begin{array}{l}
v'=\sum v_{i}\otimes\lambda_{i}\\
\lambda_{i}\in K,\,v_{i}\in V
\end{array}\right\} 
\end{eqnarray*}
for $v\in\mathcal{V}$ and $v'\in\varphi^{\ast}\mathcal{V}:=\mathcal{V}\otimes_{K,\varphi}K$.
Note that for $\alpha,\beta\in\mathbf{B}(\mathcal{V})$,
\[
\mathbf{d}(\varphi_{\mathcal{V}}(\alpha),\varphi_{\mathcal{V}}(\beta))=q\cdot\mathbf{d}(\alpha,\beta)\in\mathbb{R}_{\geq}^{r}\quad\mbox{and}\quad\nu(\varphi_{\mathcal{V}}(\alpha),\varphi_{\mathcal{V}}(\beta))=q\cdot\nu(\alpha,\beta)\in\mathbb{R}.
\]

\subsubsection{~}

We may then consider the following setup:
\[
\begin{array}{rcl}
\A & = & \Rep_{k}(\Gal_{K})\\
\C & = & \Norm_{K}^{\varphi}
\end{array}\quad\mbox{with}\quad\left\{ \begin{array}{rcl}
\omega(\mathcal{V},\varphi_{\mathcal{V}},\alpha) & = & \left(\mathcal{V}\otimes_{K}K^{s}\right)^{\varphi_{\mathcal{V}}\otimes\varphi=\mathrm{Id}},\\
\rank(V,\rho) & = & \dim_{k}V,\\
\deg(\mathcal{V},\varphi_{\mathcal{V}},\alpha) & = & \nu(\alpha,\varphi_{\mathcal{V}}(\alpha)).
\end{array}\right.
\]
These data again satisfy the assumptions of sections~\ref{subsec:HypOnA}-\ref{subsec:HypOnC}.
For instance, if 
\[
f:(\mathcal{V}_{1},\varphi_{1},\alpha_{1})\rightarrow(\mathcal{V}_{2},\varphi_{2},\alpha_{2})
\]
is a mono-epi in $\Norm_{K}^{\varphi}$, then $f:(\mathcal{V}_{1},\varphi_{1})\rightarrow(\mathcal{V}_{2},\varphi_{2})$
is an isomorphism and 
\begin{eqnarray*}
\nu\left(\alpha_{1},\varphi_{1}(\alpha_{1})\right) & = & \nu\left(f_{\ast}(\alpha_{1}),f_{\ast}(\varphi_{1}(\alpha_{1}))\right)\\
 & = & \nu(f_{\ast}(\alpha_{1}),\alpha_{2})+\nu(\alpha_{2},\varphi_{2}(\alpha_{2}))+\nu(\varphi_{2}(\alpha_{2}),\varphi_{2}(f_{\ast}(\alpha_{1})))\\
 & = & \nu(\alpha_{2},\varphi_{2}(\alpha_{2}))-(q-1)\nu(f_{\ast}(\alpha_{1}),\alpha_{2})
\end{eqnarray*}
where $f_{\ast}(\alpha)(x)=\alpha\circ f^{-1}(x)$, so that $f_{\ast}(\varphi_{1}(\alpha_{1}))=\varphi_{2}(f_{\ast}(\alpha_{1}))$,
thus 
\[
\deg(\mathcal{V}_{1},\varphi_{1},\alpha_{1})\leq\deg(\mathcal{V}_{1},\varphi_{1},\alpha_{1})
\]
with equality if and only if $f_{\ast}(\alpha_{1})=\alpha_{2}$. We
thus obtain a HN-formalism on $\Norm_{K}^{\varphi}$. 

We will show that for any reductive group $G$ over $k$, any faithful
exact $\otimes$-functor $\Rep(G)\rightarrow\Norm_{K}^{\varphi}$
is good, and the pair $(\Norm_{K}^{\varphi},\deg)$ itself is good.
In particular, the corresponding HN-filtration on $\Norm_{K}^{\varphi}$
is a $\otimes$-functor 
\[
\mathcal{F}_{HN}:\Norm_{K}^{\varphi}\rightarrow\F\left(\Rep_{k}\left(\Gal_{K}\right)\right).
\]

\subsubsection{~}

Since $\mathcal{O}$ is Henselian, the absolute value of $K$ has
a unique extension to $K^{s}$, which we also denote by $\left|-\right|:K^{s}\rightarrow\mathbb{R}_{+}$.
The corresponding valuation ring $\mathcal{O}^{s}:=\left\{ x\in K^{s}:\left|x\right|\leq1\right\} $
is the integral closure of $\mathcal{O}$ in $K^{s}$, and it is a
strictly Henselian local ring. There is a commutative diagram of $\otimes$-functors
\[
\xymatrix{\A=\Rep_{k}(\Gal_{K})\ar[d]^{\mathrm{forget}\,\rho}\ar@{<->}[r] & \Vect_{K}^{\varphi}\ar[d]^{-\otimes_{K}K^{s}} & \Norm_{K}^{\varphi}=\C\ar[d]^{-\otimes_{K}K^{s}}\ar[l]\ar[r] & \Norm_{K}\ar[d]^{-\otimes_{K}K^{s}}\\
\A^{s}=\Vect_{k}\ar@{<->}[r] & \Vect_{K^{s}}^{\varphi} & \Norm_{K^{s}}^{\varphi}=\C^{s}\ar[l]\ar[r] & \Norm_{K^{s}}
}
\]
in which the horizontal functors are equivalence of categories in
the first square, forget the norms in the second square, and map $(\mathcal{V},\varphi_{\mathcal{V}},\alpha)$
to either $(\mathcal{V},\alpha)$ or $(\mathcal{V},\varphi_{\mathcal{V}}(\alpha))$
in the third square. The last vertical functor maps $(\mathcal{V},\alpha)$
to $(\mathcal{V}^{s},\alpha^{s})$ with 
\[
\mathcal{V}^{s}\eqd\mathcal{V}\otimes_{K}K^{s}\quad\mbox{and}\quad\alpha^{s}(v)\eqd\min\left\{ \max\left\{ \left|\lambda_{i}\right|\alpha(v_{i}):i\right\} \left|\begin{array}{l}
v=\sum v_{i}\otimes\lambda_{i}\\
v_{i}\in\mathcal{V},\,\lambda_{i}\in K^{s}
\end{array}\right.\right\} .
\]
By~\cite[Lemma 132]{Co14}, there is an extension $(K',\left|-\right|)$
of $(K^{s},\left|-\right|)$ with $K'$ algebraically closed (in which
case $\mathcal{O}':=\left\{ x\in K':\left|x\right|\leq1\right\} $
is strictly Henselian) and $\left|K'\right|=\mathbb{R}$. We may then
add a third row to our commutative diagram,
\[
\xymatrix{\A^{s}=\Vect_{k}\ar@{=}[d]\ar@{<->}[r] & \Vect_{K^{s}}^{\varphi}\ar[d]^{-\otimes_{K^{s}}K^{\prime}} & \Norm_{K^{s}}^{\varphi}=\C^{s}\ar[d]^{-\otimes_{K^{s}}K^{\prime}}\ar[l]\ar[r] & \Norm_{K^{s}}\ar[d]^{-\otimes_{K^{s}}K^{\prime}}\\
\A^{\prime}=\Vect_{k}\ar@{<->}[r] & \Vect_{K^{\prime}}^{\varphi} & \Norm_{K^{\prime}}^{\varphi}=\C^{\prime}\ar[l]\ar[r] & \Norm_{K^{\prime}}
}
\]

\subsubsection{~}

Let now $G$ be a reductive group over $k$ and let $x:\Rep(G)\rightarrow\Norm_{K}^{\varphi}$
be a faithful exact $k$-linear $\otimes$-functor, with base change
\[
x^{s}:\Rep(G)\rightarrow\Norm_{K^{s}}^{\varphi}\quad\mbox{and}\quad x^{\prime}:\Rep(G)\rightarrow\Norm_{K'}^{\varphi}
\]
and Galois representation $\omega_{G,\A}:\Rep(G)\rightarrow\Rep_{k}(\Gal_{K})$.
We denote by
\[
\omega_{G,\A}=:(V,\rho),\quad x=:(\mathcal{V},\varphi_{\mathcal{V}},\alpha),\quad x^{s}=:(\mathcal{V}^{s},\varphi_{\mathcal{V}^{s}},\alpha^{s})\quad\mbox{and}\quad x'=:(\mathcal{V}^{\prime},\varphi_{\mathcal{V}^{\prime}},\alpha^{\prime})
\]
the components of $\omega_{G,\A}$, $x$, $x^{s}$ and $x'$. Let
$\tau$ be a faithful representation of $G$ and 
\[
p:\mathbf{F}(\omega_{G,\A}(\tau))\twoheadrightarrow\mathbf{F}(\omega_{G,\A})(\tau)
\]
the projection to the image of $\mathbf{F}(\omega_{G,\A})\hookrightarrow\mathbf{F}(\omega_{G,\A}(\tau))$.
We want to show that 
\[
\left\langle x(\tau),f\right\rangle \leq\left\langle x(\tau),p(f)\right\rangle 
\]
for every $f\in\mathbf{F}(\omega_{G,\A}(\tau))$. As in~\ref{subsec:Formula4Busemann},
this amounts to 
\[
\left\langle \overrightarrow{\alpha(\tau)\varphi_{\mathcal{V}(\tau)}(\alpha(\tau))},\mathcal{F}\right\rangle \leq\left\langle \overrightarrow{\alpha(\tau)\varphi_{\mathcal{V}(\tau)}(\alpha(\tau))},\mathcal{G}\right\rangle 
\]
for the Busemann scalar product on $\mathbf{B}(\mathcal{V}(\tau))$,
where $\mathcal{F}$ and $\mathcal{G}$ are the $\varphi_{\mathcal{V}(\tau)}$-stable
filtrations on $\mathcal{V}(\tau)$ corresponding to the $\Gal_{K}$-stable
filtrations $f$ and $p(f)$ on $V(\tau)$. Since the CAT(0)-spaces
$\mathbf{B}(\mathcal{V}(\tau)\otimes-)$ are functorial on $\HV(k)$,
this amounts to
\begin{eqnarray*}
\left\langle \overrightarrow{\alpha^{s}(\tau)\varphi_{\mathcal{V}^{s}(\tau)}(\alpha^{s}(\tau))},\mathcal{F}^{s}\right\rangle  & \leq & \left\langle \overrightarrow{\alpha^{s}(\tau)\varphi_{\mathcal{V}^{s}(\tau)}(\alpha^{s}(\tau))},\mathcal{G}^{s}\right\rangle \\
\mbox{or}\quad\left\langle \overrightarrow{\alpha^{\prime}(\tau)\varphi_{\mathcal{V}^{\prime}(\tau)}(\alpha^{\prime}(\tau))},\mathcal{F}^{\prime}\right\rangle  & \leq & \left\langle \overrightarrow{\alpha^{\prime}(\tau)\varphi_{\mathcal{V}^{\prime}(\tau)}(\alpha^{\prime}(\tau))},\mathcal{G}^{\prime}\right\rangle 
\end{eqnarray*}
for the Busemann scalar products on $\mathbf{B}(\mathcal{V}^{s}(\tau))$
or $\mathbf{B}(\mathcal{V}^{\prime}(\tau))$, where $\mathcal{F}^{\star}$
and $\mathcal{G}^{\star}$ are the $\varphi_{\mathcal{V}^{\star}(\tau)}$-stable
filtrations on $\mathcal{V}^{\star}(\tau):=\mathcal{V}(\tau)\otimes_{K}K^{\star}=V(\tau)\otimes_{k}K^{\star}$
base changed from $\mathcal{F}$ and $\mathcal{G}$ on $\mathcal{V}(\tau)$
or equivalently, from $f$ and $p(f)$ on $V(\tau)$ (for $\star\in\{s,\prime\}$).

\subsubsection{~}

Since $k$ is finite, it follows from Lang's theorem and Deligne's
work on tannakian categories that the fiber functor $V:\Rep(G)\rightarrow\Vect_{k}$
underlying $\omega_{G,\A}$ is isomorphic to the standard fiber functor
$\omega_{G,k}:\Rep(G)\rightarrow\Vect_{k}$. Without loss of generality,
we may thus assume that $V=\omega_{G,k}$, in which case 
\[
\omega_{G,\A}:\Rep(G)\rightarrow\Rep_{k}(\Gal_{K})
\]
is induced by a morphism $\rho:\Gal_{K}\rightarrow G(k)$ with open
kernel. Then 
\[
\omega_{G,\A^{s}}=\omega_{G,\A^{\prime}}=\omega_{G,k},\quad\mathcal{V}=(\omega_{G,\A}\otimes K^{s})^{\Gal_{K}}\quad\mbox{and}\quad\mathcal{V}^{\star}=\omega_{G,K^{\star}}
\]
for $\star\in\{s,\prime\}$. Moreover, the following commutative diagram
in $\CCAT$ 
\[
\xymatrix{\mathbf{F}(\omega_{G,\A})\ar@{^{(}->}[r]\ar@{^{(}->}[d] & \mathbf{F}(\omega_{G,\A}(\tau))\ar@{^{(}->}[d]\\
\mathbf{F}(\omega_{G,k})\ar@{^{(}->}[r] & \mathbf{F}(\omega_{G,k}(\tau))
}
\]
is $G(k)$-equivariant, thus also $\Gal_{K}$-equivariant, and identifies
its first row with the $\Gal_{K}$-invariants of its second row. It
follows that the corresponding diagram of convex projections is commutative:
\[
\xymatrix{\mathbf{F}(\omega_{G,\A})(\tau)\ar@{^{(}->}[d] & \mathbf{F}(\omega_{G,\A}(\tau))\ar@{->>}[l]_{p}\ar@{^{(}->}[d]\\
\mathbf{F}(\omega_{G,k})(\tau) & \mathbf{F}(\omega_{G,k}(\tau))\ar@{->>}[l]_{p}
}
\]
It is therefore sufficient to show that for every $f\in\mathbf{F}(V(\tau))$,
\[
\left\langle \overrightarrow{\alpha^{\star}(\tau)\,\varphi_{\mathcal{V}^{\star}}(\alpha^{\star})(\tau)},f\right\rangle \leq\left\langle \overrightarrow{\alpha^{\star}(\tau)\,\varphi_{\mathcal{V}^{\star}}(\alpha^{\star})(\tau)},p(f)\right\rangle 
\]
for the Busemann scalar product on $\mathbf{B}(V(\tau)\otimes K^{\star}).$
Note that since
\[
\varphi_{\mathcal{V}^{\star}}=\mathrm{Id}\otimes\varphi\quad\mbox{on}\quad\mathcal{V}^{\star}=V\otimes_{k}K^{\star},
\]
the standard $\mathcal{O}^{\star}$-lattice $V\otimes_{k}\mathcal{O}^{\star}$
is $\varphi_{\mathcal{V}^{\ast}}$-stable, and so is the corresponding
gauge norm $\alpha_{V\otimes\mathcal{O}^{\star}}=\boldsymbol{\alpha}(\circ)$.
The additivity of the Busemann scalar product gives 
\begin{eqnarray*}
\left\langle \overrightarrow{\alpha^{\star}(\tau)\,\varphi_{\mathcal{V}^{\star}}(\alpha^{\star})(\tau))},f\right\rangle  & = & \left\langle \overrightarrow{\alpha^{\star}(\tau)\,\boldsymbol{\alpha}(\circ)(\tau)},f\right\rangle +\left\langle \overrightarrow{\boldsymbol{\alpha}(\circ)(\tau)\,\varphi_{\mathcal{V}^{\star}}(\alpha^{\star})(\tau)},f\right\rangle \\
 & = & -\left\langle \overrightarrow{\boldsymbol{\alpha}(\circ)(\tau)\,\alpha^{\star}(\tau)},f\right\rangle +\left\langle \overrightarrow{\varphi_{\mathcal{V}^{\star}}(\boldsymbol{\alpha}(\circ))(\tau)\,\varphi_{\mathcal{V}^{\star}}(\alpha^{\star})(\tau)},f\right\rangle \\
 & = & (q-1)\cdot\left\langle \overrightarrow{\boldsymbol{\alpha}(\circ)(\tau)\,\alpha^{\star}(\tau)},f\right\rangle 
\end{eqnarray*}
and similarly for $p(f)$ \textendash{} using the formulas of section~\ref{subsec:Formula4Busemann}
and~\ref{subsec:CatNorm^phi_K}. For $\star=\prime$, we also know
that $\alpha^{\prime}\in\Norm_{K^{\prime}}^{\otimes}(\omega_{G,K'})$
belongs to $\mathbf{B}(\omega_{G},K')$ by lemma~\ref{lem:B(omega)isBig},
thus 
\[
\left\langle \overrightarrow{\boldsymbol{\alpha}(\circ)(\tau)\,\alpha^{\prime}(\tau)},f\right\rangle \leq\left\langle \overrightarrow{\boldsymbol{\alpha}(\circ)(\tau)\,\alpha^{\prime}(\tau)},p(f)\right\rangle 
\]
by proposition~\ref{prop:CompBusemProj}, which indeed applies since
$k=\mathbb{F}_{q}$ is perfect.

\subsubsection{~}

We have shown that any faithful exact $\otimes$-functor $x:\Rep(G)\rightarrow\Norm_{K}^{\varphi}$
is good. Starting with a pair of objects $(\mathcal{V}_{i},\varphi_{i},\alpha_{i})$
in $\Norm_{K}^{\varphi}$ (for $i\in\{1,2\}$), with Galois representations
$\rho_{i}:\Gal_{K}\rightarrow GL(V_{i})$, set $G:=GL(V_{1})\times GL(V_{2})$
and $\rho:=(\rho_{1},\rho_{2})$. Then $\rho:\Gal_{K}\rightarrow G(k)$
induces an exact and faithful $\otimes$-functor 
\[
\Rep(G)\rightarrow\Rep_{k}(\Gal_{K})
\]
 with corresponding étale $\varphi$-module $(\mathcal{V},\varphi_{\mathcal{V}}):\Rep(G)\rightarrow\Vect_{K}^{\varphi}$
given by 
\[
\mathcal{V}(\tau)=(\omega_{G,k}(\tau)\otimes K^{s})^{\Gal_{K}}\quad\mbox{and}\quad\varphi_{\mathcal{V}(\tau)}=\mathrm{Id}\otimes\varphi\vert_{\mathcal{V}(\tau)}.
\]
In particular, $(\mathcal{V},\varphi_{\mathcal{V}})(\tau_{i}^{\prime})=(\mathcal{V}_{i},\varphi_{i})$
where $\tau_{1}^{\prime}:=\tau_{1}\boxtimes1$ and $\tau_{2}^{\prime}:=1\boxtimes\tau_{2}$
for the tautological representation $\tau_{i}$ of $GL(V_{i})$ on
$V_{i}$. We have to show that the splittable $K$-norms $\alpha_{1}$
and $\alpha_{2}$ also extend to $\alpha\in\Norm_{K}^{\otimes}(\mathcal{V})$.
Since $\mathcal{V}^{s}=\mathcal{V}\otimes_{K}K^{s}\simeq\omega_{G,K^{s}}$,
the base changed norms $\alpha_{i}^{s}$ on $\mathcal{V}_{i}^{s}=\mathcal{V}_{i}\otimes_{K}K^{s}$
plainly extend to a unique $K^{s}$-norm 
\[
\alpha^{s}=(\alpha_{1}^{s},\alpha_{2}^{s})\quad\mbox{in}\quad\begin{array}{rcccc}
\mathbf{B}(\mathcal{V}_{1}^{s})\times\mathbf{B}(\mathcal{V}_{2}^{s}) & \simeq & \mathbf{B}^{e}(G,K^{s}) & \simeq & \mathbf{B}(\omega_{G},K^{s})\\
 & \subset & \Norm_{K^{s}}^{\otimes}(\omega_{G,K^{s}}) & \simeq & \Norm_{K^{s}}^{\otimes}(\mathcal{V}^{s})
\end{array}
\]
on $\mathcal{V}^{s}:\Rep(G)\rightarrow\Vect_{K^{s}}$. For every $\tau\in\Rep(G)$,
we may then define 
\[
\alpha(\tau):\mathcal{V}(\tau)\rightarrow\mathbb{R}_{+},\quad\alpha(\tau)\eqd\alpha^{s}(\tau)\vert_{\mathcal{V}(\tau)}.
\]
Plainly, $\alpha(\tau)$ is a $K$-norm on $\mathcal{V}(\tau)$ and
$\alpha(\tau_{i}^{\prime})=\alpha_{i}^{s}\vert_{\mathcal{V}_{i}}=\alpha_{i}$
on $\mathcal{V}(\tau_{i}^{\prime})=\mathcal{V}_{i}$, which is a splittable
$K$-norm on $\mathcal{V}_{i}$. Since $\tau_{1}^{\prime}$ and $\tau_{2}^{\prime}$
are $\otimes$-generators of the tannakian category $\Rep(G)$, it
follows that $\alpha(\tau)$ is a splittable $K$-norm for every $\tau\in\Rep(G)$.
Then $\alpha:\Rep(G)\rightarrow\Norm_{K}$ indeed belongs to $\Norm_{K}^{\otimes}(\mathcal{V})$,
thus 
\[
(\mathcal{V},\varphi_{\mathcal{V}},\alpha):\Rep(G)\rightarrow\Norm_{K}^{\varphi}
\]
is a faithful exact $\otimes$-functor with $(\mathcal{V},\varphi_{\mathcal{V}},\alpha)(\tau_{i}^{\prime})=(\mathcal{V}_{i},\varphi_{i},\alpha_{i})$
for $i\in\{1,2\}$. Since it is good, the pair $(\Norm_{K}^{\varphi},\deg)$
is indeed itself good.

\subsubsection{A variant}

We may also consider the quasi-abelian $k$-linear $\otimes$-category
$\Bun_{\mathcal{O}}^{\varphi}$ of pairs $(L,\varphi_{\mathcal{V}})$
where $L$ is a finite free $\mathcal{O}$-module and $\varphi_{\mathcal{V}}:\varphi^{\ast}\mathcal{V}\rightarrow\mathcal{V}$
is a Frobenius on $\mathcal{V}:=L\otimes K$, with the obvious morphisms
and tensor products. The functor 
\[
\Bun_{\mathcal{O}}^{\varphi}\rightarrow\Norm_{K}^{\varphi},\qquad(L,\varphi_{\mathcal{V}})\mapsto(\mathcal{V},\varphi_{\mathcal{V}},\alpha_{L})
\]
is a fully faithful exact $k$-linear $\otimes$-functor, whose essential
image is stable under strict subobjects and quotients. It is thus
also compatible with the corresponding HN-formalism. In particular,
the HN-filtration is a $\otimes$-functor 
\[
\mathcal{F}_{HN}:\Bun_{\mathcal{O}}^{\varphi}\rightarrow\F(\Rep_{k}(\Gal_{K})).
\]

\bibliographystyle{plain}
\bibliography{MyBib}

\end{document}